\documentclass[11pt]{amsart}

\usepackage[english]{babel}
\usepackage[utf8]{inputenc}
\usepackage{times}
\usepackage{bm}
\usepackage{bbm}
\usepackage{amssymb}
\usepackage{accents}
\usepackage{braket}
\usepackage{mathtools}
\usepackage{mathrsfs}
\usepackage{hyperref}
\usepackage{graphicx}
\usepackage{caption}
\usepackage{subcaption}
\usepackage{microtype}
\usepackage{float}
\usepackage{wrapfig, framed, caption}
\usepackage{mdframed}
\usepackage{enumitem}
\usepackage{tikz-cd}

\usepackage{makecell}

\usepackage{tikz}
\usetikzlibrary{calc, automata, chains, arrows.meta, quotes, backgrounds}
\restylefloat{figure}

\tikzset{
	LabelStyle/.style = { rectangle, rounded corners, draw,
		minimum width = 2em, fill = yellow!50,
		text = red, font = \bfseries }}

\calclayout

\setlist[enumerate]{label=\textbf{(\roman*)}}

\newtheorem{thm}{Theorem}[section]
\newtheorem{defn}{Definition}[section]
\newtheorem{rem}{Remark}[section]
\newtheorem{prop}{Proposition}[section]
\newtheorem{lem}{Lemma}[section]
\newtheorem{cor}{Corollary}[section]

\theoremstyle{definition}

\renewcommand{\set}[2]{\{\ #1 \ | \ #2 \ \}}
\renewcommand{\Set}[2]{\left\{\ #1 \ | \ #2 \ \right\}}
\renewcommand{\setminus}{\smallsetminus}
\newcommand{\R}{\mathbb{R}}

\newcommand{\N}{\mathbb{N}}
\newcommand{\Z}{\mathbb{Z}}

\newcommand{\e}{\mathrm{e}}
\renewcommand{\d}{\mathrm{d}}

\newcommand{\stack}[2]{\genfrac{}{}{0pt}{}{#1}{#2}}

\newcommand{\ra}[1]{\accentset{\rightharpoonup}{#1}}
\newcommand{\capsra}[1]{\overset{\rightharpoonup}{\vphantom{a}\smash{#1}}}

\newcommand{\capsla}[1]{\overset{\leftharpoonup}{\vphantom{a}\smash{#1}}}
\newcommand{\leftrightharpoonup}{%
	\mathrel{\mathpalette\lrhup\relax}%
}
\newcommand{\lrhup}[2]{%
	\ooalign{$#1\leftharpoonup$\cr$#1\rightharpoonup$\cr}%
}

\newcommand{\capslra}[1]{\overset{\leftrightharpoonup}{\vphantom{a}\smash{#1}}}

\let\onto\twoheadrightarrow

\newcommand{\alphset}{\mathscr{A}}
\newcommand{\alphsetra}{\capsra{\mathscr{A}}}
\newcommand{\alphsetla}{\capsla{\mathscr{A}}}

\newcommand{\stateset}{\mathscr{I}}
\newcommand{\statesetra}{\capsra{\mathscr{I}}}

\newcommand{\edgeset}{\mathscr{E}}
\newcommand{\edgesetra}{\capsra{\mathscr{E}}}

\newcommand{\causset}{\mathscr{C}}
\newcommand{\rvalph}{\mathrm{A}}
\newcommand{\rvalphra}{\capsra{\mathrm{A}}}
\newcommand{\rvalphla}{\capsla{\mathrm{A}}}
\newcommand{\rvalphlra}{\capslra{\mathrm{A}}}
\newcommand{\rvstate}{\mathrm{I}}
\newcommand{\rvstatera}{\capsra{\mathrm{I}}}

\newcommand{\rvstatelra}{\capslra{\mathrm{I}}}
\newcommand{\rvedge}{\mathrm{E}}
\newcommand{\rvedgera}{\capsra{\mathrm{E}}}

\newcommand{\trans}[1]{q^{(#1)}}
\newcommand{\transtup}{q}
\newcommand{\one}{\mathit{1}}
\newcommand{\One}{\mathbbm{1}}

\newcommand{\dfas}{\mathbb{M}}
\newcommand{\reddfas}{\tilde{\mathbb{M}}}

\newcommand{\tildeQ}{\tilde{\mathcal{Q}}}
\newcommand{\Q}{\mathcal{Q}}
\newcommand{\M}{\mathcal{M}}
\newcommand{\tildeM}{\tilde{\mathcal{M}}}

\renewcommand{\a}{\mathsf{a}}

\newcommand{\h}{\mathrm{h}}

\newcommand{\supp}{\mathrm{supp}}

\renewcommand{\Pr}{\mathrm{Pr}}
\newcommand{\Cr}{\mathrm{Cr}}

\newcommand{\ubar}[1]{\underline{#1}}

\title[Finitary Process Evolution I]{Finitary Process Evolution I:\\ Information Geometry of Configuration Space and the Process-Replicator Dynamics}

\author{Leonardo Aguirre}

\begin{document}
	\maketitle
	
	\begin{abstract}
		This report presents some fundamental mathematical results towards elucidating the information-geometric underpinnings of evolutionary modelling schemes for quasi-stationary discrete stochastic processes. The model class under consideration is that of finite causal-state processes, known from the computational mechanics programme, along with their minimal unifilar hidden Markov generators. The respective configuration space is exhibited as a collection of combinatorially related Riemannian manifolds wherein the metric tensor field is an infinitesimal version of the relative entropy rate. Furthermore, a certain evolutionary inference iteration is defined which can be executed by generator-carrying agents and generalizes the Wright-Fisher model from population genetics. The induced dynamics on configuration space is studied from the large deviation point of view and it is shown that the associated asymptotic expectation dynamics follows the Riemannian gradient flow of a given fitness potential. In fact, this flow can formally be viewed as an information-geometric generalization of the replicator dynamics from population biology.
	\end{abstract}

	\tableofcontents

	\section{Introduction}
	
	The last few decades have witnessed a steadily growing interest in the modelling of complex dynamical systems perpetuated by applications in condensed matter physics, biology, computational science, network analysis, theories of cognition etc. Complementarily, recent advances in data mining 
and high-throughput screening techniques afforded the necessary means to harvest vast amounts of raw data. These developments sparked an abundance of approaches concerned with the practical side of modelling which over time resulted in a plethora of model-archetypes and heuristics. Meanwhile, a conceptual understanding of the phenomenon of complexity itself seems far more elusive. Part of the difficulty stems from the fact that ``complex system'' is frequently used as a blanket term for a variety of phenomena with vastly different  incarnations. The common denominator is usually taken to be an irreducible combination of effective stochasticity at the local level along with the emergence of some global spatiotemporal patterns - manifestly a vestige of many strong couplings among partially unobservable degrees of freedom (\cite{varn2016did}, \cite{shalizi2006methods}). A unified treatment of complex phenomena hinges on a sufficiently general approach to the notion of pattern as opposed to randomness. The evident conceptual obstacle is of course that, à priori, there is no universal way to make such a distinction - a problem which is very much exacerbated in the presence of large empirical datasets. The \emph{computational mechanics} research programme (\cite{crutchfield1989inferring}, \cite{shalizi2001computational}, \cite{lohr2009models}) provides a principled framework addressing this issue. It formalizes the concept of an observed system as a stochastic process, i.e. a random variable $\rvalphra=\rvalph_{1:\infty}$ on the sequence space $\alphset^\N$ for some finite alphabet $\alphset$ of ``measurement values''. This mimics the typical scenario when a physical system is coupled to a discrete measuring device which provides readings at evenly-spaced points in time. It is furthermore assumed that all observation statistics are invariant with respect to time shifts - at least on the time-scale of feasible observations - which translates to $\rvalphra$ being stationary. An idealized theoretical analogue of this scenario, coming from symbolic dynamics, consists in a measure-preserving dynamical system together with a partition of its state space into finitely many regions labelled by the letters of $\alphset$. Hereby, the system's trajectories are again mapped to realizations of a stationary stochastic process $\rvalphra$ and the measure-preserving dynamics is mapped to the left-shift of sequences.\footnote{Partitions yielding an almost-everywhere invertible mapping from points of the state space to sequences are called generating partitions. In dimension larger than 1, not much is known about them. Likewise, a major difficulty in modelling complex systems (indeed dynamical systems in general) pertains to the situation-specific problem of finding a suitable collection of observables whose time evolution reflects the system's ``characteristic properties''. The details of this tricky subject are certainly beyond the scope of the present treatise (cf. \cite{bollt2000validity}, \cite{bollt2001symbolic}). Instead, we will assume the positivist stance that at some stage of modelling/simulating dynamical systems, state space discretization is a practical necessity imposed by the use of some sort of digital measurement resp. data processing architecture and that the scientifically relevant objects are not so much the purported underlying systems but rather the observable stationary processes they produce.} Thus, stationary stochastic processes are henceforth taken as the fundamental objects of inquiry. This viewpoint holds the promise of making the above broad characterization of complexity more concrete while still remaining sufficiently general to capture a wide range of phenomena. The computational mechanics programme approaches the subject of complexity by recovering an intrinsic causal architecture behind stationary processes based on the concept of predictive sufficiency. Intuitively, this generalizes the notion of an $L$-th order Markov process to ``infinite order'' in the following fashion: One starts by using stationarity to canonically extend $\rvalphra$ to a random variable $\rvalph_{-L+1:\infty}$, for any $L\in\N$. The realizations of $\rvalph_{-L+1:0}$ are considered length-$L$ histories. If now $\rvalphra$ happens to be $L$-th order Markov, then any history $\a_{-L+1:0}$ induces a conditional probability measure on the set of ensuing future sequences and this ``prediction'' cannot be refined using longer pasts. However, there may be different length-$L$ histories which induce the same conditional probability measure on futures and the corresponding equivalence classes are exactly the $L$-th order Markov-states of the process. In the general stationary case, where $\rvalphra$ is not required to be Markov of any order, computational mechanics uses the Kolmogorov-extension theorem to complete the collection of semi-infinite processes $\rvalph_{-L+1:\infty}$, $L\in\N$, to a random variable $\rvalphlra$ with values in $\alphset^\Z$. A realization $\capsla{\a}\in\alphset_{-\infty:0}$ is considered an infinite history and two histories $\capsla{\a},\capsla{\mathsf{b}}$ are said to be \emph{predictively equivalent} if the conditional probability measure of $\rvalphra$ conditioned on $\rvalphla=\capsla{\a}$ is the same as the one conditioned on $\rvalphla=\capsla{\mathsf{b}}$. The set $\causset$ of predictive equivalence classes is called the process’s \emph{causal-state memory}. The canonical projection $\alphsetla\to\causset$ is a minimal sufficient statistic for predicting the future half of $\rvalphlra$ based on infinite histories and is essentially unique up to isomorphism in an appropriate measure-theoretically defined sense (see \cite{lohr2009models}). Furthermore, adding new observations to histories yields a mechanism for updating the memory state which turns the causal-state memory into a transition-emitting hidden Markov presentation of $\rvalphlra$ – the process’s \emph{causal-state presentation}. It represents the minimal  ``causal architecture'' capable of generating the bi-infinite process $\rvalphlra$.\footnote{Observe that the causal-state presentation is not predicated on any structural modelling choices. In fact, the original goal was to operationalize the notion of forecasting-complexity (\cite{grassberger1986toward}, \cite{crutchfield1989inferring}), denoting the amount of historical information in a stationary process that is required for optimal prediction. It can be quantified in form of the entropy of the causal-state memory random variable, called the process's statistical complexity.} Irrespective of infinite histories, the causal state presentation can be viewed as a transition-emitting hidden Markov model, or \emph{hidden Markov machine} (HMM, see Sec. \ref{subsec:HMMs}), of the semi-infinite process $\rvalphra$ with internal state set $\causset$. If the latter has finite cardinality, the minimality statement can be strengthened to the assertion that the causal-state presentation is the \emph{minimal asymptotically synchronizable HMM of $\rvalphra$} in the sense of Def. \ref{defn:synchronizablehmm}. We shall refer to processes with finite causal-state memory as \emph{finitary} and to the induced causal-state HMMs as their \emph{causal-state machines (CSMs)}.\footnote{They are historically referred to as $\epsilon$-machines in computational mechanics. The $\epsilon$ originally indicated the resolution of discretizing a continuous dynamical system but has no further bearing on the theory. The term CSM is chosen here due to being more descriptive. Note also that a given finitary process may allow generating HMMs smaller than its CSM but they would not be asymptotically synchronizable.} Importantly, the state-transition combinatorics of CSMs satisfies a structural constraint called \emph{unifilarity} (Def. \ref{defn:unifilarity}), which makes for a well-behaved analytical theory. In fact, unifilar HMMs have been investigated for a long time in symbolic dynamics and formal language theory under the alternative names right-resolving presentations and probabilistic deterministic finite automata (\cite{lind1995introduction} ,\cite{delaHiguera:2010:GIL:1830440}, \cite{dupont2005links}).

Through the concept of a process's causal-state presentation, questions about pattern and complexity can be restated as structural questions about the presentation's transition structure. The article \cite{crutchfield2014dreams} proceeds to delineate an intriguing long-term objective of developing this approach into a framework for automated theory building, superseding the common pattern recognition practice of making educated guesses about the expected types of patterns. There is however one caveat to the story: While in theory there always exists an essentially unique causal state-presentation for any given stationary process, the generic case involves a causal-state memory of infinite cardinality. By contrast, the only causal-state presentations that can conceivably be reconstructed from finite sample data are CSMs. There are a number of practical algorithms dealing with the task of CSM reconstruction (\cite{crutchfield1991reconstructing}, \cite{shalizi2004blind}, \cite{strelioff2014bayesian}) but the fundamental idea of finite reconstruction raises a few conceptual questions, for instance: In what way does a reconstructed CSM, respectively its output-process, approximate the original if their causal-state memories have different cardinalities?\footnote{An approach of approximating just the causal-state memory (without transition-structure) by smaller state sets can be found in \cite{still2010optimal} under the name ``causal compression''.}  Are there on-line-reconstruction schemes whose evolution on the \emph{space of CSMs} can be understood analytically? These questions initiated the author's ongoing research effort about finitary process evolution. The present report contains some preliminary results which might be of interest in their own right.

\paragraph{\textbf{Generalized Wright-Fisher evolution and the associated information geometry: }}

The present approach to CSM reconstruction differs from existing methods by setting the reconstruction process in the context of artificial evolution according to the following blueprint. The units of selection are reproducing agents carrying a unifilar HMM as well as a simple self-resampling mechanism for recording imperfect copies of it into their offspring. Selection proceeds by weeding out offspring upon comparison of their output processes with respect to some given fitness potential function $\Phi$. The reproduction-selection iteration runs in non-overlapping generations of length $\tau$ and thereby generates an HMM-valued process which is named \emph{generalized Wright-Fisher (gWF-) process} for reasons outlined below. Note that, conceptually, this setup requires the distinction of two time-scales: First, a fast time-scale on which an agent expresses its (stationary) output-process. Second, a slow time-scale on which the agent's configuration changes according to the gWF-process.\footnote{This idea has been applied previously in \cite{gupta2007symbolic} to model gradual development of mechanical anomalies in human-engineered dynamical systems with quasi-stationary behaviour.} Before going into more details, a few comments are in order about how this general setup differs from existent reconstruction paradigms. First of all, as the word ``reconstruction'' indicates, past approaches have been geared exclusively towards approximating the causal-state presentation of some given process. In contrast, the present approach has potentially a wider scope by allowing to select offspring upon arbitrary environmental reinforcement $\Phi$, e.g. entailed by game-theoretic interactions with the environment. Another main difference is that past approaches were designed to operate mainly in batch-mode while the present approach naturally works in on-line-mode providing full-fledged unifilar HMMs at any stage of evolution.\footnote{In fact, unifilar HMMs are generically CSMs (Cor. \ref{cor:degvariety}).} This evolutionary paradigm seems particularly well-suited to situations where some process appears stationary on a dynamically significant time-scale but may change gradually in the long run.

We shall now provide some details about the gWF-process. As the name indicates, it will be constructed as a generalization of the well-known (frequency-dependent) Wright-Fisher process from population genetics. It is instructive to swiftly review this classical model of evolution as well as its asymptotic behaviour: 
The Wright-Fisher model describes the evolution of a finite population of constant size, comprised of species $\a \in\alphset$, through self-resampling. The population's configuration can be viewed as a point $q=(q^\a)_{\a \in\alphset}$ of the $(|\alphset|-1)$-dimensional standard-simplex $\bar{\Sigma}^\alphset$ in $\R^\alphset$. Evolution proceeds by  (sequentially) sampling members of the population (without permanently removing them). Thus a sample of length $L\in\N$ becomes a string $\a_{1:L}\in\alphset_{1:L}$ and in fact this sample occurs with probability
\begin{equation}\label{eq:classicalprobL}
	\Pr_{\rvalph_{1:L}}(q)[\a_{1:L}]=\prod_{l=1}^L q^{\a_l}.
\end{equation}
It gives rise to the new population configuration $\left( \frac{1}{L}\mathrm{card}\set{1\leq l\leq L}{\a_l=\mathsf{b}}\right)_{\mathsf{b}\in\alphset}$ and such self-resamplings occur after constant generation intervals of length $\tau$. The \emph{frequency-dependent Wright-Fisher model} additionally modifies (\ref*{eq:classicalprobL}) by skewing the sampling probability of species $\a$ with its selection coefficient $F_\a(q,\tau)$ depending on the current configuration $q$ and the generation length $\tau$. The probability that a configuration $q$ evolves into $q'$ in one generation\footnote{Assuming both are allowed configurations i.e. all frequencies are multiples of $1/L$.} is then given by the Markov-kernel
\begin{equation}\label{eq:wrightfisher}
	R^{L}[q'|q] =\frac{1}{Z_L} \frac{L!}{\prod_{\a \in\alphset} (Lq'^{\a})!} \prod_{\a \in\alphset}(q^{\a} F_\a(q,\tau))^{L q'^{\a}},
\end{equation}
where $Z_L=\left(\sum_{\a} q^\a F_\a(q,\tau)\right)^L$ is for normalization. The resulting Markov-process on $\bar{\Sigma}^\alphset$ is called the \emph{(frequency-dependent) Wright-Fisher (WF-) process}.
Good asymptotic results for $L\to\infty$, $\tau\to 0$ can be obtained under the assumption of weak selection, meaning that there exist ``differential selection functions'' $f_\a:\bar{\Sigma}^\alphset\to\R$ such that $F_\a(q,\tau)= 1+ \tau f_\a(q)+o(\tau).$ Depending on the asymptotics of $L\tau$, the respective WF-processes converge weakly to a combination of a diffusion and a convection process. If $L\tau\to\infty$ the limit process is purely convective and its characteristic curves solve the well-known replicator equation:\footnote{See \cite{chalub2014frequency} for this result and \cite{chalub2009discrete}, \cite{schlag1998imitate} for other discrete models of evolution leading asymptotically to the replicator equation. It originally arose as an model of selection in an infinite population of replicators whose rate of reproduction is prescribed by the differential selection functions $f_\a$. The classical linear example $f_\a(q)=\sum_{\mathsf{b}\in\alphset} r_{\a \mathsf{b}} q^\mathsf{b}$ comes from evolutionary game theory, specifically sequential normal-form games, where $r_{\a \mathsf{b}}$ is the payoff for behaviour $\a$ in an interaction with behaviour $\mathsf{b}$ (\cite{page2002unifying}, \cite{hofbauer2003evolutionary}).}
\begin{equation}\label{eq:classicalreplicator}
\dot{q}^\a = q^\a (f_\a(q)-\bar{f}(q))\ ,\ \a \in\alphset.
\end{equation}
Note that (\ref*{eq:classicalreplicator}) can also be obtained from the WF-process in a simpler, albeit more heuristic, way bypassing the limit Markov-process: Consider the deterministic process, which maps $q$ to the expectation value $\langle q'\rangle_{R^{L}[q' | q]}$. It can easily be verified that the trajectories of this expectation-process converge in the limit $\tau\to 0$ to the integral curves of (\ref*{eq:classicalreplicator}).

We shall now concretize the gWF-process and extend the previous Wright-Fisher narrative to the setting of finitary process evolution. Note that, in the language of probability theory, the WF-kernel (\ref*{eq:wrightfisher}) describes evolution simply as iterated replacement of a current population-distribution $q$ by an empirical distribution $q'$ whose probability is the likelihood to be sampled from $q$ weighted by the selection function $F(q',q,\tau)=\prod_{\a}(F_\a(q,\tau))^{L q'^{\a}}$ quantifying the effective fitness of the variation $q'$ from $q$. More generally, let us from now on denote by $q$ a unifilar HMM  and by $\Pr_{\rvalphra}(q)\in\mathscr{P}[\alphsetra]$ its (stationary) output-process. The \emph{gWF-model} works by alternating between the following two steps:
\begin{itemize}
	\item \textbf{Selection:} A fitness potential function $\Phi$ shall be given on a sufficiently large subset of $\mathscr{P}[\alphsetra]$, i.e. containing all occurring output-processes. The current generation of agents is comprised of $N$ members carrying unifilar HMMs $q_1,\ldots, q_N$. The agent number $m$ whose HMM yields the highest value of $\Phi\circ\Pr_{\rvalphra}$ is selected for reproduction.\footnote{If several agents maximize $\Phi\circ\Pr_{\rvalphra}$ simultaneously, one of them can be chosen at random. The way in which this is done will have no bearing on the asymptotic results in this paper. More generally, a sequel to the present report will examine probabilistic selection mechanisms.}
	\item \textbf{Reproduction:}
	Agent number $m$ produces the next generation of $N$ agents with unifilar HMMs $q'_1,\ldots, q'_N$ by creating sequence realizations of length $L$ from $q_m$ and using the empirical transition counts between internal states to obtain the new transition probabilities for the $q'_n$.
\end{itemize}
More formal definitions will be given in Sec. \ref{sec:evolution} once the relevant notations have been introduced. Looking just at the evolution of the selected procreator HMMs over the course of generations yields the \emph{gWF-process}. Its transitions occur according to a Markov-kernel $R^{L}_{N,\Phi}$ on the space of unifilar HMMs having constant ``combinatorial type''. Much like in (\ref*{eq:wrightfisher}), the transition probability $R^{L}_{N,\Phi}[q'|q]$ is the product of the likelihood that $q'$ is empirically sampled as an offspring of $q$ times a selection term $F(q',q)$ being the likelihood that the other offspring of $q$ yield a lower value of $\Phi\circ\Pr_{\rvalphra}$. Sec. \ref{sec:evolution} investigates the deterministic expectation-process with transition $q\mapsto \langle q' \rangle_{R^{L}_{N,\Phi}[q'|q]}$ and its asymptotic behaviour as the time-scales of output-processes and gWF-process decouple in the sense $L\to\infty, \tau\propto \sqrt{1/L}$. The upshot is that, in this limit, the expectation-process's trajectories converge to the integral curves of a generalized form of the replicator equation, which will be called the process-replicator equation:
\begin{equation}\label{eq:replicatorconservative}
\dot{q}= \nabla_\mathrm{g} (\Phi\circ\Pr_{\rvalphra}).
\end{equation}
Herein $\nabla_\mathrm{g}$ is the Riemannian gradient with respect to a certain metric tensor field $\mathrm{g}$ which is obtained as the leading asymptotics of the relative entropy rate on the space of CSMs. Note that, by this token, the gWR-iteration induces a specific information geometry on the space of CSMs. The classical population-replicator result is recovered in the special case of a single causal memory state (see Fig. \ref{fig:1statedfa}): Here, the output-processes are simply i.i.d. sequences of $\alphset$-valued random variables and the relative entropy rate is $\mathrm{h}(\Pr_{\rvalphra}(q')\parallel \Pr_{\rvalphra}(q))=\sum_{\a} q'^\a \log\frac{q'^\a}{q^\a}$. Hence we obtain $\mathrm{g}=\sum_{\a} \frac{1}{q^\a} (\d q^{\a})^{\otimes 2}$ which turns (\ref*{eq:replicatorconservative}) into (\ref*{eq:classicalreplicator}) for $f_\a = \frac{\partial}{\partial q^\a} (\Phi\circ\Pr_{\rvalphra})$.
\begin{figure}
	\begin{tikzpicture}[-Triangle, every loop/.append style = {-Triangle}]
	
	\path[use as bounding box] (-2,-2) rectangle (2,2);
	
	\node[state, inner sep=0pt, minimum size=7pt,fill=gray!30]  (1) {};
	
	\draw    (1)  edge[loop left, out=45, in=315, looseness=15, inner sep=1 pt, "$\cdots$"] (1);
	
	\draw    (1)  edge[loop left, out=70, in=290, looseness=50, inner sep=1 pt, "$\mathsf{n}$"] (1);
	
	\draw    (1)  edge[loop left, out=135, in=225, looseness=15, inner sep=1 pt, "$\cdots$", swap] (1);
	
	\draw    (1)  edge[loop left, out=110, in=250, looseness=50, inner sep=1 pt, "$\mathsf{0}$", swap] (1);

	\end{tikzpicture}
	\caption{Transition graph of a CSM with one causal state and alphabet $\alphset=\{\mathsf{0},\ldots, \mathsf{n}\}$. The transition probabilities are omitted.}
	\label{fig:1statedfa}
\end{figure}
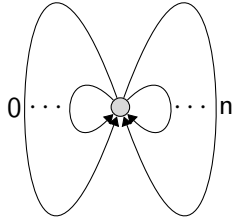

\paragraph{\textbf{Organization of the report:}}

Sec. \ref{sec:HMM} fixes notations and gives some background on presentations of stationary processes. In particular, finitary processes will be formally introduced along with their CSMs and the latter will be characterized in terms of their transition structure. Everything in this section is common knowledge in the field of hidden Markov processes and computational mechanics. The proofs are referred to the literature. In Sec. \ref{sec:relentrate}, the formal extension of the replicator equation is undertaken. To that end, Sec. \ref*{subsec:geomsetup} establishes the space of all CSMs on a given internal state set as a collection of combinatorially related \emph{configuration manifolds} which together constitute an open subset of some polytopal cell complex. Sec. \ref*{subsec:relentrate} derives an exact expression for the relative entropy rate $\mathrm{h}$ holding between processes whose CSMs lie in the same configuration manifold. Sec. \ref*{subsec:Riemanniangeom} proceeds by deriving the local asymptotics of $\mathrm{h}$ which then serves to equip the configuration manifolds with a Riemannian metric tensor $\mathrm{g}$. In Sec. \ref*{subsec:processreplicator} the process-replicator equation with given fitness potential $\Phi:\mathscr{P}[\alphsetra]\to\R$ is generalized from the simplex to the configuration manifolds as the gradient flow equation of $\Phi\circ \Pr_{\rvalphra}$. Furthermore, an extension of the folk theorem from evolutionary game theory is proved for the solutions of the process-replicator equation. Sec. \ref{sec:evolution} reveals the significance of the process-replicator equation by relating it to the gWF-process. To that end, Sec. \ref*{subsec:empiricalmachines} expresses the self-resampling probabilities from the reproduction-step of the gWF-model as a discrete statistical ensemble on configuration space - the \emph{empirical configuration ensemble} - and shows that it satisfies a large deviation principle. In Sec. \ref*{subsec:fluctuationmetric}, a corresponding central limit theorem is obtained in which $\mathrm{g}(q)$ features as the covariance matrix of a Gaussian measure on the tangent space at $q$ of the respective configuration manifold. Sec. \ref*{subsec:greedyasymptotics} concludes the narrative by showing that the trajectories of the gWF-iteration's expectation-process asymptotically, for large self-resampling lengths $L$, converge to the integral curves of the process-replicator equation. Sec. \ref{sec:examples} illustrates the process-replicator dynamics in a few low-dimensional examples. To not interrupt the flow of argument, the mathematically more involved proofs are collected in the appendix and the respective theorems in the main text are followed by short proof outlines.

\paragraph{\textbf{Relevance and further directions:}}

The main goal of the present report is to expound some technical results which seem to be useful in understanding evolutionary CSM inference schemes, contributing to the computational mechanics programme and more broadly to the theory of hidden Markov processes. The notable mathematical results are: (i) The exact relative-entropy rate formula of Thm. \ref{thm:relentrate}, which is a quite extraordinary fact in the context of general hidden Markov processes. (ii) A natural Riemannian geometry on the configuration manifolds of finitary processes along with an associated process-replicator dynamics, which generalizes previous approaches from theoretical biology and information geometry. (iii) A large deviation principle (Thm. \ref{thm:largeconsistentdev}) and a corresponding central limit theorem (Thm. \ref{thm:centrallimitthm}) pertaining to the empirical configuration ensemble, which generalize well-known results familiar from (generic) ergodic Markov processes. Beyond those technical results, another contribution consists in founding a new CSM inference paradigm which should ultimately be suitable for on-line-learning in software-agent implementations and, at the same time, induce a Markovian configuration space dynamics with a tractable continuum limit. The construction and analysis of this scheme is however not completed in the present report. Two main aspects are missing and will be investigated in follow-up papers: First, the construction of the full continuum limit process in the context of stochastic Itô calculus: Results to that effect should be of independent relevance as they lead to a automata-theoretic generalization of the Wright-Fisher diffusion -  a well-studied limit dynamics of genetic drift in mathematical biology (\cite{chalub2014frequency},\cite{chalub2009discrete}). Second, the issue of inferring the combinatorial type of CSMs: Incorporating methods to that effect into the present parameter inference scheme and investigating the ensuing dynamics in terms of the geometry of configuration space promises new insights towards a notion of ``lossy causal compression'' of stationary processes (see also \cite{still2010optimal}).

On the other hand, the proposed inference scheme could be evaluated on empirical datasets, firstly in cases where CSM reconstruction by other methods has been applied previously, such as: x-ray diffraction patterns in complex materials (\cite{riechers2015pairwise},\cite{varn2016did}), conformational dynamics of single molecules (\cite{nerukh2012non}, \cite{li2008multiscale}), meteorological and geomagnetic datasets (\cite{palmer2000complexity}, \cite{clarke2003application}), automated visualization of fluid dynamics (\cite{janicke2009steady}), on-line anomaly detection in mechanical systems (\cite{gupta2007symbolic}), analysis of neural spike trains and discovery of functional connectivity in neural ensembles (\cite{haslinger2010computational}, \cite{shalizi2007discovering}), natural language processing (\cite{padro2005named}), user behaviour and dynamics of content creation in social media (\cite{darmon2013predictability}, \cite{cointet2007intertemporal}). Secondly, expanding the scope of present CSM reconstruction methods through the use of general fitness potentials $\Phi$, the outlined inference paradigm seems to be compatible with basic concepts in stochastic game theory and reinforcement learning where agents and their environments are frequently modelled as partially observable Markov decision processes. For that purpose, the inference paradigm would need to be further elaborated to include conditional input-output processes (see \cite{barnett2015computational}).

Lastly, the proposed inference paradigm might have a stimulating influence on theoretical biology by aligning with a larger conceptual debate that has been simmering in the back of evolutionary biologists' minds for some time: Is there a rigorous way to phrase evolutionary adaptation of ecological processes as gradual inference of environmental conditions, and are there information-theoretic universalities about the way this happens (\cite{smith2000concept},\cite{wallace1998information},\cite{rivoire2011value},\cite{van2013biological})? The present report tries to provide a new input to this debate by offering a mathematically tractable paradigm of finitary process evolution.

\paragraph{\textbf{General notation and terminology:}}
For any finite set $\mathscr{S}$, the vector space $\R^{\mathscr{S}}$ is equipped with the standard basis $(\ket{s})_{s\in\mathscr{S}}$, and $(\bra{s})_{s\in\mathscr{S}}$ denotes the linearly dual basis in $(\R^{\mathscr{S}})^*$. The interior of the standard simplex in $\R^{\mathscr{S}}$ is denoted by $\Sigma^\mathscr{S}$. The evaluation of a linear form $\eta$ on a vector $v$ is written as $\eta\cdot v$ and this notation also applies to the pairing of mutually dual tensor fields on a differentiable manifold such as the evaluation of a Riemannian metric on a pair of vector fields written as $\mathrm{g}\cdot (\xi_1\otimes\xi_2)$. Measurable spaces are denoted by capital script letters such as $\mathscr{X}$ and their elements by lower case letters $x\in\mathscr{X}$. Measurable spaces with an evident topological structure are by default equipped with the Borel-$\sigma$-algebra and products are equipped with the respective product $\sigma$-algebras. The set $\mathscr{P}[\mathscr{X}]$ of probability measures on $\mathscr{X}$ is also considered a measurable space through its canonical $\sigma$-algebra, generated by pulling back the Borel-$\sigma$-algebra from $[0,1]$ along evaluations in measurable subsets of $\mathscr{X}$. A chosen probability measure $P\in \mathscr{P}[\mathscr{X}]$ makes $\mathscr{X}$ into a probability space and any measurable map $\mathrm{Y}:\mathscr{X}\to\mathscr{Y}$ into a $\mathscr{Y}$-valued random variable. The law of $\mathrm{Y}$ is the push-forward probability measure $P_\mathrm{Y}:=\mathrm{Y}_*P=P\circ\mathrm{Y}^{-1}$ and the probability of a measurable event $S\subset\mathscr{Y}$ will sometimes be written $P[\mathrm{Y}\in S]:=P_\mathrm{Y}[S]$. A Markov-kernel is a measurable map $R: \mathscr{Y}\to\mathscr{P}[\mathscr{Z}],\ y\mapsto R[\cdot | y ]$. For two random variables $\mathrm{Y}, \mathrm{Z}$, which are defined on the same probability space, the conditional probability kernel $P_{\mathrm{Z} | \mathrm{Y}}$ is the Markov-kernel
$$z\mapsto P_{\mathrm{Z} | \mathrm{Y}}[\cdot | y]:=P[\ \mathrm{Z}^{-1}(\cdot)\ |\ \mathrm{Y}^{-1}(y)\ ]$$ and it is tacitly assumed that a regular version is chosen (which is always possible for the probability spaces involved).

	\section{Computational Mechanics}\label{sec:HMM}
	
	\subsection{Stationary Processes and Hidden Markov Machines}\label{subsec:HMMs}
	
	For any measurable space $\mathscr{X}$, we write $\mathscr{X}_{m:n}:=\mathscr{X}^{\{m,\ldots,n\}}$ for the set of sequences $x_{m:n}=x_m\ldots x_n$ and use the shorthands $\mathscr{X}_m:=\mathscr{X}_{m:m}$ and $\ra{\mathscr{X}}\equiv\mathscr{X}_{1:\infty}:=\mathscr{X}^\N$. The map $\mathrm{X}_{m:n}$ denotes the canonical projection to $\mathscr{X}_{m:n}$ whenever this makes sense. An $\mathscr{X}$-valued \emph{(stochastic) process} is simply an $\capsra{\mathscr{X}}$-valued random variable. It is called \emph{stationary} if its law is invariant under the \emph{left-shift map} $x_{1:\infty} \mapsto x_{2:\infty}$. Throughout this report, $\alphset$ denotes a finite alphabet.
	\begin{defn}
		Let $P\in\mathscr{P}[\alphsetra]$ and suppose there exists a measurable space $\mathscr{S}$, a probability measure $Pr\in\mathscr{P}[\capsra{\mathscr{S}}\times\alphsetra]$ with
		\begin{itemize}\label{defn:markovcondition}
			\item[(i)] $P=Pr_{\rvalphra}$
			\item[(ii)] $Pr_{(\rvalph_{l},\mathrm{S}_{l+1}) | (\rvalph_{1:l-1}, \mathrm{S}_{1:l})}= Pr_{(\rvalph_{l},\mathrm{S}_{l+1}) | \mathrm{S}_{l}}$ and this Markov-kernel is shift-invariant.
		\end{itemize}
		The process $\capsra{\mathrm{S}}\times\rvalphra$ with law $Pr$ is called a (homogeneous) hidden Markov presentation of its \emph{output-process} $\rvalphra$ (having law $P$).
	\end{defn}
	
	Part (ii) of the above definition implies that the presentation $\capsra{\mathrm{S}}\times\rvalphra$ is a Markov chain and the same holds for the internal state process $\capsra{\mathrm{S}}$. Note however that the output-process $\rvalphra$ doesn't need to be Markovian and in fact every stationary process can be presented like this (e.g. by the construction in Sec. \ref{subsec:CSM}).
	
	We are now going to focus on presentations with a finite internal state set. For that purpose, $\stateset$ will always denote a finite set unless stated otherwise. Probability measures on $\stateset$ can be conveniently encoded as elements of the standard simplex in the \emph{internal state space} $(\R^{\stateset})^*$. The basis vectors $\bra{j}$ are called \emph{pure states} and a  \emph{mixed state} $\mu$ is a convex linear combination of pure states, i.e. $\mu\in (\R^{\stateset}_{\geq 0})^*$ with $\mu\cdot\one=1$, where $\one:=\sum_{j\in\stateset} \ket{j}$.
	The Markov-kernel in Def. \ref{defn:markovcondition}.(ii) encodes the transition probabilities
	$$q^{j,\a}_k=Pr_{(\rvalph_1, \rvstate_2) | \rvstate_1}[(\a, k) | j ]$$
	and we collect them into the non-negative linear operators
	$$\trans{\a}:= \sum_{j,k\in\stateset} q^{j,\a}_k\ \ket{j}\bra{k}\ \ , \a \in\alphset,$$
	called \emph{output-operators}.	The only restriction on them is that the \emph{total transition operator} $$\ubar{q}:=\sum_{\a\in\alphset} \trans{\a}$$ ought to preserve the vector $\one$ and, hence, the set of valid tuples of output-operators can be parametrized by the convex polytope
	$$\tilde\Q_{\stateset}:=\Set{(q^{j,\a}_k)\in \R^{\stateset\times\alphset\times\stateset}_{\geq 0}}{\ubar{q}\cdot\one = \one}.$$
	Observe that $\tilde\Q_{\stateset}$ has the combinatorial type of the $|\stateset|$-th power of the $(|\stateset||\alphset|-1)$-dimensional simplex.

	\begin{defn}\label{defn:HMM}
		\begin{itemize}
			\item[(i)] An element $q\in\tilde\Q_{\stateset}$ is called a \emph{Hidden Markov Machine (HMM)}\footnote{This term is for instance used in \cite{travers2011exact}, \cite{travers2011asymptotic}. History knows several other names for the same concept such as transition-emitting hidden Markov model or probabilistic finite-state automaton (\cite{vidal2005probabilistic}).} with \emph{output-alphabet} $\alphset$ and \emph{internal state set} $\stateset$. A pair $(q,\nu)$ with $\nu\in (\R^{\stateset}_{\geq 0})^*$ satisfying $\nu\cdot\one=1$ is called an \emph{initialized HMM} and $\nu$ is called the initial mixed state. A mixed state is called \emph{stationary} if it is invariant under $\ubar{q}$.
			\item[(ii)] The \emph{transition graph} $\Gamma(q)$ of some $q\in\tildeQ_\stateset$ is the directed edge-labelled multi-graph on the vertex set $\stateset$ which contains a labelled edge $j\xrightarrow{\a} k$ whenever $\bra{j}\trans{\a}\ket{k}>0$.
			\item[(iii)] An HMM $q$ is called \emph{(semi-)simple} if the linear operator $\ubar{q}$ is (semi-)simple. The set $\Q_{\stateset}\subset\tilde{\Q}_{\stateset}$ denotes the subset of  simple HMMs.
		\end{itemize}
	\end{defn}
	Clearly, an HMM $q$ is semi-simple iff every weakly-connected component of $\Gamma(q)$ is strongly-connected and it is simple iff $\Gamma(q)$ itself is strongly-connected. In order to investigate stationary output-processes, it is enough to consider semi-simple HMMs since otherwise all involved processes are left unchanged by passing to the internal state space $(\R^{\stateset'})^*$ where $\stateset'\subset\stateset$ is comprised of the vertices of the maximal strongly-connected subgraphs of $\Gamma(q)$.
	\begin{lem}\label{lem:stationarystates}
		For any semi-simple HMM $q\in\tildeQ_\stateset$ there exists a stationary mixed state $\nu$ with $\nu\cdot\ket{j}>0$ for any $j\in\stateset$. If $q\in\Q_\stateset$ then the stationary mixed state is unique and given by
		$$\pi(q) = \frac{1}{\sum_{k\in\stateset}\bm{[}\One -\ubar{q}\bm{]}_{kk}}\sum_{j\in\stateset} \bm{[}\One- \ubar{q}\bm{]}_{jj}\ \bra{j},$$
		where $\bm{[} x \bm{]}_{ii}$ denotes the $i$-th principal minor of the matrix $(\bra{j}x\ket{k})_{j,k\in\stateset}$.
	\end{lem}
	\begin{proof}
		See \cite{seneta2006non}.
	\end{proof}
	
	We shall now swiftly review how an initialized HMM $(q,\nu)$ generates a hidden Markov presentation $\rvstatera\times\rvalphra$. The idea is of course that any finite-length realization $(i_{1:L},\a_{1:L})$ is generated by starting  with probability $\mu\ket{i_1}$ in state $i_1$, then making a transition $i_1\xrightarrow{\a_1} i_2$ with probability $\bra{i_1}\trans{\a_1}\ket{i_2}$ followed by a transition $i_2\xrightarrow{\a_2} i_3$ with probability $\bra{i_2}\trans{\a_2}\ket{i_3}$ and so forth. We use this recipe to build
	a probability distribution $\Pr_{\rvedge_{1:L}}(q,\nu)$ on $\stateset_{1:L}\times\alphset_{1:L}$, namely
	\begin{equation} \label{eq:edgeseqprob}
	\Pr_{\rvedge_{1:L}}(q,\nu)[(i_{1:L},\a_{1:L})]:=\nu\cdot \ket{i_1}\ \cdot \left(\prod_{l=1}^{L-1} \bra{i_l} \trans{\a_l}\ket{i_{l+1}}\right)\cdot \bra{i_L}\trans{\a_L}\cdot\one.
	\end{equation}
	Invoking the Kolmogorov extension theorem, we obtain a unique probability measure $\Pr(q,\nu)$ on $\statesetra\times\alphsetra$ having the distributions (\ref*{eq:edgeseqprob}) as prefix-marginals.
	Summation over all possible output-sequences results in the internal state $(1:L)$-block distribution:
	\begin{equation} \label{eq:stateseqprob}
	\Pr_{\rvstate_{1:L}}(q,\nu)[i_{1:L}]=\nu\cdot \ket{i_1}\ \cdot \prod_{l=1}^{L-1} \bra{i_l} \ubar q \ket{i_{l+1}}.
	\end{equation}
	On the other hand, summing over all possible internal state sequences, we obtain the familiar formula for output $(1:L)$-blocks:
	\begin{equation} \label{eq:alphseqprob}
	\Pr_{\rvalph_{1:L}}(q,\nu)[\a_{1:L}]=\nu\cdot \trans{\a_{1:L}}\cdot\one,
	\end{equation}
	where we used the shorthand $$\trans{\a_{1:L}}:=\trans{\a_1}\cdot\trans{\a_2}\cdots\trans{\a_L}.$$
	\begin{defn}
		Let $(q,\nu)$ be an initialized HMM on internal state set $\stateset$. We set $\edgeset:=\stateset\times\alphset$.
		\begin{itemize}
			\item[(i)] The identity random variable $\ra\rvedge$ on $\ra\edgeset$ with law $\Pr_{\rvedgera}(q,\nu)=\Pr(q,\nu)$ is called the \emph{edge-process} generated by $(q,\nu)$.
			\item[(ii)] The canonical projection random variable $\ra\rvstate: \edgesetra\to\statesetra$ with law $\Pr_{\rvstatera}(q,\nu)$ is called the \emph{internal state process} of $(q,\nu)$.
			\item[(ii)] The canonical projection random variable $\ra\rvalph: \edgesetra\to\alphsetra$ with law $\Pr_{\rvalphra}(q,\nu)$ is called the \emph{output-process} of $(q,\nu)$.
		\end{itemize}
	\end{defn}
	Of course $\rvstatera\times\rvalphra$ is by construction a hidden Markov presentation of the output-process $\rvalphra$.

	 \subsection{Causal-State Machines}\label{subsec:CSM}
	 
	 One of the fundamental building blocks of the computational mechanics programme is that, for any given \emph{stationary} process $\rvalphra$, there always exists a certain hidden Markov presentation which is constructed in the following fashion: 
	 We canonically extend the law of $\rvalphra$ to a stationary probability measure $P$ on $\capslra{\alphset}:=\alphset^{\mathbb{Z}}$ and set  $\capsla{\alphset}:=\alphset_{-\infty:0}$. We define an equivalence relation $\sim$ on $\capsla{\alphset}$ by
	 $$\capsla{\a} \sim \capsla{\mathsf{b}} \quad :\Leftrightarrow\quad P_{\capsra{\rvalph} | \capsla{\rvalph}}[ \cdot | \capsla{\a} ] = P_{\capsra{\rvalph} | \capsla{\rvalph}}[ \cdot | \capsla{\mathsf{b}} ],$$
	 i.e. past sequences are considered to be equivalent if they lead to the same conditional law on future sequences. The set of equivalence classes $\causset:=\capsla{\alphset}/\sim$ is called the \emph{causal state memory} of $\rvalphra$. Repeating this procedure for any other separation point of past and future, the canonical projections to equivalence classes can be combined into a measurable projection $\capslra{\alphset} \onto \capsra{\causset}\times\alphsetra$\footnote{It is measurable because the canonical projection $\capsla{\alphset}\to\causset$ is measurable according to \cite{lohr2009models}, Lem.3.18.}
	 Pushing forward the probability measure $P$
	 yields a hidden Markov presentation $\capsra{\mathrm{C}}\times\rvalphra$ which we call the \emph{causal-state presentation} of $\rvalphra$. It satisfies certain measure-theoretically defined minimality and uniqueness properties as elaborated in \cite{lohr2009models}. In focusing on the the finite-state case, we will be able to avoid most of the measure-theoretic complications.
	 \begin{defn}\label{def:finitary}
	 	An $\alphset$-valued stationary process $\rvalphra$ is called \emph{finitary} if its causal-state memory has finite cardinality.\footnote{Strictly speaking, the causal-state memory depends on the chosen version of conditional probability (as remarked in \cite{lohr2009models}). Thus, the above definition means that there is a version of conditional probability yielding a finite causal-state memory and we assume wlog. that all causal states obtain non-zero measure under the canonical projection. The term ``finitary'' has previously been used for instance in \cite{johnson2010enumerating}, \cite{gornerup2008hierarchical} to indicate finite causal-state memory. Its different use in the influential paper \cite{heller1965stochastic} denotes what is nowadays commonly called a finite-dimensional process.}
	 \end{defn}
	 
	 To characterize the causal-state presentation of some finitary process as a special generating HMM, we define the following important property.
	 \begin{defn}\label{defn:synchronizablehmm}
	 	A semi-simple HMM $q\in\tildeQ_\stateset$ is called \emph{asymptotically synchronizable} if, for any stationary mixed state $\nu$, we have almost surely
	 	\begin{equation*}\label{eq:synchronizing}
	 	\lim_{L\to\infty} \mathrm{H}_{\stateset}\left(\nu\cdot \trans{\rvalph_{1:L}}\right) = 0,
	 	\end{equation*}
	 	where
	 	$$\mathrm{H}_{\stateset}(\mu):= \sum_{j\in\stateset} \frac{\mu\cdot\ket{j}}{\mu\cdot\one} \log_{|\stateset|}\left(\frac{\mu\cdot\ket{j}}{\mu\cdot\one}\right)$$
	 	is the \emph{mixed state entropy} of $\mu$ with respect to the basis $(\ket{i})$.
	 \end{defn}
	 Informally speaking, asymptotic synchronizability means that, for almost any output-sequence $\ra{\a}$ and any $\ra{e}\in\statesetra\times\alphsetra$ with $\rvalphra(\ra{e})=\ra{\a}$, the internal states $\rvstate_L(\ra{e}), L\in\N,$ are ``asymptotically determined'' by $\a_{1:L}$. This property in particular requires $q$ to be \emph{unifilar} in the following sense:
	 \begin{defn}\label{defn:unifilarity}
	 	An HMM $q\in\tilde\Q_{\stateset}$ is called \emph{unifilar} if the graph $\Gamma(q)$ contains at most one outgoing $\a$-labelled edge at $j$, for any $(j,\a) \in\edgeset$. The set of unifilar HMMs is denoted by $\tildeM_\stateset\subset\tildeQ_\stateset$ and the set of simple unifilar HMMs by $\M_\stateset\subset\Q_\stateset$.
	\end{defn}
	Conversely, unifilarity does not guarantee asymptotic synchronizability. In addition, some minimality property needs to be satisfied:
	\begin{prop}\label{prop:minimality}
		Let $q\in\tildeQ_\stateset$ be semi-simple.
		\begin{itemize}
			\item[(i)] The following are equivalent:
			\begin{itemize}
				\item[(a)] Denoting by $\causset$ the causal-state memory of $q$'s output process $\rvalphra$, there is a bijection $c: \stateset \to \causset$ such that $c(\rvstatera)\times\rvalphra$ is the causal-state presentation of $\rvalphra$.
				\item[(b)] $q$ has the minimal number of internal states among all asymptotically synchronizable HMMs which are able to generate the same output-process as $q$ (upon suitable initialization).
				\item[(c)] $q$ has the minimal number of internal states among all unifilar HMMs which are able to generate the same output-process as $q$ (upon suitable initialization).
				\item[(d)] $q$ is unifilar and its internal states are predictively distinct in the sense that for any $j,k\in\stateset$:
				$$\Pr_{\rvalphra}(q,\bra{j}) = \Pr_{\rvalphra}(q,\bra{k}) \quad \Rightarrow\quad j=k.$$
			\end{itemize}
			Any of these equivalent properties specifies $q$ uniquely up to relabelling of internal states.
			\item[(ii)] Let $q\in\tildeM_\stateset$ and consider the equivalence relation $\sim$ on $\stateset$, given by
			$$j\sim k\quad \Leftrightarrow \quad \Pr_{\rvalphra}(q,\bra{j}) = \Pr_{\rvalphra}(q,\bra{k}).$$
			There is a unifilar HMM $q' \in\tildeM_{\stateset/\sim}$ with predictively distinct states and for any initialization $\nu$ of $q$ there is an initialization $\nu'$ of $q'$ such that
			$$\Pr_{\rvalphra}(q,\nu)=\Pr_{\rvalphra}(q',\nu').$$
		\end{itemize}
	\end{prop}
	\begin{proof}
		The equivalence of (a) and (d) has been shown in \cite{traversequivalence} for simple HMMs. The semi-simple case does not introduce substantial complications. For the other equivalences, note that, due to \cite{travers2011asymptotic}, asymptotic synchronizability is in the finite-state case equivalent to ``state observability'' of the extended combined process $\rvstatelra\times\rvalphlra$ in the sense of \cite{lohr2009models} and minimality of the causal state presentation of the output-process among state observable HMMs and unifilar HMMs is proved there as Cor. 3.40 and Cor. 3.42. The equivalence of (a) with (b) resp. (c) then follows from Thm. 3.41 and Prop. 3.20 of \cite{lohr2009models} which also imply (ii).
	\end{proof}
		
	\begin{defn}
		The semi-simple elements of $\tildeM_\stateset$ with predictively distinct states are called \emph{Causal-State Machines (CSMs)}.
	\end{defn}
	\begin{rem}
		As a consequence of Prop. \ref{prop:minimality}, we see that the processes which can be generated by unifilar HMMs are exactly the finitary processes from Def. \ref{def:finitary}. Moreover, the processes that can be generated by simple unifilar HMMs are exactly the \emph{ergodic} finitary ones, meaning those finitary $P\in\mathscr{P}[\alphsetra]$ satisfying $P[S]\in\{0,1\}$ for any shift-invariant event $S\subset\alphsetra$.
	\end{rem}
	The present section is concluded with an algebraic condition specifying the set of simple CSMs as a subset of $\M_\stateset$. Observe that, for a simple HMM $q\in\M_\stateset$, there is no ambiguity about the stationary initialization since Lem. \ref{lem:stationarystates} provides a unique stationary mixed state $\pi(q)$. In this case, we will henceforth write $\Pr(q)\equiv \Pr(q,\pi(q))$.
	
	\begin{prop}\label{cor:degvariety}
		Let $N=2|\stateset|-1$ and define $\mathcal{D}_\stateset\subset \R^{\stateset\times\alphset\times\stateset}$ as the set of real points of the algebraic set
		$$\bigcup_{j,k\in\stateset}\left(\bigcap_{\a_{1:N}\in\alphset_{1:N}}\mathcal{Z}\left((\bra{j}-\bra{k})\cdot q^{(\a_{1:N})}\cdot\one\right)\right)$$
		where $\mathcal{Z}(\cdot)$ denotes the zero locus of a polynomial in $\mathbb{C}[(q^{j,\a}_k)]$.
		The set of simple CSMs on the internal state set $\stateset$ is $\M_\stateset\setminus\mathcal{D}_\stateset$.
	\end{prop}
	\begin{proof}
		It is well-known (and can for example be gleaned from \cite{schonhuth2009characterization}), that $\Pr_{\rvalphra}(q)$ is determined by $\Pr_{\rvalph_{1:N}}(q)=\pi(q)\cdot q^{(\rvalph_{1:N})}\cdot\one$. Hence, the subset of non-CSMs in $\M_\stateset$ is comprised precisely of those unifilar simple HMMs $q$ possessing at least two internal states $j, k$ such that $\Pr_{\rvalph_{1:N}}(q,\bra{j})=\Pr_{\rvalph_{1:N}}(q,\bra{k})$, which is the set  $\mathcal{D}_\stateset\cap\M_\stateset$.
	\end{proof}

	\section{Geometry of the Relative Entropy Rate}\label{sec:relentrate}

The \emph{relative entropy rate}
$$\h(P'\parallel P):=\lim_{L\to\infty}\frac1L \sum_{\a_{1:L}\in\alphset_{1:L}} P'_{\rvalph_{1:L}}[\a_{1:L}]\log\frac{P'_{\rvalph_{1:L}}[\a_{1:L}]}{P_{\rvalph_{1:L}}[\a_{1:L}]}$$
is a positive-(semi-)definite function of two stochastic processes with laws $P, P'\in\mathscr{P}[\alphsetra]$ whenever the respective limit exists. In statistical estimation theory, it is also known as the \emph{Kullback-Leibler divergence rate} and commonly used to quantify the dissimilarity of $P'$ and $P$. Heuristically, it serves as a sort of ``distance'' - with the understanding that it is not a proper distance function e.g. failing to satisfy the triangle inequality. It however leads to an actual (semi-)Riemannian geometry on parameter manifolds of stochastic processes provided that it pulls back to a twice continuously differentiable function.\footnote{The analogous statement for the relative entropy of probability measures which either have finite support or possess a density is a basic fact in information geometry (\cite{amari2007methods}). The present results can be seen as an extension of that framework to the setting of ergodic finitary process measures.}
This construction shall be carried out in detail on the \emph{configuration space} $\M_\stateset\setminus\mathcal{D}_\stateset$ of simple CSMs. As a preparation, Sec. \ref*{subsec:geomsetup} exhibits this set as an open subset of the polytopal cell complex $\tildeM_\stateset$ of unifilar HMMs with the constituting open subsets of faces being viewed as differentiable manifolds. In Sec. \ref*{subsec:relentrate}, we shall see that the pullback of $\h$ to each of these manifolds can be expressed by an exact formula. In Sec. \ref*{subsec:Riemanniangeom}, we compute the induced Riemannian metric tensor $\mathrm{g}$ and relate $\mathrm{h}$ to the respective Riemannian action integral. In Sec. \ref*{subsec:processreplicator}, we establish the process-replicator equation (\ref{eq:replicatorconservative}) on configuration space and prove a statement about asymptotically stable rest points extending the folk theorem from evolutionary game theory.

\subsection{Combinatorial Geometric Setup}\label{subsec:geomsetup}
	The convex polytope $\tildeQ_\stateset$ comes with a natural face decomposition whose face poset is isomorphic to
	$$\set{S\subset\stateset\times\alphset\times\stateset}{\forall j\in\stateset\ \exists \a\in\alphset,\ k\in\stateset:\ (j,\a ,k)\in S}.$$
	The subset $\tildeM_\stateset\subset\tildeQ_\stateset$ of unifilar HMMs is a polytopal subcomplex of the $|\stateset|(|\alphset|-1)$-skeleton because all boundary faces of unifilar  faces are also unifilar.\footnote{In this report, a ``face'' is always an open face by convention.} To get a firmer grasp of the face-combinatorics of $\tildeM_\stateset$, we would like to classify the combinatorial types of unifilar transition graphs:	
	\begin{defn}
		Let $j\in\stateset$ and $L\in\N$.
		\begin{itemize}
			\item[(i)] For any $q\in\tildeM_\stateset$ and $\gamma=\Gamma(q)$ we set
			\begin{align*}
				\edgeset_{1:L}(\gamma,j)&:=\set{e_{1:L}\in\edgeset_{1:L}}{\rvstate(e_{1})=j \textnormal{ and the edge sequence } e_{1:L} \textnormal{ is a directed path in  } \gamma.}\\
				\alphset_{1:L}(\gamma,j)&:=\rvalph_{1:L}(\edgeset_{1:L}(\gamma,j))\\
				\edgeset_{1:L}(\gamma)&:=\bigcup_{j\in\stateset} \edgeset_{1:L}(\gamma,j)\\
				\alphset_{1:L}(\gamma)&:= \rvalph_{1:L}(\edgeset_{1:L}(\gamma))
			\end{align*}
			If $L=1$ we omit the subscript $1:L$.
			\item[(ii)] For any $q\in\tildeM_\stateset$, the graph $\gamma=\Gamma(q)$ is identified with the unique function
			$$\gamma:\ \bigcup_{j\in\stateset} \{j\}\times \alphset_{1:L}(\gamma,j)\ \to \stateset,$$
			satisfying $$\bra{j}\trans{\a_{1:L}}\ket{\gamma(j,\a_{1:L})}> 0\ ,\ \forall j\in\stateset,\ \a_{1:L}\in\alphset_{1:L}(\gamma,j),$$
			and called the \emph{DFA-type} of $q$.\footnote{Indeed $\gamma$ is just the transition function of a Deterministic Finite Automaton (DFA) on the state set $\stateset$ and alphabet $\alphset$.}
			\item[(iii)] $\reddfas_\stateset$ denotes the set of all DFA-types on the vertex set $\stateset$ and $\dfas_\stateset$ denotes the set of strongly-connected DFA-types. Furthermore $\reddfas_\stateset$ is made into a poset by setting $\gamma' < \gamma$ iff
			$$\mathrm{dom}(\gamma')\subsetneq\mathrm{dom}(\gamma)\qquad \textnormal{and}\qquad \gamma'=\gamma|_{\mathrm{dom}(\gamma')}.$$ In that case we call $\gamma'$ a \emph{subtype} of $\gamma$.
		\end{itemize}
	\end{defn}
	Observe that the projection
	$$\Gamma : \tildeM_\stateset \onto \reddfas_\stateset.$$
	is constant precisely on open faces thereby identifying $\reddfas_\stateset$ with the face poset of $\tildeM_\stateset$. We set $$\M_\gamma:= \Gamma^{-1}(\gamma).$$
	The subset $\M_\stateset\subset\tildeM_\stateset$ of simple unifilar HMMs forms an open subset in the combinatorial topology of the cell complex, because its complement of unifilar HMMs with not-strongly-connected DFA-types is a subcomplex.
	
	In order to exhibit the set of (semi-)simple CSMs as a collection of manifolds, it remains to describe the degenerate subsets $\mathcal{D}_\gamma:=\mathcal{D}_\stateset\cap \M_\gamma$. Fix $\gamma\in\reddfas_\stateset$ and consider the projection
	\begin{equation}\label{eq:faceembedding}
		\R^{\stateset\times\alphset\times\stateset}\onto\R^{\edgeset(\gamma)},\ q\mapsto \left(q^{j,\a}_{\gamma(j,\a)}\right)_{(j,\a)\in\edgeset(\gamma)}.
	\end{equation}
	It embeds $\M_\gamma$ as a semi-affine subspace\footnote{By this we mean the intersection of an affine subspace with a finite number of half-spaces.} of $\R^{\edgeset(\gamma)}$ which we again denote by $\M_\gamma$. Likewise we identify $\mathcal{D}_\gamma$ with its image under the above projection and write $q\equiv(q^{j,\a})_{(j,\a)\in\edgeset(\gamma)}$ in the \emph{ambient coordinates} on $\R^{\edgeset(\gamma)}$. Every $q\in\M_\gamma$ induces a partition $\Pi(q)$ of $\stateset$ into equivalence classes according to Prop. \ref{prop:minimality}.(ii). The set of partitions
	$\mathbb{L}_\gamma:=\set{\Pi(q)}{q\in\M_\gamma}$ is a lattice ordered by reverse refinement with maximal element $\{\stateset\}$ and minimal element $\{\{j\}\}_{j\in\stateset}$. For any $\sigma\in\mathbb{L}_\gamma$ we define
	$$\mathcal{D}_\gamma(\sigma):=\set{q\in\M_\gamma\subset \R^{\edgeset(\gamma)}}{\Pi(q)\geq\sigma}.$$
	This is the subset of unifilar HMMs such that, for any $s\in\sigma,\ j,k\in s,$ the internal states $j$ and $k$ are not predictively distinct.
	The collection $\{\mathcal{D}_\gamma(\sigma)\}_{\sigma\in\mathbb{L}_\gamma}$ can be partially ordered by reverse inclusion and we obtain the following  stratification of the degenerate locus:
	\begin{prop}\label{prop:subspacearrangement}
		Let $\gamma\in\reddfas_\stateset$. The map
		$$\mathbb{L}_\gamma\to\{\mathcal{D}_\gamma(\sigma)\}_{\sigma\in\mathbb{L}_\gamma}\ ,\quad \sigma\mapsto \mathcal{D}_\gamma(\sigma)$$ is an isomorphism of lattices. Removing the minimal element yields the intersection semi-lattice of the stratification $$\mathcal{D}_\gamma=\bigcup_{\sigma\in\mathbb{L}_\gamma \setminus \{\min \mathbb{L}_\gamma\}} \mathcal{D}_\gamma(\sigma)$$
		with the semi-affine subspaces
		\begin{equation}\label{eq:stratum}
			\mathcal{D}_\gamma(\sigma)=\set{q\in\M_\gamma\subset\R^{\edgeset(\gamma)}}{\forall s\in\sigma,\ j,k\in s,\ \a\in\alphset:\ q^{j,\a}-q^{k,\a}=0}.
		\end{equation}
	\end{prop}
	\begin{proof}
		The equation expressing $\mathcal{D}_\gamma$ as a union is clear by construction of the $\mathcal{D}_\gamma(\sigma)$. Now fix $\sigma\in\mathbb{L}_\gamma$. It is obvious that
		\begin{equation*}
			\mathcal{D}_\gamma(\sigma)\subset\set{q\in\M_\gamma\subset\R^{\edgeset(\gamma)}}{\forall s\in\sigma,\ j,k\in s,\ \a\in\alphset:\ q^{j,\a}-q^{k,\a}=0}.
		\end{equation*}
		Moreover, we have
		\begin{equation*}
			\forall \a \in\alphset;\ s_1, s_2 \in\sigma;\ j,j'\in s_1:\  \gamma(j,\a)\in s_2 \Leftrightarrow \gamma(j',\a)\in s_2,
		\end{equation*}
		which can be deduced from the fact that there exists $q_0\in\M_\gamma$ with $\sigma=\Pi(q_0)$ and, as internal states of $q_0$, any two members of the same element of $\sigma$ are predictively equivalent.
		It follows by induction on the length of output-words that, for any HMM $q$ in the right-hand side of (\ref*{eq:stratum}), any $j,k\in s\in\sigma$ are predictively equivalent as internal states of $q$. Therefore $q\in\mathcal{D}_\gamma(\sigma)$ which proves (\ref*{eq:stratum}).
		
		It remains to be shown that $\mathbb{L}_\gamma$ is isomorphic to the intersection lattice of the subspace arrangement. The question is whether $\mathcal{D}_\gamma(\sigma\vee \sigma')\subset \mathcal{D}_\gamma(\sigma)\cap \mathcal{D}_\gamma(\sigma')$ is actually an equality. But this follows directly from (\ref*{eq:stratum}).
	\end{proof}

	\subsection{The Relative Entropy Rate Formula}\label{subsec:relentrate}
	
	  \ \vspace{.3cm}
	
		\begin{center}
			\fbox{
				\begin{minipage}{15cm}
					For the remainder of this report, $\gamma\in\dfas_\stateset$ will denote a simple DFA-type and $\M_\gamma$ is considered a subset of the ambient space $\R^{\edgeset(\gamma)}$.
				\end{minipage}
			}
		\end{center}
	
	\vspace{.3cm}
	
	Having set up the \emph{configuration manifold} $\M_\gamma\setminus\mathcal{D}_\gamma$, we shall now turn our attention to the relative entropy rate between the respective output-processes, i.e. the function
	$$\h_\gamma:\ (\M_\gamma\setminus\mathcal{D}_\gamma)^{2}\to\R_{\geq 0}\ ,\quad (q',q)\mapsto \h_\gamma(q'\parallel q):=\mathrm{h}\big(\Pr_{\rvalphra}(q')\parallel \Pr_{\rvalphra}(q)\big).$$
	Existence and smoothness properties of the relative entropy rate restricted to output-processes of specific model classes have been investigated for a long time. Existence is for instance guaranteed in the case of two \emph{generic state-emitting} hidden Markov processes as demonstrated in the seminal papers \cite{baum1966statistical}, \cite{petrie1969probabilistic}. Note that finitary processes never fall in this class except for trivial cases. On the other hand, a particularly well-behaved process-class, which can always be realized by simple unifilar HMMs, is that of ergodic $\alphset$-valued Markov processes. In fact, there is an exact formula computing $\mathrm{h}$ in terms of the Markov chains' transition probabilities (\cite{rached2004kullback}). Obtaining similar exact expressions for more general process classes has proven to be notoriously difficult and seems indeed intractable for general stationary or even (state-emitting) hidden Markov processes (as remarked e.g. in \cite{finesso2010approximation}). However, the specific case of ergodic finitary processes looks promising due a well-known formula for the ordinary entropy rate (see Remark \ref{rem:entropyrate}) which originally goes back to Shannon (\cite{shannon2001mathematical}) and parallels the respective formula for ergodic Markov chains. The obvious guess is that the relative entropy rate formula also generalizes from the Markov to the finitary setting and the following theorem shows this to be true indeed.
	\begin{thm}\label{thm:relentrate}
		Let $\gamma\in\dfas_\stateset$ and $\mathcal{C}$ be a connected component of the configuration manifold $\M_\gamma\setminus\mathcal{D}_\gamma\subset\R^{\edgeset(\gamma)}$. The relative entropy rate pulls back to a (real-)analytic function on $\mathcal{C}\times \mathcal{C}$ given by
		\begin{equation}\label{eq:relentrate}
			 \h_\gamma(q'\parallel q)=\sum_{j\in\stateset}\pi_j(q')\sum_{\a\in\alphset(\gamma, j)} q'^{j,\a} \log\left(\frac{q'^{j,\a}}{q^{j,\a}}\right),
		\end{equation}
		for $\transtup, \transtup'\in\mathcal{C}$.
	\end{thm}
	\begin{proof}[Proof (Outline, Details in Appendix \ref{subsec:proofrelentrate})]
	We consider the cross-entropy
	$$\mathrm{Cr}(q'\parallel q)(\rvalph_{m:n}):= \left\langle -\log \Pr_{\rvalph_{m:n}}(q) \right\rangle_{\Pr_{\rvalph_{m:n}}(q')}$$ The conditional cross-entropy $\mathrm{Cr}(q'\parallel q)(\rvalph_{m:n} | \mathrm{X})$ is defined similarly by averaging $-\log \Pr_{\rvalph_{m:n}|\mathrm{X}}(q)$. Evidently, it is sufficent to show
	$$\lim_{L\to\infty} \frac1L \mathrm{Cr}(q'\parallel q)(\rvalph_{1:L}) = \lim_{L\to\infty} \mathrm{Cr}(q'\parallel q)(\rvalph_{L}|\rvstate_L) = \sum_{j\in\stateset}\pi_j(q')\sum_{\a\in\alphset} q'^{j,\a} \log q^{j,\a}.$$
	While the right-hand equality simply follows from stationarity, establishing the left-hand equality is the difficult part. In a first step one shows by the usual Cesàro-mean argument that
	$$\lim_{L\to\infty} \frac1L \mathrm{Cr}(q'\parallel q)(\rvalph_{1:L}) = \lim_{L\to\infty} \mathrm{Cr}(q'\parallel q)(\rvalph_{L}| \rvalph_{1:L-1})$$
	and the remaining task is to estimate $|\mathrm{Cr}(q'\parallel q)(\rvalph_{L}|\rvstate_L)-\mathrm{Cr}(q'\parallel q)(\rvalph_{L}| \rvalph_{1:L-1})|$. Intuitively, we need to make use of the asymptotic synchronization property, namely, that there is a large subset of sufficiently long output-strings $\a_{1:L-1}$ which constrain the subsequent mixed states $\pi(q)\cdot q^{(\a_{1:L-1})}$ and $\pi(q')\cdot q'^{(\a_{1:L-1})}$ to nearly pure states. The crucial result is that these dominating pure states are in fact identical for $q$ and $q'$ provided they lie in the same connected component of $\M_\gamma\setminus \mathcal{D}_\gamma$. This result extends the non-exact machine synchronization theorem of \cite{travers2011asymptotic} by showing that the latter doesn't just hold at the points $q$ and $q'$ separately but also locally on sufficiently small neighbourhoods inside $\M_\gamma\setminus \mathcal{D}_\gamma$. One then employs a compactness argument to show that it holds globally on the connected component $\mathcal{C}$.
	\end{proof}
	
	\begin{rem}\label{rem:entropyrate}
		Let $\gamma\in\dfas_\stateset$ and $q\in\M_\gamma$. There exists a maximal DFA-type $\gamma<\gamma_0\in\dfas_\stateset$, meaning that the cell $\M_{\gamma_0}$ is top-dimensional in $\M_\stateset$ and that $\M_\gamma\subset\bar{\M}_{\gamma_0}$. Hence there is a connected component $\mathcal{C}$ of $\M_{\gamma_0}\setminus\mathcal{D}_{\gamma_0}$ such that $q\in\bar{\mathcal{C}}$. On the other hand, due to Prop. \ref{prop:subspacearrangement}, the HMM $q_0=(1/|\alphset|)_{e\in\edgeset(\gamma_0)}$ lies in $\bar{\mathcal{C}}\cap\M_{\gamma_0}$. We can extend $\mathrm{h}_{\gamma_0}$ continuously to $\bar{\mathcal{C}}\times(\bar{\mathcal{C}}\cap\M_{\gamma_0})$ and apply Thm. \ref{thm:relentrate} to obtain
			$$\lim_{\stack{\tilde{q}\to q}{\tilde{q}_0\to q_0}}\mathrm{h}\big(\Pr_{\rvalphra}(\tilde{q})\parallel \Pr_{\rvalphra}(\tilde{q}_0)\big) = \mathrm{h}_{\gamma_0}(q\parallel q_0) = \log |\alphset|+ \sum_{j\in\stateset}\pi_j(q)\sum_{\a\in\alphset(\gamma,j)} q^{j,\a} \log q^{j,\a},$$
		where the limit is performed for $\tilde{q}, \tilde{q}_0\in\mathcal{C}$. Up to the constant $\log|\alphset|$, this formula has been known for a long time (\cite{shannon2001mathematical},\cite{travers2011asymptotic}) as minus the ordinary entropy rate of the simple unifilar HMM $q$ with DFA-type $\gamma$. It is now recovered as a limiting case of Thm. \ref{thm:relentrate} being the entropy rate relative to the ``uninformative prior'' machine $q_0$.
	\end{rem}

	\subsection{Differential Geometry and the Entropy Rate Tensor}\label{subsec:Riemanniangeom}
	
	In this section we relate the positive-definite function $\mathrm{h}_\gamma(\cdot\parallel q)$ to its local asymptotic approximation at $q\in\M_\gamma\setminus\mathcal{D}_\gamma$ which equips $\M_\gamma$ with a Riemannian metric $\mathrm{g}$. It will be useful in order to study continuum limits of stochastic processes on $\M_\gamma$, specifically, to asymptotically compute the cumulated relative entropy rates of their trajectories. Namely, assume $(q_m^L)_{0\leq m\leq n_L}, L\in\N,$ is a sequence of such trajectories. Its cumulated relative entropy rates are $\sum_{m=1}^{n_L} \mathrm{h}_\gamma(q^L_m\parallel q^L_{m-1})$. If the trajectories, for $L\to\infty$, approach the image of some curve $c: [0,1]\to \M_\gamma$, the cumulated relative entropy rates tend to
	$$\mathrm{E}(c):=\lim_{\delta\to 0}\sup_{(t_l)\in \mathcal{Z}_\delta} \sum_{l=1}^n \mathrm{h}_\gamma(c(t_l)\parallel c(t_{l-1})),$$
	where $\mathcal{Z}_\delta$ is the set of finite subdivisions $0=t_0<t_1<\ldots < t_n=1,\ n\in\N,$ of the unit interval with $|t_l -t_{l-1}|<\delta,\ l=1,\ldots,n$.	
	To compute this quantity, we need to study the 2nd order local approximation of the smooth function $\mathrm{h}_\gamma(\cdot\parallel q)$. Positive-definiteness guarantees that we have $\d\mathrm{h}_\gamma(\cdot\parallel q)|_{q}=0$ and therefore the following is well-defined:
	\begin{defn}\label{defn:metrictensor}
		\begin{itemize}
			\item[(i)] The \emph{entropy rate tensor} associated to $\mathrm{h}_\gamma$ on $\M_\gamma$ is the section $\mathrm{g}\in\mathrm{Sym}(\mathcal{T}^*(\M_\gamma)\otimes \mathcal{T}^*(\M_\gamma))$ given by
			\begin{align*}
			\mathcal{T}(\M_\gamma) \otimes \mathcal{T}(\M_\gamma) &\to \mathcal{O}(\M_\gamma)\\ 
			(\xi,\eta) &\mapsto \big(q\mapsto (\xi\circ\eta)\big{|}_q \mathrm{h}_\gamma(\cdot\parallel q)\big),
			\end{align*}
			where $\mathcal{O}(\mathcal{M}_\gamma)$ is the algebra of smooth functions on $\mathcal{M}_\gamma$ and $\mathcal{T}(\M_\gamma)$ resp. $\mathcal{T}^*(\M_\gamma)$ are the modules of smooth vector fields resp. differential one-forms.
			\item[(ii)] $\mathrm{E}(c)$ is called the \emph{energy} of the curve $c$ whenever it exists and, for any $q,q'\in\M_\gamma$, we set
			$$\mathrm{E}(q'\parallel q):=\inf_{c}\frac12\int_0^1 \mathrm{g}(c(t))\cdot \dot{c}(t)^{\otimes 2} \d t,$$
			where the infimum is taken over all piecewise smooth curves with endpoints $c(0)=q$ and $c(1)=q'$.
		\end{itemize}
	\end{defn}
	
	\begin{prop}\label{prop:hessianmetricjet}
		The entropy rate tensor $\mathrm{g}$ is a Riemannian metric on $\M_\gamma$. Furthermore:
		\begin{enumerate}
			\item For any $q\in\M_\gamma$, the function $\mathrm{E}(\cdot \parallel q)$ is a 2nd order approximation of $\mathrm{h}_\gamma(\cdot \parallel q)$, i.e.
			$$\mathcal{J}_{q}^{2} \mathrm{E}( \cdot \parallel q) = \mathcal{J}_{q}^{2} \mathrm{h}_\gamma( \cdot \parallel q)\ ,\quad \textnormal{for any } q\in\mathcal{M},$$
			where $\mathcal{J}_{q}^{2}$ stands for the second-order jet at $q$.
			\item For any piecewise smooth curve $c:[0,1]\to\M_\gamma$, we have 
			$$\mathrm{E}(c)=\frac12\int_0^1 \mathrm{g}(c(t))\cdot \dot{c}(t)^{\otimes 2} \d t.$$		
		\end{enumerate}
	\end{prop}
	\begin{proof}
		See Appendix \ref{subsec:proofhessianmetricjet}.
	\end{proof}
	We are now going to use Thm. \ref{thm:relentrate} to derive an explicit expression for $\mathrm{g}$ in terms of the ambient coordinates of $\R^{\edgeset(\gamma)}$. For this purpose, we view $\M_\gamma$ as a submersed submanifold via
	$$\R^{\edgeset(\gamma)}_+\to \R^\stateset\ ,\quad (q^{j,\a})\mapsto (q^j)_{j\in\stateset}:=\left(\sum_{\a\in\alphset(\gamma,j)} q^{j,\a} \right)_{j\in\stateset}$$
	at the regular value $(1,\ldots,1)$. The modules of smooth vector fields resp. differential one-forms can be expressed as
	\begin{align*}
		\mathcal{T}(\M_\gamma) &= \bigcap_{j\in \stateset} \mathrm{ker} (\d q^j)\\
		\mathcal{T}^*(\M_\gamma) &= \bigoplus_{(j,\a )\in\edgeset(\gamma)} \mathcal{O}(\M_\gamma)\cdot \d q^{j,\a}\ \bigg/ \ \bigoplus_{j\in\stateset} \mathcal{O}(\M_\gamma)\cdot \d q^j.
	\end{align*}
	By slight abuse of notation we will denote the equivalence class of $\d q^{j,\a}$ in $\mathcal{T}^*(\M_\gamma)$ also by $\d q^{j,\a}$.
	The Riemannian metric $\mathrm{g}$ induces a module-isomorphism
	$$^\flat:\ \mathcal{T}(\M_\gamma) \to \mathcal{T}^*(\M_\gamma), \ \xi\mapsto \mathrm{g}\cdot (\xi\otimes \cdot)$$
	called the \emph{index-lowering isomorphism}. Its inverse is called the \emph{index-raising isomorphism} $^\sharp$.
	\begin{cor}\label{cor:entropyratemetric}
		Let $\gamma\in\dfas_\stateset$. The entropy rate tensor $\mathrm{g}$ on $\M_\gamma$ has the ambient coordinate expression
		$$\mathrm{g}(q) = \sum_{(j,\a)\in\edgeset(\gamma)} \frac{\pi_j(q)}{q^{j,\a}}\ (\d q^{j,\a})^{\otimes 2}$$
		and the associated index-raising isomorphism is given by
		$$^\sharp:\ \eta \mapsto \sum_{(j,\a)\in\edgeset(\gamma)}\frac{q^{j,\a}}{\pi_j(q)} \left(\eta\cdot\frac{\partial}{\partial q^{j,\a}}- \eta\cdot\left(\sum_{\mathsf{b}\in\alphset(\gamma,j)} q^{j,\mathsf{b}} \frac{\partial}{\partial q^{j,\mathsf{b}}}\right)\right)\ \frac{\partial}{\partial q^{j,\a}}.$$
	\end{cor}
	\begin{proof}
	These formulae can be shown by straightforward computations.
	\end{proof}
	
	\begin{rem}\label{rem:coordinatechange}
		Under the ambient coordinate change $q^e(z)=(z^e)^2$, $e\in\edgeset(\gamma)$, the entropy rate tensor changes into
		$$\mathrm{g}=\sum_{(j,\a)\in\edgeset(\gamma)} \pi_j(q(z))\ (\d z^{j,\a})^{\otimes 2},$$
		which reveals that $\mathrm{g}$ can in fact be continuously extended to the partially closed faces $\bar{\M}_\gamma\cap \M_\stateset$ equipping them with the length metric $\sqrt{2 \mathrm{E}(\cdot\parallel\cdot)}$. This observation could be used to turn all of $\M_\stateset$ into a length-metric space by composing paths across different faces. There is however a more elegant way to obtain that result by extending $\mathrm{g}$ itself to a global Riemannian metric on all of $\M_\stateset$. The procedure involves a suitable ramified differentiable structure to handle transitions between different faces and will be presented in a separate note.
	\end{rem}
	
	\subsection{The Process-Replicator Equation}\label{subsec:processreplicator}
	
	We shall now see that the entropy rate tensor may be used to formally extend the replicator equation (\ref{eq:classicalreplicator}) to the setting of finitary processes. It is well-known that (\ref{eq:classicalreplicator}) can be expressed in a coordinate-free manner on the open simplex $\Sigma^\alphset$, equipped with the Fisher metric $\mathrm{g}=\sum_{\a} \frac{1}{q^\a} (\d q^\a)^{\otimes 2}$. Explicitly, one collects the selection functions into a differential one-form $\phi:=\sum_{\a\in\alphset} f_\a \d q^\a$ and rewrites (\ref{eq:classicalreplicator}) as
	\begin{equation}\label{eq:replicatornonconservative}
	\dot{q} = \phi^\sharp.
	\end{equation}
	where $^\sharp$ is the index-raising isomorphism, induced by the Fisher metric.\footnote{All these are well-known facts from information geometry, where $\mathrm{g}$ is called the Fisher information metric (\cite{cover2012elements},\cite{amari2007methods}), and theoretical biology, where it sometimes goes under the name Shahshahani metric (\cite{shahshahani1979new}, \cite{akin1990differential}).}  We will be particularly interested in the case of a closed selection form, i.e. $\d \phi =0$. In this case, there exists a potential\footnote{An example is the linear selection case from evolutionary game theory which leads to the potential $\Phi(q)=\frac12 \sum_{\a, \mathsf{b}}  r_{\a\mathsf{b}}q^\a q^{\mathsf{b}}$.} $\Phi: \Sigma^\alphset \to \R$ with $\phi = \d\Phi$ and equation (\ref*{eq:replicatornonconservative}) reads
	$$\dot{q}= \nabla_\mathrm{g} \Phi,$$
	where $\nabla_\mathrm{g}\Phi=\d\Phi^\sharp$ is the Riemannian gradient of $\Phi$ with respect to the Fisher metric. Note that the latter is in fact just the entropy rate tensor on $\M_\alpha=\Sigma^\alphset$ if $\alpha$ is taken to be the DFA-type from Fig. \ref{fig:1statedfa}.
	This suggests the following generalization:
	\begin{defn}
		Let $\mathcal{C}\subset \M_\gamma$ be an open subset. A \emph{fitness potential} on $\mathcal{C}$ is a function $\Phi: \Pr_{\rvalphra}(\mathcal{C})\to\R$ such that $(\Phi\circ\Pr_{\rvalphra})|_{\mathcal{C}}$ is continuously differentiable. The associated \emph{process-replicator equation} is the differential equation
		$$\dot{q}=\nabla_\mathrm{g}(\Phi\circ\Pr_{\rvalphra}).$$
	\end{defn}
	
	\begin{cor}\label{cor:procrepeq}
		Let $\mathcal{C}\subset\M_\gamma$ be open, $\Phi$ be a fitness potential on $\mathcal{C}$ and denote $f_e:=\frac{\partial}{\partial q^e}(\Phi\circ\Pr_{\rvalphra})$, $e\in\edgeset(\gamma)$. The associated process-replicator equation can be written in the ambient coordinates as
		$$\dot{q}^{j,\a} = \frac{q^{j,\a}}{\pi_j(q)}\left(f_{j,\a}(q)-\bar{f}_j(q)\right)\ ,\quad \text{for } (j,\a)\in\edgeset(\gamma),$$
		where $\bar{f}_j(q)=\sum_{\a\in\alphset(\gamma,j)} q^{j,\a} f_{j,\a}(q)$.
		\qed
	\end{cor}
	In evolutionary game theory, the concept of evolutionary stable states has proven useful to analyze local stability properties of the replicator dynamics. A configuration $q\in\Sigma^\alphset$ is called locally evolutionary stable iff it has a neighbourhood $\mathcal{U}$ such that, for any $q'\in\mathcal{U}$, the restriction of the replicator dynamics to the affine line segment between $q$ and $q'$ has a unique asymptotically-stable rest point at $q$, which is the case iff
	\begin{equation}\label{eq:evolutionarystablestates}
		\sum_{\a\in\alphset} q^\a f_\a(q') > \sum_{\a\in\alphset} q'^\a f_\a(q').
	\end{equation}
	This condition has the intuitive meaning that, if a population is comprised of two types of coherent subpopulations in configurations $q$ resp. $q'$ (which cannot modifiy their internal compositions), then a small enough quantity of $q'$-subpopulations will eventually die out in a $q$-dominated population (see e.g. \cite{taylor1978evolutionary}). The \emph{folk theorem} of evolutionary game theory states that a locally evolutionary stable configuration is also an asymptotically-stable rest point for the replicator equation.\footnote{Note that the converse is not true since the replicator equation can have asymptotically-stable rest points which are not evolutionary stable (\cite{hofbauer2003evolutionary}).} We can now extend the folk theorem to ergodic finitary processes:
	\begin{cor}
		Let $\mathcal{C}\subset\M_\gamma$ be open, $\Phi$ be a fitness potential on $\mathcal{C}$ and denote $f_e:=\frac{\partial}{\partial q^e}(\Phi\circ\Pr_{\rvalphra})$, $e\in\edgeset(\gamma)$. If some $q\in\mathcal{C}$ has a neighbourhood $\mathcal{U}\subset\mathcal{C}$ such that any $q'\in\mathcal{U}\setminus\{q\}$ satisfies
		$$\sum_{\a \in\alphset(\gamma,j)} q^{j,\a} f_{j,\a}(q') > \bar{f}_j(q')\ ,\quad \text{for all } j\in\stateset,$$
		then $q$ is an asymptotically stable rest point for the process-replicator dynamics.		
	\end{cor}
	\begin{proof}
		It is straightforward to verify that the condition makes $\mathrm{h}_\gamma(q\parallel \cdot)$ into a local Lyapunov-function with isolated minimum at $q$.
	\end{proof}

	The situation where the fitness potential is given by the negative relative entropy rate to some goal process $\Pr_{\rvalphra}(q_\infty)$ is of fundamental interest in statistical estimation theory (\cite{baum1966statistical},\cite{finesso2010approximation}). It yields a particularly simple process-replicator equation:
\begin{cor}\label{cor:relentprocessrepl}
	Let $\mathcal{C}$ be a connected component of $\M_\gamma\setminus\mathcal{D}_\gamma$ and $q_\infty\in\mathcal{C}$. The process-replicator equation on $\mathcal{C}$ with fitness potential $-\mathrm{h}\big(\Pr_{\rvalphra}(q_\infty)\parallel \Pr_{\rvalphra}(\cdot)\big)$ has the following ambient coordinate expression:
	$$\dot{q}^{j,\a} = \frac{\pi_j(q_\infty)}{\pi_j(q)}(q^{j,\a}_\infty-q^{j,\a})\ ,\quad (j,\a)\in\edgeset(\gamma).$$
\end{cor}
\begin{proof}
	This is an immediate consequence of Cor. \ref{cor:procrepeq}.
\end{proof}
The resulting integral curves will be illustrated in Sec. \ref{sec:examples} for a few low-dimensional examples.
	
	\section{The Generalized Wright-Fisher Process}\label{sec:evolution}
	
		We will now substantiate the process-replicator dynamics' significance by deriving it from the gWF-process's asymptotic behaviour. Recall that the gWF-process's transitions  occur according to the following
	\begin{center}
		\fbox{
			\begin{minipage}[c]{0.8\linewidth}
				\textbf{gWF-iteration:}
				\begin{enumerate}
					\item A (previously selected) procreator carrying a unifilar HMM in configuration $q\in\bar{\M}_\gamma\cap\M_\stateset$ reproduces by generating edge-sequences $e^1_{1:L},\ldots, e^N_{1:L}$ according to $\Pr_{\rvedge_{1:L}}(q)$ and imprinting them into $N$ offspring according to the relative edge counts
					$$\mathrm{Q}^L(e^m_{1:L}):=\left(\frac{\mathrm{card}\set{1\leq l\leq L}{e^m_l=(j,\a)}}{\mathrm{card}\set{1\leq l\leq L}{\rvstate(e^m_l)=j}}\right)_{(j,\a)\in\edgeset(\gamma)}$$
					\item Given a fitness potential $\Phi$, one of the $N$ configurations $\mathrm{Q}^L(e^m_{1:L})$ is selected as the next procreator upon maximization of $\Phi\circ\Pr_{\rvalphra}$.
				\end{enumerate}
			\end{minipage}
		}
	\end{center}
	This scheme only requires minimal computational capabilities on the agent level. In particular, agents don't possess an explicit representation of the configuration space $\bar{\M}_\gamma\cap\M_\stateset$. Indeed, an agent's potential to evolve stems solely from the ability to imprint its own array of behaviours into finite resamplings of itself. The innate variance in the agent's output distribution is thus transformed into a diversity of possible offspring behaviours. In fact, as explained in the introduction, the self-resampling step (i) can be seen as a generalization of the Wright-Fisher evolutionary iteration. This holds the possibility of importing ideas about genetic drift into the theory of finitary processes.\footnote{The connection of genetic drift with CSMs has previously been probed in \cite{crutchfield2012structural} using computer simulations. The present mathematical framework shows a path which is more transparent with respect to the underlying information-theoretic mechanisms.}
	The gWF-iteration can be viewed as a Markov transition kernel $R^L_{N,\Phi}$ on $\bar{\M}_\gamma\cap\M_\stateset$ and it shall be demonstrated that its expectation dynamics $q\mapsto \langle q' \rangle_{R^L_{N,\Phi}[q'|q]}$ asymptotically follows the process-replicator integral curves in the limit $L\to\infty$. For that purpose, Sec. \ref{subsec:empiricalmachines} will show that the probability distribution $\Pr_{Q^L}(q)$ is governed by a large deviation principle with rate function $\mathrm{h}_\gamma(\cdot \parallel q)$. Sec. \ref{subsec:fluctuationmetric} then establishes a corresponding central limit theorem which ultimately reveals the significance of $\mathrm{g}$ as a fluctuation tensor describing the directional covariances  of infinitesimal deviations. Sec. \ref{subsec:greedyasymptotics} makes use of these results to prove that the expectation trajectories of $R^L_{N,\Phi}$ asymptotically follow the integral curves of the process-replicator equation.

	\subsection{Large Deviations of the Empirical Configuration Ensemble}\label{subsec:empiricalmachines}
	
	This section elaborates on the probability distribution of the random variable $\mathrm{Q}^L$ whose image will be denoted by $\bar{\M}_\gamma^{L}$.
	\begin{defn}
		Let $q\in\M_\gamma$ and $L\in\N$. The \emph{empirical configuration ensemble} at sampling length $L$ based at $q\in\bar{\M}_\gamma\cap\M_\stateset$ is the probability space $\left(\bar{\M}_\gamma^{L}, \Pr_{\mathrm{Q}^{L}}(q)\right)$.
	\end{defn}
	The ensemble satisfies the following large deviation principle:
	\begin{thm}\label{thm:largeconsistentdev}
		Let $q\in\M_\gamma$ and $q_L\in\bar{\M}_\gamma^{L}$, $L\in\N$, with $\lim_{L\to\infty} q_L\in\M_\gamma$. Then we have
		$$\Pr_{\mathrm{Q}^{L}}(q)[q_L]= \kappa_L(q_L)\ \e^{-L\mathrm{h}_\gamma(q_L\parallel q)}$$
		for functions $\kappa_L$ satisfying $\lim_{L\to\infty}\frac1L \log \kappa_L(q_L)= 0$.
	\end{thm}
	\begin{proof}[Proof (Outline, Details in Appendix \ref{subsec:prooflargedev})] We factor $\mathrm{Q}^L$ according to
		\begin{center}
			\begin{tikzcd}
				\edgeset_{1:L}(\gamma) \arrow[d,"\hat{\pi}"] \arrow[rd,"\mathrm{Q}^L"] & \\
				\bar{\Sigma}^{\edgeset(\gamma)} \arrow[r,two heads,swap,"\psi"] & \bar{\M}_\gamma\cap\M_\stateset
			\end{tikzcd}
		\end{center}
		with the maps
		\begin{align*}
			\hat{\pi}&:\ \edgeset_{1:L}(\gamma)\to \bar{\Sigma}^{\edgeset(\gamma)}\ ,\quad e_{1:L}\mapsto \left(\frac{1}{L}\mathrm{card}\set{1\leq l\leq L}{e_l=(j,\a)}\right)_{(j,\a)\in\edgeset(\gamma)},\\
			\psi&:\ \bar{\Sigma}^{\edgeset(\gamma)}\onto \bar{\M}_\gamma\ ,\quad (x^{j,\a})\mapsto \left(\frac{x^{j,\a}}{\sum_{\mathsf{b}\in\alphset(\gamma,j)} x^{j,\mathsf{b}}}\right).
		\end{align*}
		A point $x=\hat{\pi}(e_{1:L})$ will be identified with the multi-graph having edgeset $\bigsqcup_{1\leq l\leq L} \{e_l\}$. By this device, $\hat{\pi}(\edgeset_{1:L}(\gamma))$ corresponds bijectively to the set of directed edge-$\alphset$-labelled multi-graphs with vertex set $\stateset$ and $L$ edges, picked from the set $\edgeset(\gamma)$, admitting an Euler path. The number of edges $j\overset{\a}{\to} \gamma(j,\a)$ in such a graph $x$ is given by $L x^{j,\a}$. We call these graphs the \emph{empirical types}. Observe that if the empirical type $x$ admits an Euler path with labelled edge sequence $e_{1:L}$ then
		$$\Pr_{\rvedge_{1:L}}(q)[e_{1:L}]= \pi_{\rvstate(e_1)}(q)\ \prod_{(j,\a)\in\edgeset(\gamma)} (q^{j,\a})^{Lx^{j,\a}}$$
		and thus the conditional probability $\Pr_{\rvedge_{1:L}|\rvstate_1}(q)[e_{1:L} | \rvstate(e_1)]$ does not depend on the actual edge sequence but only on the empirical type. Therefore, the value of $\hat{\pi}_*\Pr_{\rvedge_{1:L}}(q)[x]$ can be computed by multiplying this conditional probability with $\sum_{e_{1:L}\in\hat{\pi}^{-1}(x)} \pi_{\rvstate(e_1)}(q)$. The latter sum can be bounded from above and below provided one can compute the cardinality of $\hat{\pi}^{-1}(x)$, i.e. the number of different Euler paths in the empirical type. By standard arguments of graph theory this problem can be reduced to the problem of counting arborescences and their number can in turn be computed as a determinant by invoking Kirchhoff's matrix-tree theorem. One obtains the asymptotic behaviour
		$$\hat{\pi}_*\Pr_{\rvedge_{1:L}}(q)[x]=\hat{\kappa}_L(x) \e^{-L\mathrm{h}_\gamma(\psi(x)\parallel q)},$$
		for some function $\hat{\kappa}_L$ satisfying the growth condition $\lim_{L\to\infty}\frac1L \log\hat{\kappa}_L(x)=0$.
		Pushing this forward along $\psi$ yields the expression from the theorem with $\kappa_L(q_L)=\sum_{x\in\psi^{-1}(q_L)\cap\hat{\pi}(\edgeset_{1:L}(\gamma))} \hat{\kappa}_L(x)$. It remains to be shown that $\kappa_L$ satisfies the same growth condition as $\hat{\kappa}_L$ which can readily be verified by observing that the number of points in $\hat{\pi}(\edgeset_{1:L}(\gamma))$ grows at most like a power of $L$.
	\end{proof}
	
	\begin{rem}
		One should bear in mind that, although the large deviation rate function $\mathrm{h}_\gamma(\cdot\parallel q)$ is defined on $\bar{\M}_\gamma\cap\M_\stateset$, it represents the relative entropy rate only on the connected component of $\M_\gamma\setminus \mathcal{D}_\gamma$ containing $q$ (see Thm. \ref{thm:relentrate}).
	\end{rem}
	
	\subsection{The Rôle of $\mathrm{g}$ as a Fluctuation Tensor}\label{subsec:fluctuationmetric}
	
	Viewing the empirical configuration ensembles $\Pr_{\mathrm{Q}^{L}}(q),\ L\in\N$, as measures on $\bar{\M}_\gamma\cap\M_\stateset$, the limit measure for $L\to\infty$ is the Dirac measure at $q$. The manner in which convergence to this limit occurs will turn out to be important in order to study the asymptotics of the gWF-process. 
	We are going to derive a central limit theorem showing that an appropriately rescaled weak limit of the empirical configuration ensemble yields a Gaussian measure on the tangent space at $q$. The latter is hereby viewed as the set of infinitesimal empirical deviations at $q$.
	
	We identify $T_q\M_\gamma$ with the linear subspace $\set{v\in\R^{\edgeset(\gamma)}}{\forall j\in\stateset:\ \sum_{\a\in\alphset(\gamma,j)} \bra{(j,\a )}\cdot v=0}$ and define the rescaling embedding
	$$\mathrm{u}_{L}: \M_\gamma\times\M_\gamma \to T\M_\gamma\ ,\quad (q,q')\mapsto \mathrm{u}_{q,L}(q'):=\left(q, L^{\frac12} (q'-q)\right).$$
	We would like to state the central limit theorem in terms of convergence of functionals. Define the space of \emph{configuration observables}  as the space of bounded Borel-measurable functions $\mathscr{L}^\infty(\bar{\M}_\gamma\cap\M_\stateset)$ which factor through $\Pr_{\rvalphra}$. The expectation value of a configuration observable with respect to an empirical configuration ensemble is a well-defined functional
	$f\mapsto \langle f \rangle_{\Pr_{\mathrm{Q}^{L}}(q)},$
	where $\Pr_{\mathrm{Q}^{L}}(q)$ is regarded as a probability measure on $\bar{\M}_\gamma\cap\M_\stateset$. Pushing it forward along $\mathrm{u}_{q,L}$
	we obtain a functional on $\mathscr{L}^\infty(T_q\M_\gamma)$:
	
	\begin{thm}\label{thm:centrallimitthm}
		Let $q\in\M_\gamma$. The sequence $\left(\langle\cdot\rangle_{(\mathrm{u}_{q,L})_*\Pr_{\mathrm{Q}^{L}}(q)}\right)_{L\in\N}$ of functionals on $\mathscr{L}^\infty(T_q\M_\gamma)$ converges weakly to the functional
		$$\Lambda:\ f\mapsto \frac{1}{(2\pi)^{\frac12\dim(\M_\gamma)}}\int_{T_q\M_\gamma} f(v)\ \e^{-\frac{1}{2}\mathrm{g}(q)\cdot v^{\otimes 2}} |\omega_{\mathrm{g},q}|,$$
		where $|\omega_{\mathrm{g},q}|$ is the constant density on the affine manifold $T_q\M_\gamma$ given by evaluating the canonical Riemannian density $|\omega_\mathrm{g}|$ at $q$.
	\end{thm}
	\begin{proof}[Proof (Outline, Details in Appendix \ref{subsec:proofcentrallimitthm})]
		We fix $q\in\M_\gamma$ and consider the measures
		$$\mu_{L}[S]:=\langle \chi_S\rangle_{(\mathrm{u}_{q,L})_*\Pr_{\mathrm{Q}^{L}}(q)}\quad \textnormal{and}\quad \mu[S]:=\Lambda(\chi_S)$$
		for $S$ a Borel-subset of $T_q\M_\gamma$. Due to the Portmanteau-Theorem we need to show
		\begin{equation}\label{eq:portmanteauoutline}
		\mu[U]\leq\liminf_{L\to\infty}\mu_L[U]\ ,\quad \textnormal{for any open set } U\subset T_q\M_\gamma.
		\end{equation}
		By a standard argument exploiting the fact that $\mu$ is a probability measure on a $\sigma$-compact metric space, it is enough to show (\ref*{eq:portmanteauoutline}) for $U$ relatively-compact. We again use the maps $\hat{\pi}, \psi$ like in the proof of Thm. \ref{thm:largeconsistentdev} as well as the 
		diffeomorphic section $\varphi$ of $\psi$ from Lem. \ref{lem:emptypemap} given by
		$$q\mapsto (\pi_j(q) q^{j,\a})_{(j,\a)\in\edgeset(\gamma)}.$$
		All the occurring maps are displayed in the following commutative diagram		
		\begin{center}
			\begin{tikzcd}
				\edgeset_{1:L}(\gamma) \arrow[d,"\hat{\pi}"] \arrow[rd,"\mathrm{Q}^L"] & T_q\M_\gamma\\
				\bar{\Sigma}^{\edgeset(\gamma)} \arrow[r,two heads,swap,"\psi"] & \bar{\M}_\gamma\cap\M_\stateset \arrow[l,bend left,"\varphi"] \arrow[u,swap,"\mathrm{u}_{q,L}"]
			\end{tikzcd}
		\end{center}
		The proof is now subdivided into 4 steps:
		
		\emph{Step 1}: We examine the discrete point-set $\hat{\pi}(\edgeset_{1:L}(\gamma))\cap\Sigma^{\edgeset(\gamma)}$. More precisely, we will define an auxiliary measure $\hat{\nu}_L$ on $\Sigma^{\edgeset(\gamma)}$ supported on this discrete set such that for any $x\in \hat{\pi}(\edgeset_{1:L}(\gamma))\cap\Sigma^{\edgeset(\gamma)}$:
		$$\hat{\pi}_*\Pr_{\rvedge_{1:L}}(q)[x]=\sum_{j\in\stateset} \mathrm{card}\ (\hat{\pi}^{-1}(x))\cdot \Pr_{\rvedge_{1:L}| \rvstate_1}(q)[e_{1:L} | j]\cdot  \hat{\nu}_L[x],$$
		where $e_{1:L}\in\hat{\pi}^{-1}(x)$ can be chosen arbitrarily. We show that there exists a grid of parallelepipeds covering $\hat{\pi}(\edgeset_{1:L}(\gamma))$ such that any two parallelepipeds which lie in $\Sigma^{\edgeset(\gamma)}$ have the same measure under $\hat{\nu}_L$.
		
		\emph{Step 2}: We push $\hat{\nu}_L$ forward along $\mathrm{u}_{q,L}\circ \psi$ to a measure $\nu_L$ on $T_q\M_\gamma$ and use step 1 to construct a certain collection $\mathfrak{K}_L(U)$ of pairwise disjoint subsets of $U$ with the following three properties for $L\to\infty$:
		\begin{itemize}
			\item[-] The union of the sets from $\mathfrak{K}_L(U)$ asymptotically exhausts $U$.
			\item[-] The diameters of the sets from $\mathfrak{K}_L(U)$ go uniformly to zero.
			\item[-] For any $K,K'\in \mathfrak{K}_L(U)$, we have uniformly $\frac{\nu_L(K)}{\nu_L(K')}\to 1$.
		\end{itemize}
		
		\emph{Step 3:} We define another auxiliary measure $\lambda_L$ on $\mathrm{u}_{q,L}(\M_\gamma)\subset T_q\M_\gamma$ as the integral with respect to a certain volume form and show that we have $\frac{\nu_L(K)}{\lambda_L(K)}\to 1$, uniformly for $K\in \mathfrak{K}_L(U)$.
		
		\emph{Step 4:} We express $\mu_L(U)$ in terms of $\nu_L(U)$ and estimate it from below using the results of steps 1, 2 and 3 as well as the large deviation principle from Thm. \ref{thm:largeconsistentdev}, verifying (\ref*{eq:portmanteauoutline}).
	\end{proof}	

	\begin{rem}\label{rem:fluctuationmetric}
		The above theorem makes contact with thermodynamic fluctuation theory by recovering the entropy rate tensor $\mathrm{g}$ as a fluctuation tensor in the following sense: $T\M_\gamma$ is made into a bundle of probability spaces through the fibre measure
		$$S\mapsto \Lambda[\chi_S]\ ,\quad S\subset T_q\M_\gamma\ \textnormal{Borel-measurable}.$$
		If we pick some frame $(\xi_m)_{1\leq m\leq \dim \M_\gamma}$ of $T\M_\gamma$ and consider the $\mathrm{g}$-dual coframe $(\xi_m^\flat)$, we can view the $\xi_m^\flat(q)$ as random variables $T_q\M_\gamma\to\R$ and easily verify
		$$\mathrm{Cov}(\xi_l^\flat,\xi_m^\flat)=\mathrm{g}\cdot(\xi_l\otimes \xi_m).$$
		In Section \ref{subsec:Riemanniangeom}, we defined the energy $\mathrm{E}(c)$. Intuitively, the above theorem says that it asymptotically quantifies the amount of infinitesimal empirical fluctuations making up the curve $c$. In the context of fluctuation theory and finite-time thermodynamics, this notion has previously been investigated for statistical mechanical systems under the name thermodynamic divergence (\cite{crooks2007measuring},\cite{ruppeiner1995riemannian}). The counterpart of empirical fluctuations is in that context played by the system's thermodynamic fluctuations.
	\end{rem}

\subsection{Convergence of Expectation Trajectories}\label{subsec:greedyasymptotics}
	
	We shall now use Thm. \ref{thm:centrallimitthm} in order to derive the process-replicator equation from the asymptotic behaviour of the gWF-process  $\rho^{L,\tau}$ which starts with an initial probability measure $\rho^{L,\tau}(0)$ over $\bar{\M}_\gamma\cap\M_\stateset$ and then iterates the transition kernel $R^L_{N,\Phi}$ at times $t_n:=n \tau,\ n\in\N$, for a fixed generation length $\tau>0$. It thus produces a sequence of probability distributions $(\rho^{L,\tau}(t_n))_{n\in\N}$ supported on $\bar{\M}_\gamma^L$ which satisfy the recursion
	$$\rho^{L,\tau}(t_n)[q']:=\sum_{q\in \bar{\M}_\gamma^{L}} R^L_{N,\Phi}[q'|q]\ \rho^{L,\tau}(t_{n-1})[q].$$
	We focus on the associated deterministic expectation process, starting at $q(0)\in \M_\gamma$ and iterating the transition
	$q\mapsto \langle q' \rangle_{R^L_{N,\Phi}[q' |q]}$
	at times $t_n$. The resulting trajectory is made into a piecewise-linear path $q^{L,\tau}$ by setting
	\begin{equation}\label{eq:expectationiteration}
		q^{L,\tau}(0):=q(0)\ ,\quad q^{L,\tau}(t_{n+1}):= \langle q' \rangle_{R^L_{N,\Phi}[q' | q^{L,\tau}(t_{n})]}
	\end{equation}
	and linear interpolation
	$$q^{L,\tau}(t):= \frac{t_{n+1}-t}{t_{n+1}-t_n} q^{L,\tau}(t_n)+ \frac{t-t_n}{t_{n+1}-t_n} q^{L,\tau}(t_{n+1}),$$
	for $t\in (t_n,t_{n+1})$.
	
	\begin{thm}\label{thm:graddescent}
		Let $\mathcal{C}$ be an open subset of $\M_\gamma\setminus\mathcal{D}_\gamma$ and $\Phi$ be a fitness potential with nowhere-vanishing differential on $\mathcal{C}$. For any $\beta>0$ and any sequence $(\tau_L)$ satisfying $\lim_{L\to\infty}L \tau_L^2=\beta$, the sequence of paths $q^{L}:=q^{L,\tau_L}$, $L\in\N$, with fixed initial point $q^{L}(0)=:q(0)\in\mathcal{C}$ converges uniformly to the respective integral curve of
		$$\dot{q}=\frac{\alpha(N)}{\sqrt{\beta}} \frac{\nabla_{\mathrm{g}}(\Phi\circ\Pr_{\rvalphra})(q)}{\|\nabla_{\mathrm{g}}(\Phi\circ\Pr_{\rvalphra})(q)\|}$$
		on any interval $[0,T]$ on which it remains in $\mathcal{C}$. Herein, $\alpha(N)$ is the last order statistic of $N$ independent standard normal random variables which can be expressed as
		$$\alpha(N)= \frac{N 2^{2-N}}{\sqrt{2\pi}} \sum_{l=1}^{\lfloor N/2 \rfloor} \begin{pmatrix} N-1 \\ 2l -1\end{pmatrix} \int_{-\infty}^{\infty} x \e^{-x^2} \mathrm{erf}(x)^{2l-1} \d x.$$
	\end{thm}
	\begin{proof}[Proof (Outline, Details in Appendix \ref{subsec:proofgraddescent})]
		Combining Thm. \ref{thm:centrallimitthm} with a dominated convergence argument, we obtain
		$$\langle q' \rangle_{R^L_{N,\Phi}[q' |q]}=q+L^{-\frac12}\sum_{m=1}^{\dim(\M_\gamma)}\Lambda[f^m] v_m + o(L^{-\frac12}),$$
		where $(v_m)$ is any $\mathrm{g}$-orthonormal basis of $T_q\M_\gamma$ with $v_1=\frac{\nabla_{\mathrm{g}}(\Phi\circ\Pr_{\rvalphra})(q)}{\|\nabla_{\mathrm{g}}(\Phi\circ\Pr_{\rvalphra})(q)\|_\mathrm{g}}$ and the $f^m$ are functions on $T_q\M_\gamma$ with $\Lambda[f^m]=0$, for $m> 1$, and $\Lambda[f^1]=\alpha(N)$. Hence:
		$$q^{L}(t_{n+1})=q^{L}(t_n)+L^{-\frac12} X(q^{L}(t_n)) + o(L^{-\frac12})\ ,\quad \textnormal{with } X(q)= \alpha(N) \frac{\nabla_{\mathrm{g}}(\Phi\circ\Pr_{\rvalphra})(q)}{\|\nabla_{\mathrm{g}}(\Phi\circ\Pr_{\rvalphra})(q)\|_\mathrm{g}}.$$
		It only remains to be verified that the (linearly interpolated) paths $q^{L}$ converge uniformly to the respective integral curve. This is accomplished by standard ODE-methods.
	\end{proof}
	
	\begin{rem}
		Exact expressions for the last order statistic $\alpha(N)$ are known for $N\leq 5$ (see \cite{david1970order}). For instance, $\alpha(2)=\frac{1}{\sqrt{\pi}}$ and $\alpha(3)=\frac{3}{2\sqrt{\pi}}$.
	\end{rem}
	
	\begin{cor}
		The limits of expectation trajectories $(q^{L}(t_n))_{n\in\N}$ of the gWF-process with some fitness potential $\Phi$ as in Thm. \ref{thm:graddescent} are constant-speed reparametrizations of the respective process-replicator integral curves.
		\qed
	\end{cor}
	
	If some goal process with law $P\in\mathscr{P}[\alphsetra]$ is given and $\Phi$ is set to be the negative of some divergence function relative to $P$, e.g. minus the relative entropy rate, then the gWF-process describes the parameter learning aspect of evolutionary generator reconstruction.\footnote{The inference of DFA-types is of course at least equally important. There seems to be a principled way of incorporating it into the present framework led by the observation that the set of finitary processes with arbitrarily large (finite) causal-state memory can be made into a cell complex. This topic will be explored on another occasion.}
	\begin{cor}\label{cor:asympttrajrelentrate}
		Let $q_\infty\in \M_\gamma\setminus\mathcal{D}_\gamma$ and $\mathcal{C}$ be the connected component of $\M_\gamma\setminus\mathcal{D}_\gamma$ containing $q_\infty$.
		The expectation trajectories in $\mathcal{C}$ of the gWF-process with fitness potential $\Phi=-\mathrm{h}(\Pr_{\rvalphra}(q_\infty) \parallel \cdot\ )$ converge to (constant-speed reparametrizations of) the integral curves of
		$$\dot{q}^{j,\a} = \frac{\pi_j(q_\infty)}{\pi_j(q)}(q^{j,\a}_\infty-q^{j,\a})\ ,\quad (j,\a)\in\edgeset(\gamma).$$
		\qed
	\end{cor}
	 We saw in Thm. \ref{thm:relentrate} that  $\mathrm{h}(\Pr_{\rvalphra}(q_\infty) \parallel \Pr_{\rvalphra}(\cdot))$ is convex on $\mathcal{C}$. It follows that any asymptotic expectation trajectory $q(\cdot)$ with initial point $q(0)\in\mathcal{C}$ converges to $q_\infty$. On the other hand, it was suggested in Rem. \ref{rem:fluctuationmetric} that the energy $\mathrm{E}(q(\cdot))$ quantifies asymptotically the amount of infinitesimal empirical fluctuations necessary to generate $q(\cdot)$. The difference relative to an energy-minimizing geodesic with the same endpoints, $$\mathrm{E}(q(\cdot))-\mathrm{E}(q_\infty \parallel q(0))\geq 0,$$
	 intuitively measures ``how misdirected'' the expected parameter inference steps are in total compared to the optimal choices at every point. This quantity can be used to assess the effect of different fitness potentials as well as starting points on the gWF-scheme's efficiency. The last section illustrates the dependence on starting points for a few low-dimensional examples.

	\newpage
	
	\section{Examples}\label{sec:examples}
	
		\subsection{One Internal State}
	
	The case $\stateset=\{1\}$ corresponds to the replicator dynamics. We have $\M_\stateset=\tildeM_\stateset=\bar{\Sigma}^\alphset$. Every $\gamma\in\dfas_\stateset$, corresponding to a face of the closed simplex $\bar{\Sigma}^\alphset$, consists of several loops at the internal state $1$ (see Fig. \ref{fig:1statedfa}). We abbreviate $\alphset(\gamma)\equiv\alphset(\gamma,1)$ and $q^\a \equiv q^{1,\a}$. The output-process generated by some $q\in\M_\gamma$ is just an i.i.d. random sequence of letters from $\alphset(\gamma)$ and, accordingly, we have
	$$\mathrm{h}\big(\rvalphra_*\Pr(q')\parallel \rvalphra_*\Pr(q)\big)=\mathrm{h}_\gamma(q'\parallel q) =\sum_{\a \in \alphset(\gamma)} q'^\a \log\frac{q'^\a}{q^\a}.$$
	The resulting entropy rate tensor
	$$\mathrm{g}=\sum_{\a\in\alphset(\gamma)}\frac{1}{q^\a} (\d q^\a)^{\otimes 2}$$
	is the well-known Fisher information metric as previously noted. The transformation of ambient coordinates $q^\a(z)= (z^{\a})^2 $ from Rem. \ref{rem:coordinatechange} is particularly illuminating in this case. It yields
	$$\mathrm{g}=4 \sum_{\a\in\alphset(\gamma)} (\d z^\a)^{\otimes 2},$$
	which is just $4$ times the Euclidean metric of the ambient space, restricted to the first orthant of the unit hypersphere. Thus, in these ambient coordinates, the energy-minimizing geodesics are precisely great circle segments. On the other hand, considering the gWF-scheme with fitness potential $\Phi=-\mathrm{h}(\Pr_{\rvalphra}(q_\infty) \parallel \cdot\ )$, we see that the system of differential equations from Cor. \ref{cor:asympttrajrelentrate} decouples and, setting $\beta:=\alpha(N)^2$, it leads to the asymptotic expectation trajectories being the linear unit-speed curves
	$$q(t)= q(0)+t \frac{q_\infty-q(0)}{\|q_\infty-q(0)\|_{\mathrm{g}}} \ ,\quad \textnormal{for } 0\leq t\leq  \|q_\infty-q(0)\|_{\mathrm{g}}.$$
	Fig. \ref*{fig:1stateflow} illustrates the difference between these asymptotic trajectories and the geodesics ending at $q_\infty$ corresponding to asymptotically optimal inference paths.
	
	\begin{figure}[H]
			\centering
			\includegraphics[scale=0.9]{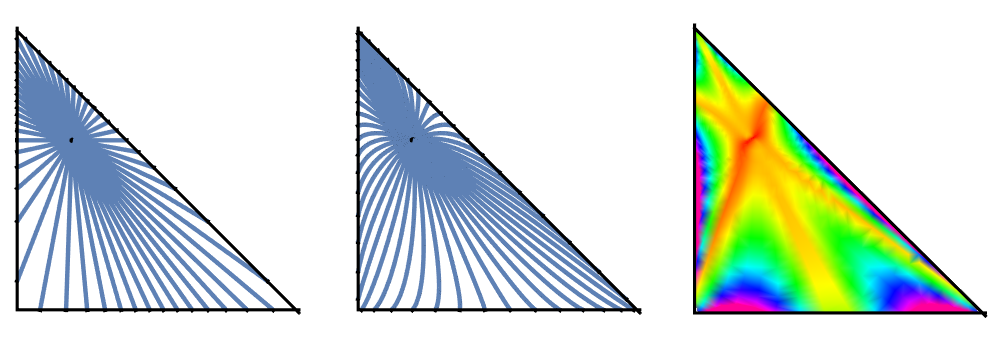}
			\caption{Inference dynamics for $|\stateset|=1,\ |\alphset|=3$. Left: The asymptotic expectation trajectories of the gWF-scheme with fitness potential $-\mathrm{h}(\Pr_{\protect\rvalphra}(q_\infty) \parallel \cdot\ )$ and $(q_\infty^{\mathsf{0}}, q_\infty^{\mathsf{1}}, q_\infty^{\mathsf{2}})=(0.2,0.6,0.2)$. Middle: Geodesics with respect to the entropy rate tensor $\mathrm{g}$ ending at $q_\infty$. Right: Colour-coded differences in path divergence from the respective initial points to $q_\infty$.}
			\label{fig:1stateflow}
	\end{figure}
	
	\subsection{Two Internal States on a Binary Alphabet}
	
	In the case $\stateset=\{1,2\}$ and $\alphset=\{\mathsf{0},\mathsf{1}\}$, the polytopal cell complex $\tildeM_\stateset$ is the subdivision of a torus into 16 top-dimensional rectangular faces and $\M_\stateset$ is a domain made up of $3\times 3$ top-dimensional faces with intermittent lower dimensional walls but without the boundary as illustrated in Fig. \ref*{fig:twostateuhmms}. There are in total 64 unifilar faces in $\tildeM_\stateset$ and the subdomain $\M_\stateset$ is composed of 25 of those faces. Every top-dimensional rectangular face $\M_\gamma$ is disconnected along a diagonal by $\mathcal{D}_\gamma=\set{q\in\M_\gamma}{q^{1,0}=q^{2,0}}$. Among the triangles which come from top-dimensional simple faces after removing $\mathcal{D}_\gamma$, there are only 5 types with qualitatively different  gWF-dynamics since all other domains can be obtained by either permuting internal states or swapping the alphabet letters $\mathsf{0}$ and $\mathsf{1}$. Asymptotic expectation trajectories of the gWF-scheme and geodesics with the same final points are illustrated in Fig. \ref*{fig:modeldynamics} for each of the qualitatively different domain types.
	
	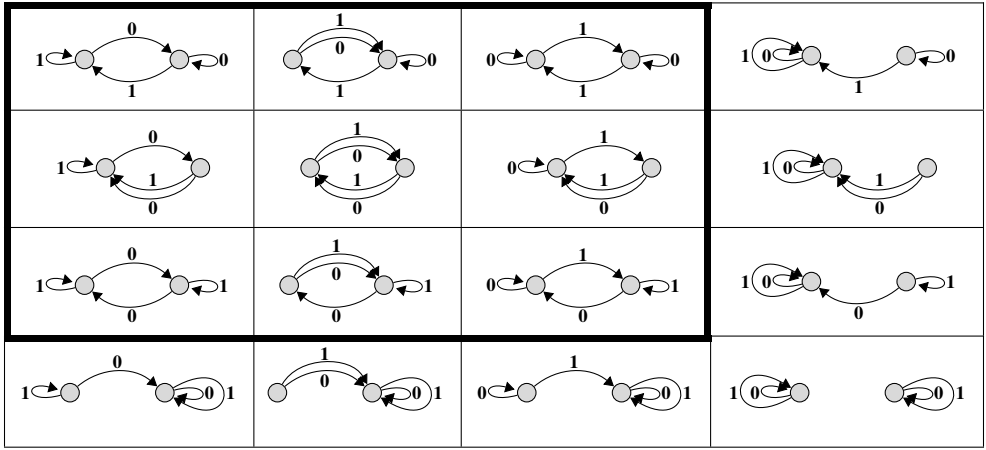
\begin{figure}[H]
		\centering		
		\begin{tabular}{|c|c|c|c|}
		\Xcline{1-3}{2.5pt} \cline{4-4}
		\multicolumn{1}{!{\vrule width 2.5pt}c|}{
		\gape{\begin{tikzpicture}[start chain = going right,
		-Triangle, every loop/.append style = {-Triangle}]
		\node[state, on chain, inner sep=0pt, minimum size=7pt,fill=gray!30]  (1) {};
		\node[state, on chain, inner sep=0pt, minimum size=7pt,fill=gray!30]  (2) {};
		
		\draw    (1)  edge[loop left, inner sep=1 pt, "1", font=\tiny\bfseries] (1);
		\draw    (1)  edge[bend left=40, inner sep=1 pt, "0", font=\tiny\bfseries]   (2);
		\draw    (2)  edge[loop right, inner sep=1 pt, "0", font=\tiny\bfseries]   (2);
		\draw    (2)  edge[bend left=40, inner sep=1 pt, "1", font=\tiny\bfseries]   (1);
		\end{tikzpicture}}
		}
		&
		\gape{
		\begin{tikzpicture}[start chain = going right,
		-Triangle, every loop/.append style = {-Triangle}]
		\node[state, on chain, inner sep=0pt, minimum size=7pt,fill=gray!30]  (1) {};
		\node[state, on chain, inner sep=0pt, minimum size=7pt,fill=gray!30]  (2) {};
		
		\draw    (1)  edge[bend left=40, inner sep=1 pt, "0", font=\tiny\bfseries,swap] (2);
		\draw    (1)  edge[bend left=60, inner sep=1 pt, "1", font=\tiny\bfseries]   (2);
		\draw    (2)  edge[loop right, inner sep=1 pt, "0", font=\tiny\bfseries]   (2);
		\draw    (2)  edge[bend left=40, inner sep=1 pt, "1", font=\tiny\bfseries]   (1);
		\end{tikzpicture}}
		&
		\multicolumn{1}{c!{\vrule width 2.5pt}}{
		\gape{\begin{tikzpicture}[start chain = going right,
		-Triangle, every loop/.append style = {-Triangle}]
		\node[state, on chain, inner sep=0pt, minimum size=7pt,fill=gray!30]  (1) {};
		\node[state, on chain, inner sep=0pt, minimum size=7pt,fill=gray!30]  (2) {};
		
		\draw    (1)  edge[bend left=40, inner sep=1 pt, "1", font=\tiny\bfseries] (2);
		\draw    (1)  edge[loop left, inner sep=1 pt, "0", font=\tiny\bfseries]   (1);
		\draw    (2)  edge[loop right, inner sep=1 pt, "0", font=\tiny\bfseries]   (2);
		\draw    (2)  edge[bend left=40, inner sep=1 pt, "1", font=\tiny\bfseries]   (1);
		\end{tikzpicture}}
		}
		&
		\gape{\begin{tikzpicture}[start chain = going right,
		-Triangle, every loop/.append style = {-Triangle}]
		\node[state, on chain, inner sep=0pt, minimum size=7pt,fill=gray!30]  (1) {};
		\node[state, on chain, inner sep=0pt, minimum size=7pt,fill=gray!30]  (2) {};
		
		\draw    (1)  edge[loop left, inner sep=1 pt, "0", font=\tiny\bfseries] (1);
		\draw    (1)  edge[loop left, out=210, in=150, looseness=20, inner sep=1 pt, "1", font=\tiny\bfseries]   (1);
		\draw    (2)  edge[loop right, inner sep=1 pt, "0", font=\tiny\bfseries]   (2);
		\draw    (2)  edge[bend left=40, inner sep=1 pt, "1", font=\tiny\bfseries]   (1);
		\end{tikzpicture}}
		\\\hline
		\multicolumn{1}{!{\vrule width 2.5pt}c|}{
		\gape{\begin{tikzpicture}[start chain = going right,
		-Triangle, every loop/.append style = {-Triangle}]
		\node[state, on chain, inner sep=0pt, minimum size=7pt,fill=gray!30]  (1) {};
		\node[state, on chain, inner sep=0pt, minimum size=7pt,fill=gray!30]  (2) {};
		
		\draw    (1)  edge[loop left, inner sep=1 pt, "1", font=\tiny\bfseries] (1);
		\draw    (1)  edge[bend left=40, inner sep=1 pt, "0", font=\tiny\bfseries]   (2);
		\draw    (2)  edge[bend left=40, inner sep=1 pt, "1", font=\tiny\bfseries,swap]   (1);
		\draw    (2)  edge[bend left=60, inner sep=1 pt, "0", font=\tiny\bfseries]   (1);
		\end{tikzpicture}}
		}
		&
		\gape{\begin{tikzpicture}[start chain = going right,
		-Triangle, every loop/.append style = {-Triangle}]
		\node[state, on chain, inner sep=0pt, minimum size=7pt,fill=gray!30]  (1) {};
		\node[state, on chain, inner sep=0pt, minimum size=7pt,fill=gray!30]  (2) {};
		
		\draw    (1)  edge[bend left=40, inner sep=1 pt, "0", font=\tiny\bfseries,swap] (2);
		\draw    (1)  edge[bend left=60, inner sep=1 pt, "1", font=\tiny\bfseries]   (2);
		\draw    (2)  edge[bend left=40, inner sep=1 pt, "1", font=\tiny\bfseries,swap]   (1);
		\draw    (2)  edge[bend left=60, inner sep=1 pt, "0", font=\tiny\bfseries]   (1);
		\end{tikzpicture}}
		&
		\multicolumn{1}{c!{\vrule width 2.5pt}}{
		\gape{\begin{tikzpicture}[start chain = going right,
		-Triangle, every loop/.append style = {-Triangle}]
		\node[state, on chain, inner sep=0pt, minimum size=7pt,fill=gray!30]  (1) {};
		\node[state, on chain, inner sep=0pt, minimum size=7pt,fill=gray!30]  (2) {};
		
		\draw    (1)  edge[bend left=40, inner sep=1 pt, "1", font=\tiny\bfseries] (2);
		\draw    (1)  edge[loop left, inner sep=1 pt, "0", font=\tiny\bfseries]   (1);
		\draw    (2)  edge[bend left=40, inner sep=1 pt, "1", font=\tiny\bfseries,swap]   (1);
		\draw    (2)  edge[bend left=60, inner sep=1 pt, "0", font=\tiny\bfseries]   (1);
		\end{tikzpicture}}
		}
		&
		\gape{\begin{tikzpicture}[start chain = going right,
		-Triangle, every loop/.append style = {-Triangle}]
		\node[state, on chain, inner sep=0pt, minimum size=7pt,fill=gray!30]  (1) {};
		\node[state, on chain, inner sep=0pt, minimum size=7pt,fill=gray!30]  (2) {};
		
		\draw    (1)  edge[loop left, inner sep=1 pt, "0", font=\tiny\bfseries] (1);
		\draw    (1)  edge[loop left, out=210, in=150, looseness=20, inner sep=1 pt, "1", font=\tiny\bfseries]   (1);
		\draw    (2)  edge[bend left=40, inner sep=1 pt, "1", font=\tiny\bfseries,swap]   (1);
		\draw    (2)  edge[bend left=60, inner sep=1 pt, "0", font=\tiny\bfseries]   (1);
		\end{tikzpicture}}
		\\\hline
		\multicolumn{1}{!{\vrule width 2.5pt}c|}{
		\gape{\begin{tikzpicture}[start chain = going right,
		-Triangle, every loop/.append style = {-Triangle}]
		\node[state, on chain, inner sep=0pt, minimum size=7pt,fill=gray!30]  (1) {};
		\node[state, on chain, inner sep=0pt, minimum size=7pt,fill=gray!30]  (2) {};
		
		\draw    (1)  edge[loop left, inner sep=1 pt, "1", font=\tiny\bfseries] (1);
		\draw    (1)  edge[bend left=40, inner sep=1 pt, "0", font=\tiny\bfseries]   (2);
		\draw    (2)  edge[bend left=40, inner sep=1 pt, "0", font=\tiny\bfseries]   (1);
		\draw    (2)  edge[loop right, inner sep=1 pt, "1", font=\tiny\bfseries]   (2);
		\end{tikzpicture}}
		}
		&
		\gape{\begin{tikzpicture}[start chain = going right,
		-Triangle, every loop/.append style = {-Triangle}]
		\node[state, on chain, inner sep=0pt, minimum size=7pt,fill=gray!30]  (1) {};
		\node[state, on chain, inner sep=0pt, minimum size=7pt,fill=gray!30]  (2) {};
		
		\draw    (1)  edge[bend left=40, inner sep=1 pt, "0", font=\tiny\bfseries,swap] (2);
		\draw    (1)  edge[bend left=60, inner sep=1 pt, "1", font=\tiny\bfseries]   (2);
		\draw    (2)  edge[bend left=40, inner sep=1 pt, "0", font=\tiny\bfseries]   (1);
		\draw    (2)  edge[loop right, inner sep=1 pt, "1", font=\tiny\bfseries]   (2);
		\end{tikzpicture}}
		&
		\multicolumn{1}{c!{\vrule width 2.5pt}}{
		\gape{\begin{tikzpicture}[start chain = going right,
		-Triangle, every loop/.append style = {-Triangle}]
		\node[state, on chain, inner sep=0pt, minimum size=7pt,fill=gray!30]  (1) {};
		\node[state, on chain, inner sep=0pt, minimum size=7pt,fill=gray!30]  (2) {};
		
		\draw    (1)  edge[bend left=40, inner sep=1 pt, "1", font=\tiny\bfseries] (2);
		\draw    (1)  edge[loop left, inner sep=1 pt, "0", font=\tiny\bfseries]   (1);
		\draw    (2)  edge[bend left=40, inner sep=1 pt, "0", font=\tiny\bfseries]   (1);
		\draw    (2)  edge[loop right, inner sep=1 pt, "1", font=\tiny\bfseries]   (2);
		\end{tikzpicture}}
		}
		&
		\gape{\begin{tikzpicture}[start chain = going right,
		-Triangle, every loop/.append style = {-Triangle}]
		\node[state, on chain, inner sep=0pt, minimum size=7pt,fill=gray!30]  (1) {};
		\node[state, on chain, inner sep=0pt, minimum size=7pt,fill=gray!30]  (2) {};
		
		\draw    (1)  edge[loop left, inner sep=1 pt, "0", font=\tiny\bfseries] (1);
		\draw    (1)  edge[loop left, out=210, in=150, looseness=20, inner sep=1 pt, "1", font=\tiny\bfseries]   (1);
		\draw    (2)  edge[bend left=40, inner sep=1 pt, "0", font=\tiny\bfseries]   (1);
		\draw    (2)  edge[loop right, inner sep=1 pt, "1", font=\tiny\bfseries]   (2);
		\end{tikzpicture}}
		\\\Xcline{1-3}{2.5pt} \cline{4-4}
		\gape{\begin{tikzpicture}[start chain = going right,
		-Triangle, every loop/.append style = {-Triangle}]
		\node[state, on chain, inner sep=0pt, minimum size=7pt,fill=gray!30]  (1) {};
		\node[state, on chain, inner sep=0pt, minimum size=7pt,fill=gray!30]  (2) {};
		
		\draw    (1)  edge[loop left, inner sep=1 pt, "1", font=\tiny\bfseries] (1);
		\draw    (1)  edge[bend left=40, inner sep=1 pt, "0", font=\tiny\bfseries]   (2);
		\draw    (2)  edge[loop left, out=30, in=330, looseness=20,inner sep=1 pt, "1", font=\tiny\bfseries]   (2);
		\draw    (2)  edge[loop right, inner sep=1 pt, "0", font=\tiny\bfseries]   (2);
		\end{tikzpicture}}
		&
		\gape{\begin{tikzpicture}[start chain = going right,
		-Triangle, every loop/.append style = {-Triangle}]
		\node[state, on chain, inner sep=0pt, minimum size=7pt,fill=gray!30]  (1) {};
		\node[state, on chain, inner sep=0pt, minimum size=7pt,fill=gray!30]  (2) {};
		
		\draw    (1)  edge[bend left=40, inner sep=1 pt, "0", font=\tiny\bfseries,swap] (2);
		\draw    (1)  edge[bend left=60, inner sep=1 pt, "1", font=\tiny\bfseries]   (2);
		\draw    (2)  edge[loop left, out=30, in=330, looseness=20,inner sep=1 pt, "1", font=\tiny\bfseries]   (2);
		\draw    (2)  edge[loop right, inner sep=1 pt, "0", font=\tiny\bfseries]   (2);
		\end{tikzpicture}}
		&
		\gape{\begin{tikzpicture}[start chain = going right,
		-Triangle, every loop/.append style = {-Triangle}]
		\node[state, on chain, inner sep=0pt, minimum size=7pt,fill=gray!30]  (1) {};
		\node[state, on chain, inner sep=0pt, minimum size=7pt,fill=gray!30]  (2) {};
		
		\draw    (1)  edge[bend left=40, inner sep=1 pt, "1", font=\tiny\bfseries] (2);
		\draw    (1)  edge[loop left, inner sep=1 pt, "0", font=\tiny\bfseries]   (1);
		\draw    (2)  edge[loop left, out=30, in=330, looseness=20,inner sep=1 pt, "1", font=\tiny\bfseries]   (2);
		\draw    (2)  edge[loop right, inner sep=1 pt, "0", font=\tiny\bfseries]   (2);
		\end{tikzpicture}}
		&
		\gape{\begin{tikzpicture}[start chain = going right,
		-Triangle, every loop/.append style = {-Triangle}]
		\node[state, on chain, inner sep=0pt, minimum size=7pt,fill=gray!30]  (1) {};
		\node[state, on chain, inner sep=0pt, minimum size=7pt,fill=gray!30]  (2) {};
		
		\draw    (1)  edge[loop left, inner sep=1 pt, "0", font=\tiny\bfseries] (1);
		\draw    (1)  edge[loop left, out=210, in=150, looseness=20, inner sep=1 pt, "1", font=\tiny\bfseries]   (1);
		\draw    (2)  edge[loop left, out=30, in=330, looseness=20,inner sep=1 pt, "1", font=\tiny\bfseries]   (2);
		\draw    (2)  edge[loop right, inner sep=1 pt, "0", font=\tiny\bfseries]   (2);
		\end{tikzpicture}}
		\\\hline
		\end{tabular}
		\caption{The cell complex $\tildeM_\stateset$. Opposite sides are supposed to be identified and each top-dimensional face shows its DFA-type. The upper-left $3\times 3$ subdomain (without the bold boundary) is the subset of simple unifilar HMMs $\M_\stateset$.}
		\label{fig:twostateuhmms}
	\end{figure}
	
	\subsection{More than Two Internal States}
	
	Here, it is more difficult to visualize the polytopal complex $\tildeM_\stateset$. It is comprised of $|\stateset|^{|\stateset||\alphset|}$ top-dimensional polytopes, each of which is the $|\stateset|$-th power of a $(|\alphset|-1)$-dimensional simplex, resulting in a total of $((|\stateset|+1)^{|\alphset|}-1)^{|\stateset|}$ faces resp. DFA-types. The number of simple DFA-types also grows at least like $\mathcal{O}(|\stateset|^{|\stateset|})$ as demonstrated in \cite{RADKE1965377}. We illustrate the asymptotic dynamics for the family of binary \emph{Nemo processes} which have been investigated for example in \cite{mahoney2009information}. Their DFA-type $\gamma$ is given by
	
	\begin{center}
	\begin{tikzpicture}[-Triangle, every loop/.append style = {-Triangle}]
	\node[state, inner sep=0pt, minimum size=7pt,fill=gray!30]  (1) {};
	\node[state, below left=1, inner sep=0pt, minimum size=7pt,fill=gray!30]  (2) {};
	\node[state, below right=1, inner sep=0pt, minimum size=7pt,fill=gray!30]  (3) {};
	
	\draw    (1)  edge[loop above, out=45, in=135, looseness=10, inner sep=1 pt, "1", font=\tiny\bfseries, swap] (1);
	\draw    (1)  edge[bend right=40, inner sep=1 pt, "0", font=\tiny\bfseries, swap]   (2);
	\draw    (2)  edge[bend right, inner sep=1 pt, "0", font=\tiny\bfseries, swap]   (3);
	\draw    (3)  edge[bend left=40, inner sep=1 pt, "1", font=\tiny\bfseries]   (1);
	\draw    (3)  edge[bend right=40, inner sep=1 pt, "0", font=\tiny\bfseries, swap]   (1);
	\end{tikzpicture}
\end{center}
	
	It corresponds to a 2-dimensional rectangular face $\M_\gamma$ whose degenerate subset $\mathcal{D}_\gamma$ is empty. The asymptotic gWF-expectation dynamics on $\M_\gamma$ is shown in Fig. \ref*{fig:nemoflow}.

	\begin{figure}[H]
		\centering
		\begin{tabular}{c}
			\includegraphics[scale=0.9]{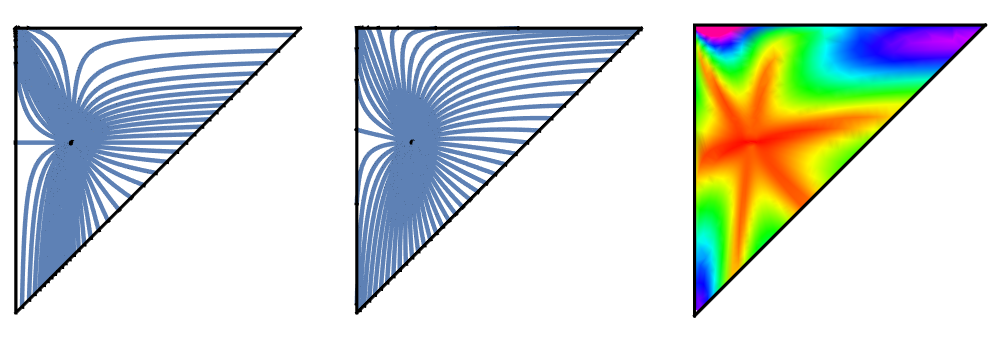}\\
			\includegraphics[scale=0.9]{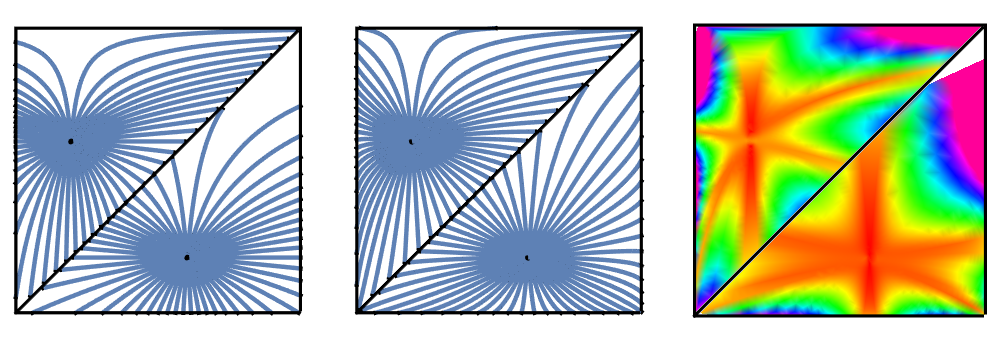}\\
			\includegraphics[scale=0.9]{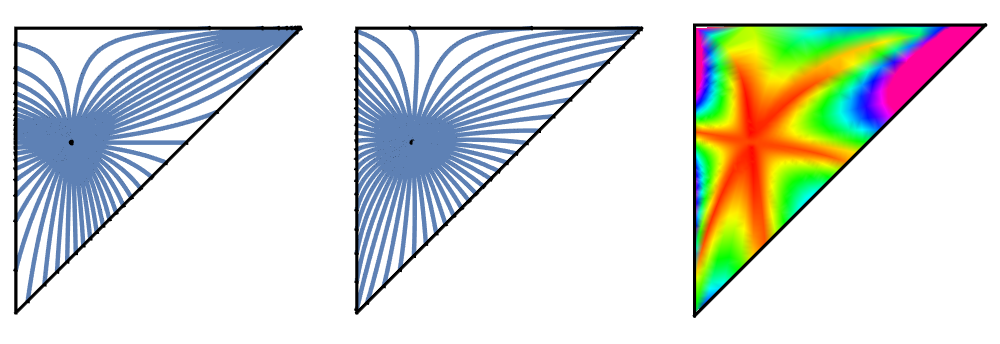}\\
			\includegraphics[scale=0.9]{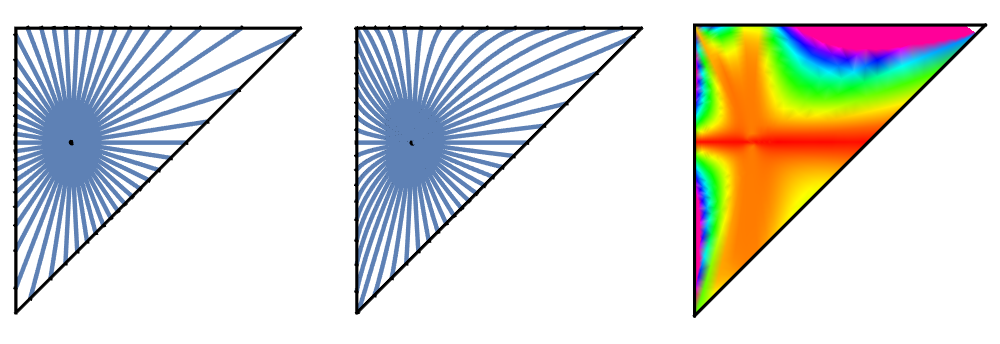}
		\end{tabular}
		\caption{Qualitatively different dynamics for $|\stateset|=2,\ |\alphset|=2$. Left column: Trajectories of the asymptotic gWF-expectation dynamics with fitness potential $-\mathrm{h}(\Pr_{\protect\rvalphra}(q_\infty) \parallel \cdot\ )$. Middle column: Geodesics with respect to the entropy rate tensor $\mathrm{g}$ ending at $q_\infty$. Right column: Colour-coded differences in path divergence from the respective initial points to $q_\infty$. The rows correspond to DFA-types $\gamma(1,0)\gamma(1,1)\gamma(2,0)\gamma(2,1)= 2121$ (1st row), $2221$ (2nd row), $1221$ (3rd row), $2211$ (4th row). Apart from DFA-type $2221$, all plots show the subdomain $q^{1,\mathsf{0}}<q^{2,\mathsf{0}}$ of $\M_\gamma$. The value of $q_\infty$ is $(0.2,0.6)$ except for DFA-type $2221$ where it is $(0.2,0.6)$ above the diagonal and $(0.6,0.2)$ under the diagonal. Interestingly, there are initial points at the boundaries from where the asymptotic expectation trajectory coincides with a geodesic (inside the red corridors), i.e. all expected parameter inference steps are asymptotically optimal.}
		\label{fig:modeldynamics}
	\end{figure}
	
	\begin{figure}
		\includegraphics[scale=0.9]{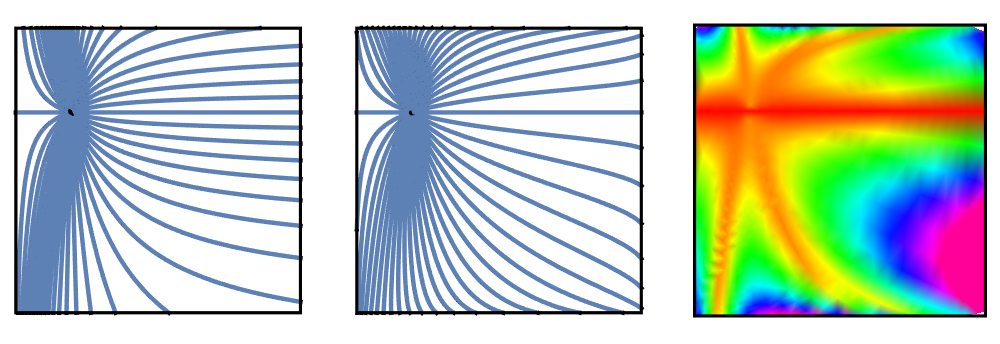}
		\caption{Asymptotic parameter inference dynamics for the Nemo processes. Left: The asymptotic gWF-expectation trajectories with fitness potential $-\mathrm{h}(\Pr_{\protect\rvalphra}(q_\infty) \parallel \cdot\ )$ and $(q_\infty^{\mathsf{0}}, q_\infty^{\mathsf{1}})=(0.2,0.6)$. Middle: Geodesics with respect to the entropy rate tensor $\mathrm{g}$ ending at $q_\infty$. Right: Colour-coded differences in path divergence from the respective initial points to $q_\infty$.}
		\label{fig:nemoflow}
	\end{figure}

	\newpage
	
	\appendix
	
		\addcontentsline{toc}{section}{Appendix}
	
	\addtocontents{toc}{\protect\setcounter{tocdepth}{-1}}

	\section{Proof of Thm \ref*{thm:relentrate}}\label{subsec:proofrelentrate}
	
	Let $\transtup, \transtup'\in\mathcal{C}$. The \emph{cross-entropy} of a random variable $\mathrm{X}: \edgesetra\to\mathscr{X}$ (with $|\mathscr{X}|<\infty$) is
	$$\Cr(q'\parallel q)(\mathrm{X}):=-\sum_{x\in\mathscr{X}}\Pr_\mathrm{X}(q')[x] \log\Pr_\mathrm{X}(q)[x].$$
	We have
	\begin{equation}\label{eq:relentbycrossent}
	\mathrm{h}(\Pr_{\rvalphra}(q')\parallel \Pr_{\rvalphra}(q))=\lim_{L\to\infty} \frac1L \big(\Cr(q'\parallel q)(\rvalph_{1:L})-\Cr(q'\parallel q')(\rvalph_{1:L})\big).
	\end{equation}
	The conditional cross-entropy $\Cr(q'\parallel q)(\mathrm{X}|\mathrm{Y})$ is defined accordingly as
	$$\Cr(q'\parallel q)(\mathrm{X}|\mathrm{Y}):=-\sum_{x,y}\Pr_{(\mathrm{X},\mathrm{Y})}(q')[x,y]\log\Pr_{\mathrm{X}|\mathrm{Y}}(q)[x | y].$$
	There is a chain rule, reading
	$$\Cr(q'\parallel q)(\mathrm{X},\mathrm{Y})=\Cr(q'\parallel q)(\mathrm{X}|\mathrm{Y})+\Cr(q'\parallel q)(\mathrm{Y}),$$
	which immediately implies
	\begin{lem}\label{lem:Cesaro}
		If $\lim_{L\to\infty}\Cr(q'\parallel q)(\rvalph_{L} | \rvalph_{1:L-1})$ exists, we have
		$$\lim_{L\to\infty}\frac1L \Cr(q'\parallel q)(\rvalph_{1:L})=\lim_{L\to\infty}\Cr(q'\parallel q)(\rvalph_{L} | \rvalph_{1:L-1}).$$
	\end{lem}
	\begin{proof}Iteration of the chain rule converts the left-hand side into a Cesàro-mean.	\end{proof}
	
	On the other hand, again using the chain rule we can rewrite $\Cr(q'\parallel q)(\rvalph_{1:L+1},\rvstate_{L+1})$ in the two different ways:
	\begin{align*}
	&\Cr(q'\parallel q)(\rvstate_{L+1} | \rvalph_{1:L+1})+ \Cr(q'\parallel q)(\rvalph_{L+1} | \rvalph_{1:L})+
	\Cr(q'\parallel q)(\rvalph_{1:L})\\
	=&\Cr(q'\parallel q)(\rvalph_{L+1} | \rvstate_{L+1},\rvalph_{1:L})+ \Cr(q'\parallel q)(\rvstate_{L+1} | \rvalph_{1:L})+
	\Cr(q'\parallel q)(\rvalph_{1:L}).
	\end{align*}
	On the second line we can additionally use the hidden Markov-property from Def. \ref{defn:markovcondition}.(ii) to rewrite the first term as $\Cr(q'\parallel q)(\rvalph_{L+1} | \rvstate_{L+1})$ and we rearrange the above equation into
	$$
	\Cr(q'\parallel q)(\rvalph_{L+1} | \rvalph_{1:L})-\Cr(q'\parallel q)(\rvalph_{L+1} | \rvstate_{L+1}) = \Cr(q'\parallel q)(\rvstate_{L+1} | \rvalph_{1:L}) - \Cr(q'\parallel q)(\rvstate_{L+1} | \rvalph_{1:L+1}).
	$$
	The right-hand side vanishes in the limit $L\to\infty$, provided that
	\begin{equation} \label{eq:synchronizationbound1}
	\lim_{L\to\infty} \Cr(q'\parallel q)(\rvstate_{L+1} | \rvalph_{1:L})=0,
	\end{equation}
	\begin{equation} \label{eq:synchronizationbound2}
	\lim_{L\to\infty} \Cr(q'\parallel q)(\rvstate_{L+1} | \rvalph_{1:L+1}) = 0.
	\end{equation}
	Assuming this for a moment to be true, we would have
	$$\lim_{L\to\infty} \Cr(q'\parallel q)(\rvalph_{L} | \rvalph_{1:L-1}) = \lim_{L\to\infty} \Cr(q'\parallel q)(\rvalph_{L} | \rvstate_{L})$$
	Together with Lem. \ref{lem:Cesaro} and stationarity:
	\begin{align*}
	\lim_{L\to\infty}\frac1L \Cr(q'\parallel q)(\rvalph_{1:L}) &= \lim_{L\to\infty} \Cr(q'\parallel q)(\rvalph_{L} | \rvstate_{L})\\
	&= \Cr(q'\parallel q)(\rvalph_{1} | \rvstate_{1})\\
	&= -\sum_{j\in\stateset} \pi_j(q') \sum_{\a\in\alphset} q'^{j,\a} \log q^{j,\a}.
	\end{align*}
	Taking into account (\ref{eq:relentbycrossent}), this would prove the theorem.
	Therefore it remains to verify (\ref{eq:synchronizationbound1}) and (\ref{eq:synchronizationbound2}).
	We begin by proving an extension of the \emph{non-exact machine synchronization theorem} of \cite{travers2011asymptotic}:
	\begin{lem}\label{lem:compsync}
		Let $\mathcal{K}$ be a compact connected subset of $\M_\gamma\setminus \mathcal{D}_\gamma$. There exist constants $0\leq b<1$, $B\geq 0$, $L_0\in\N$ and, for any $L\geq L_0$, a subset $\mathscr{W}_L\subset\alphset_{1:L}(\gamma)$ and a map
		$$\iota_{L+1}:\ \bigcup_{L\in\N} \mathscr{W}_L \to \stateset$$
		such that:
		\begin{align*}
			\sup_{q\in\mathcal{K}}\Pr(q)[\rvalph_{1:L}\notin\mathscr{W}_L]&\leq Bb^L,\\
			\forall \a_{1:L}\in\mathscr{W}_L:\quad \inf_{q\in\mathcal{K}} \Pr_{\rvstate_{L+1}|\rvalph_{1:L}}(q) [\iota_{L+1}(\a_{1:L})| \a_{1:L}] &\geq 1-b^L\geq \frac12.
		\end{align*}
	\end{lem}
	\begin{proof}		
		This lemma was shown in \cite{travers2011asymptotic} (Thm. 1.1) for the special case of single-point sets $\mathcal{K}=\{q\}$ and without the clause about $L_0$. We thus assume it to be true in that case and write $b_q, B_q, \mathscr{W}_{q,L}, \iota_{q,L+1}$ for the respective items. Furthermore, we set $L_{0,q}:=\max\{1, \lceil \log_{b_q} \frac12 \rceil +1 \}$. Due to the function $\Pr_{\rvstate_{L+1}|\rvalph_{1:L}}(\cdot)[ \iota_{q,L+1}(\a_{1:L}) | \a_{1:L}]$ being rational and hence continuous around $q$, there is an open neighbourhood $q\in\mathcal{V}(q)\subset \M_\gamma\setminus\mathcal{D}_\gamma$ such that for all $q'\in\mathcal{V}(q)$ and $L\geq L_{0,q}$, $\a_{1:L}\in \mathscr{W}_{q,L}$:
		$$\Pr_{\rvstate_{L+1}|\rvalph_{1:L}}(q')[\iota_{q,L+1}(\a_{1:L}) | \a_{1:L}]> \frac12.$$
		We consider the open cover $\mathcal{K}\subset \bigcup_{q\in\mathcal{K}}\mathcal{V}(q)$. Due to compactness there exist finitely many points $q_1,\ldots,q_n\in\mathcal{K}$ such that
		$\mathcal{K}\subset \bigcup_{m=1}^n \mathcal{V}(q_m)$. We set $L_0:=\max_{1\leq m\leq n} L_{0,q_m}$ and for $L\geq L_0$: $\mathscr{W}_L:=\bigcap_{m=1}^n \mathscr{W}_{q_m,L}$. Observe that by construction, for any  $q'\in\mathcal{V}(q_m)$ and  $\a_{1:L}\in\mathscr{W}_L$, there is exactly one $j\in\stateset$ such that
		$\Pr_{\rvstate_{L+1}|\rvalph_{1:L}}(q')[ j | \a_{1:L} ] > \frac12,$
		namely $j=\iota_{q_m,L+1}(\a_{1:L})$. If $\mathcal{V}(q_l)\cap \mathcal{V}(q_m)\neq \varnothing$, we can apply this reasoning to $q'$ in this intersection and obtain that $\iota_{q_l,L+1}(\a_{1:L})=\iota_{q_m,L+1}(\a_{1:L})$.
		Since the finite open cover $\bigcup_m \mathcal{V}(q_m)$ is connected, we can conclude by transitivity that the internal state $\iota_{q,L+1}(\a_{1:L})$ does not depend on $q\in\mathcal{K}$ for any $\a_{1:L}\in \mathscr{W}_L$. This defines the function $\iota_{L+1}:  \mathscr{W}_L\to\stateset$. The validity of the two bounds from the lemma can easily be verified with $$b:=\max_{1\leq m\leq n} b_{q_m}\quad \textnormal{and}\quad B:=\sum_{m=1}^n B_{q_m}.$$
	\end{proof}
	
	We shall now use Lemma \ref{lem:compsync} to compute the limit (\ref{eq:synchronizationbound1}). Choose some compact connected set $\mathcal{K}\subset \mathcal{C}$ with $q,q'\in\mathcal{K}$ and denote by $b, B, L_0, \mathscr{W}_L, \iota_{L+1}$ the respective items from the lemma. Moreover, there is a constant $0\leq M <\infty$ such that we have the bound
	\begin{align*}
	\max_{j,\a_{1:L}}\left(-\log \left(\Pr_{\rvstate_{L+1}|\rvalph_{1:L}}(q)[ j | \a_{1:L} ]\right)\right) &= - \min_{j,\a_{1:L}}\log\frac{\pi(q)\trans{\a_{1:L}}\ket{j}}{\pi(q)\trans{\a_{1:L}}\one}\\
	&\leq -\log\left(\min_{j\in\stateset}\pi_j(q)\left(\min_{(j,\a)\in\edgeset(\gamma)} q^{j,a}\right)^L\right)\\
	&\leq L M,
	\end{align*}
	for any $q\in\mathcal{K}$. We now apply this bound together with Lem. \ref{lem:compsync}. For better readability we omit the subscripts in $\Pr_{\mathrm{X|\mathrm{Y}}}$:
	\begin{align*}
	&\lim_{L\to\infty} \Cr(q'\parallel q)(\rvstate_{L+1} | \rvalph_{1:L}) \leq\\
	&\leq -\lim_{L\to\infty} \sum_{\a_{1:L}\in\mathscr{W}_L} \Pr(q')[ \a_{1:L}] \sum_{i_{L+1}\in\stateset} \Pr(q')[ i_{L+1} | \a_{1:L} ]\ \log\Pr(q)[ i_{L+1} | \a_{1:L} ]\\
	&\quad +\underbrace{\lim_{L\to\infty} L M B b^L}_{=0}.
	\end{align*}
	The Taylor estimate $-\log(1-x)= x+\mathcal{O}(x^2)\leq M_0 x$, holding for some $M_0>0$ and $\frac12\leq 1-x \leq 1$, yields
	$-\log \Pr(q) [\iota_{L+1}(\a_{1:L}) |  \a_{1:L}] \leq -\log(1-b^L) \leq M_0 b^L,$
	for any $\a_{1:L}\in \mathscr{W}_L$ and we can complete the proof of (\ref{eq:synchronizationbound1}):
	\begin{align*}
	\lim_{L\to\infty} \Cr(q'\parallel q)(\rvstate_{L+1} | \rvalph_{1:L}) &\leq \lim_{L\to\infty} \sum_{\a_{1:L}\in\mathscr{W}_L} \Pr(q')[\a_{1:L}]\ \Pr(q')[ \iota_{L+1}(\a_{1:L}) | \a_{1:L} ]\ M_0 b^L\\
	&\quad +\lim_{L\to\infty} \sum_{\a_{1:L}\in\mathscr{W}_L} \Pr(q')[\a_{1:L}] \sum_{i_{L+1}\neq \iota_{L+1}(\a_{1:L})} \Pr(q')[ i_{L+1} | \a_{1:L} ]\ L M\\
	&\leq \lim_{L\to\infty} M_0 b^L + b^L L M\\
	&=0.
	\end{align*}
	
	To compute the limit in (\ref{eq:synchronizationbound2}), we need to make an additional argument.
	\begin{lem}\label{lem:exactsync}
		Let $\mathcal{K}$ be a compact connected subset of $\M_\gamma$. There exist constants $0\leq b_0<1$ and $B_0>0$ and, for any $L\in\N$, a subset $\mathscr{W}_{0,L}\subset \alphset_{1:L}(\gamma)$ such that for any $q\in\mathcal{K}$:
		$$\Pr(q)[ \rvalph_{1:L}\notin \mathscr{W}_{0,L} ] \leq B_0 b_0^L$$
		and
		$$
		\Pr_{\rvstate_L|\rvalph_{1:L}}(q)[ j | a_{1:L} ] = \Pr_{\rvstate_{L+1}|\rvalph_{1:L}}(q)[ \gamma(j,\a_L) | a_{1:L} ],
		$$
		for any $j\in\stateset$ and $a_{1:L}\in\mathscr{W}_{0,L}$.
	\end{lem}
	\begin{proof}
		Observe that, for any $\a_{1:L}\in\alphset_{1:L}(\gamma)$, the internal state subset $$\supp\big(\Pr_{\rvstate_{L+1}|\rvalph_{1:L}}(q)[\ \cdot\ | \a_{1:L}]\big)\subset \stateset$$
		does not depend on the point $q\in\M_\gamma$ but only on the DFA-type $\gamma$. Due to unifilarity, the positive integers
		$$r(\a_{1:L}):=\mathrm{card}\left(\supp\big(\Pr_{\rvstate_{L+1}|\rvalph_{1:L}}(q)[\ \cdot\ | \a_{1:L}]\big)\right),$$
		are non-increasing in the sense that $r(\a_{1:L+1})\leq r(\a_{1:L})$ for any $\a_{1:L+1}\in \alphset_{1:L+1}(\gamma)$. Denote
		$$r_\gamma:=\min\set{r(\a_{1:L})}{L\in\N,\ \a_{1:L}\in\alphset_{1:L}(\gamma)}.$$
		Let $\mathsf{w}\in\alphset_{1:L_1}(\gamma,k)$ be some word with $r(\mathsf{w})=r_\gamma$ and, for $L>L_1$, set
		$$\mathscr{W}_{0,L}:=\set{\a_{1:L}\in\alphset_{1:L}(\gamma)}{\mathsf{w} \textnormal{ is a subword of } \a_{1:L-1}}.$$
		Because $\gamma$ is strongly-connected, there is for any $j\in\stateset$ a word $\mathsf{w}_j\in\alphset_{1:L_j}(\gamma,j)$ having $\mathsf{w}$ as a subword. We set
		$$p:=\inf_{q\in\mathcal{K}}\min_{j\in\stateset}\Pr_{\rvalph_{1:L_j} |\rvstate_1}(q)[\mathsf{w}_j|j].$$
		This number is positive because $\Pr_{\rvalph_{1:L_j} |\rvstate_1}(\cdot)[\mathsf{w}_j|j]$ are positive continuous functions on the compact set $\mathcal{K}$. Like in the proof of Thm. 1 in \cite{travers2011exact}, it can be shown that for any $q\in\mathcal{K}$: $\Pr(q)[ \rvalph_{1:L}\notin \mathscr{W}_{0,L} ] \leq B_0 b_0^L$ for $B_0=\frac{1}{1-p}$, $b_0=(1-p)^{1/\max_{j} L_j}<1$.
		
		Furthermore, let $(i_{1:L},\a_{1:L})\in\edgeset_{1:L}(\gamma)$ such that $\a_{1:L}\in \mathscr{W}_{0,L}$. Then we have
		$$\forall j\in \supp\big(\Pr_{\rvstate_{L}|\rvalph_{1:L-1}}(q)[\ \cdot\ | \a_{1:L-1}]\big):\ \gamma(j,\a_L)=\gamma(i_L,\a_L)\ \Leftrightarrow \ j=i_L,$$
		because otherwise $r(\a_{1:L})<r(\a_{1:L-1})$. This implies the remaining statement from the lemma.
	\end{proof}
	
	We will now finish the proof of Thm. \ref{thm:relentrate} by computing the limit (\ref{eq:synchronizationbound2}). Take $b_0, B_0, \mathscr{W}_{0,L}$ to be the items from Lem. \ref{lem:exactsync} with respect to the previously chosen $\mathcal{K}$. Set $\mathscr{W}'_{L}:=\mathscr{W}_L\cap\mathscr{W}_{0,L}$ and immediately observe that, for $L\geq L_0$:
	\begin{equation}\label{eq:smallmeasureset}
	\Pr(q')[ \rvalph_{1:L}\notin \mathscr{W}'_{L} ]\leq B b^L + B_0 b_0^L\leq B' b'^L,
	\end{equation}
	for the constants $B':= B+B_0>0$ and $0\leq b':=\max\{b,b_0\} <1$. We set
	$$\mathrm{Cr}(q'\parallel q)(\mathrm{X}|y):=-\sum_{x\in\mathscr{X}} \Pr_{\mathrm{X}|\mathrm{Y}}(q')[x | y]\log\Pr_{\mathrm{X}|\mathrm{Y}}(q)[x | y]$$
	and note that Lem. \ref{lem:exactsync} implies, for any $\a_{1:L+1}\in\mathscr{W}'_{L+1}$:
	\begin{equation}\label{eq:crossentsync}
	\mathrm{Cr}(q'\parallel q)(\rvstate_{L+1}|\a_{1:L+1})=\mathrm{Cr}(q'\parallel q)(\rvstate_{L+2}|\a_{1:L+1}).
	\end{equation}
	We will now perform the final chain of estimates to prove (\ref{eq:synchronizationbound2}) and again omit the subscripts in $\Pr_{\mathrm{X|\mathrm{Y}}}$:
	\begin{align*}
	&\lim_{L\to\infty} \Cr(q'\parallel q)(\rvstate_{L+1} | \rvalph_{1:L+1}) =\\
	&= -\lim_{L\to\infty} \sum_{\a_{1:L+1}\in\alphset_{1:L+1}(\gamma)} \Pr(q')[\a_{1:L+1}]\ \Cr(q'\parallel q)(\rvstate_{L+1} | \a_{1:L+1})\\
	&\overset{(\ref{eq:smallmeasureset})}{\leq} -\lim_{L\to\infty}\sum_{\a_{1:L+1}\in \mathscr{W}'_{L+1}} \Pr(q')[\a_{1:L+1}]\  \Cr(q'\parallel q)(\rvstate_{L+1} | \a_{1:L+1})\\
	&\hspace{2cm} +\underbrace{\lim_{L\to\infty} B' b'^{L+1} (L+1) M}_{=0}\\
	&\overset{(\ref{eq:crossentsync})}{=} -\lim_{L\to\infty}\sum_{\a_{1:L+1}\in\mathscr{W}'_{L+1}} \Pr(q')[\a_{1:L+1}]\  \Cr(q'\parallel q)(\rvstate_{L+2} | \a_{1:L+1})\\
	&\leq -\lim_{L\to\infty}\sum_{\a_{1:L+1}\in\mathscr{W}_{L+1}} \Pr(q')[\a_{1:L+1}]\  \Cr(q'\parallel q)(\rvstate_{L+2} | \a_{1:L+1})\\
	&\leq \lim_{L\to\infty} M_0 b^{L+1} + b^{L+1} (L+1) M\\
	&=0.
	\end{align*}
	where the last estimate is achieved in the same way as in the end of the computation of (\ref{eq:synchronizationbound1}).
	\qed
	
	\newpage
	
	\section{Proof of Prop. \ref*{prop:hessianmetricjet}}\label{subsec:proofhessianmetricjet}
	Since $\mathrm{h}_\gamma(\cdot\parallel q)$ is smooth and non-negative on a neighbourhood of $q$ with isolated minimum at $q$, the metric tensor $\mathrm{g}(q)$ is positive-definite and hence a Riemannian metric.
	
	\textbf{(i)} It has to be shown that for an arbitrary smooth curve $c: [-\varepsilon,\varepsilon]\to \M_\gamma$ with $c(0)=q$, we have
	\begin{equation}\label{eq:jets}
		\frac{\d^m}{\d t^m}\Big|_{t=0} \mathrm{h}_\gamma(c(t)\parallel q)= \frac{\d^m}{\d t^m}\Big|_{t=0} \mathrm{E}(c(t)\parallel q)\ ,\quad \textnormal{for } m=0,1,2.
	\end{equation}
	To this end, observe first that the left-hand side is zero for $m=0,1$, and for $m=2$ it equals $\mathrm{g}(q)\cdot\dot{c}(0)^{\otimes 2}$ by definition of $\mathrm{g}$. For the right-hand side, we consider the exponential map $\mathrm{Exp}_q: T_q \M_\gamma \to \M_\gamma$. In a sufficiently small neighbourhood of $q$, the infimum in the definition of $\mathrm{E}(\cdot\parallel q)$ is achieved precisely at the geodesics through $q$ i.e. the curves $t\mapsto \mathrm{Exp}_q(tv)$ for $v\in T_q \M_\gamma$ and $t\in\mathbb{R}$ such that $||tv||_\mathrm{g}$ is sufficiently small. On the other hand, as the exponential map is a diffeomorphism of a neighbourhood of $0$ onto its image, there are smooth functions $r: [-\varepsilon,\varepsilon]\to \R$, with $r(0)=0$, $\dot{r}(0)=1$, and $X: [-\varepsilon,\varepsilon]\to T_q\M_\gamma$, with $X(0)=\dot{c}(0)$, such that $c(t)=\mathrm{Exp}_q(r(t) X(t))$. We can now write
	\begin{align*}
		\mathrm{E}(c(t)\parallel q)&= \frac12 \int_{s=0}^{1} \mathrm{g}(\mathrm{Exp}_q(s r(t) X(t)))\cdot\left(T\mathrm{Exp}_{q}(s r(t) X(t)) (r(t) X(t))\right)^{\otimes 2} \d s\\
		&= \frac{r(t)^2}2 \int_{s=0}^{1} \mathrm{g}(\mathrm{Exp}_q(s r(t) X(t)))\cdot\left(T\mathrm{Exp}_{q}(s r(t) X(t))\ X(t)\right)^{\otimes 2} \d s,
	\end{align*}
	where $T\mathrm{Exp}_{q}(v):\ T_v T_q\M_\gamma \to T_{\mathrm{Exp}_q(v)}\M_\gamma$ is the tangent map. Since $r(0)=0$ we can conclude $\frac{\d^m}{\d t^m}\Big|_{t=0} \mathrm{E}(c(t)\parallel q)=0$ for $m=0,1$ and since $\dot{r}(0)=1$:
	\begin{align*}
		\frac{\d^2}{\d t^2}\Big|_{t=0} \mathrm{E}(c(t)\parallel q)&= \int_{s=0}^{1} \mathrm{g}(\mathrm{Exp}_q(0))\cdot\left(T\mathrm{Exp}_{q}(0)\ X(0)\right)^{\otimes 2} \d s\\
		&=  \mathrm{g}(q)\cdot\left(X(0)\right)^{\otimes 2} \int_{s=0}^{1} \d s\\
		&= \mathrm{g}(q)\cdot\dot{c}(0)^{\otimes 2}.
	\end{align*}
	
	\textbf{(ii)} Item (i) together with Taylor's theorem implies that for $|t_l-t_{l-1}|\leq \delta$ small enough, there is $M>0$ such that $\big|\mathrm{h}_\gamma(c(t_l) \parallel c(t_{l-1})) \ -\ \mathrm{E}(c(t_l)\parallel c(t_{l-1}))\big|\leq M (t_l-t_{l-1})^3$ and therefore
	\begin{align*}
		&\lim_{\delta\to 0}\sup_{(t_l)\in \mathcal{Z}_\delta} \left|\ \sum_{l=1}^n \mathrm{h}_\gamma(c(t_l) \parallel c(t_{l-1})) \ -\ \mathrm{E}[c]\ \right|\leq \\
		&\leq \lim_{\delta\to 0}\sup_{(t_l)\in \mathcal{Z}_\delta} \sum_{l=1}^n \left|\  \mathrm{h}_\gamma(c(t_l) \parallel c(t_{l-1})) \ -\ \mathrm{E}(c(t_l)\parallel c(t_{l-1}))\ \right| + \left|\ \mathrm{E}(c(t_l)\parallel c(t_{l-1})) \ -\ \mathrm{E}[c|_{[t_{l-1},t_l]}]\ \right|\\	
		&\leq \lim_{\delta\to 0}\sup_{(t_l)\in \mathcal{Z}_\delta} \sum_{l=1}^n M (t_l-t_{l-1})^3 + \sum_{l=1}^n \left|\ \mathrm{E}(c(t_l)\parallel c(t_{l-1})) \ -\ \mathrm{E}[c|_{[t_{l-1},t_l]}]\ \right|\\
		&= \lim_{\delta\to 0}\sup_{(t_l)\in \mathcal{Z}_\delta} \sum_{l=1}^n \left|\ \mathrm{E}(c(t_l)\parallel c(t_{l-1})) \ -\ \mathrm{E}[c|_{[t_{l-1},t_l]}]\ \right|.
	\end{align*}
	Denoting by $c_{l-1}: [t_{l-1},t_l]\to\M_\gamma$ the unique constant-speed geodesic running from $c(t_{l-1})$ to $c(t_l)$:
	\begin{align*}
		&\lim_{\delta\to 0}\sup_{(t_l)\in \mathcal{Z}_\delta} \left|\ \sum_{l=1}^n \mathrm{h}_\gamma(c(t_l) \parallel c(t_{l-1})) \ -\ \mathrm{E}[c]\ \right|\\
		&\leq \lim_{\delta\to 0}\sup_{(t_l)\in \mathcal{Z}_\delta} \sum_{l=1}^n \frac12 \int_{s=t_{l-1}}^{t_l} \left|\ \mathrm{g}(c_{l-1}(s))\cdot\dot{c}_{l-1}(s)^{\otimes 2} - \mathrm{g}(c(s))\cdot\dot{c}(s)^{\otimes 2} \ \right| \d s\\
		&\leq \lim_{\delta\to 0}\sup_{(t_l)\in \mathcal{Z}_\delta} \sup_{\stack{1\leq l\leq n}{s\in [t_{l-1},t_l]}}\frac12 \left|\ \mathrm{g}(c_{l-1}(s))\cdot\dot{c}_{l-1}(s)^{\otimes 2} - \mathrm{g}(c(s))\cdot\dot{c}(s)^{\otimes 2} \ \right|.
	\end{align*}
	The expression in absolute values depends continuously on $s\in [t_{l-1}, t_l]$. Therefore, all we need to show is that it vanishes in the limit $s=t_l\to t_{l-1}$. We employ the exponential map around $c(t_{l-1})$ and write $c(s)=\mathrm{Exp}_{c(t_{l-1})}(r(s)X(s))$ with smooth functions $r,X$ defined on a neighbourhood of $t_{l-1}$ satisfying $r(t_{l-1})=0,\ \dot{r}(t_{l-1})=1$ and $X(t_{l-1})=\dot{c}(t_{l-1})$. We can write
	$$c_{l-1}(s)=\mathrm{Exp}_{c(t_{l-1})}\left(\frac{s-t_{l-1}}{t_l-t_{l-1}}r(t_l)X(t_l)\right)$$
	and conclude the proof with
	\begin{align*}
		\lim_{t_l\to t_{l-1}} \mathrm{g}(c_{l-1}(t_{l}))\cdot\dot{c}_{l-1}(t_{l})^{\otimes 2} 
		&= \mathrm{g}(c(t_{l-1}))\cdot \left(\lim_{t_l\to t_{l-1}} \frac{r(t_l)}{t_l-t_{l-1}}X(t_l) \right)^{\otimes 2}\\
		&= \mathrm{g}(c(t_{l-1}))\cdot \left(\dot{r}(t_{l-1})X(t_{l-1}) \right)^{\otimes 2}\\
		&= \mathrm{g}(c(t_{l-1}))\cdot \dot{c}(t_{l-1})^{\otimes 2}\\
	\end{align*}
	\qed
	
	\newpage
	
	\section{Proof of Thm. \ref*{thm:largeconsistentdev}}\label{subsec:prooflargedev}
	
	According to the arguments in the proof outline, we have for any $x\in\hat{\pi}(\edgeset_{1:L}(\gamma))$:
	\begin{equation*}
		\hat{\pi}_*\Pr_{\rvedge_{1:L}}(q)[x]= \left(\sum_{e_{1:L}\in\hat{\pi}^{-1}(x)}\pi_{\rvstate(e_1)}(q)\right) \prod_{(j,\a)\in\edgeset(\gamma)} \e^{Lx^{j,\a}\log q^{j,\a}}.
	\end{equation*}
	We would like to estimate the sum in brackets using graph combinatorics of the empirical type $x$. Observe that every Euler path of $x$ yields a string $e_{1:L}\in\edgeset_{1:L}(\gamma)$ but this assignment is not injective. Indeed, Euler paths which are obtained from one another by exchanging parallel edges labelled with the same letter lead to the same string $e_{1:L}$. Taking this multiplicity into account, we have
	$$\mathrm{card}(\hat{\pi}^{-1}(x))=\frac{\mathrm{ET}(x)}{\prod_{j,\a} (Lx^{j,\a})!},$$
	where $\mathrm{ET}(x)$ denotes the number of Euler paths in the empirical type $x$. The fact that $x$ admits an Euler path implies that all except for two vertices of $x$ have coinciding in- and out-degrees and the remaining two vertices either also have coinciding in- and out-degrees or one of them has in-degree by 1 lower than out-degree and the other one has out-degree by 1 lower than in-degree. In the second case, the initial vertex $\mathrm{init}(x)$ and the terminal vertex $\mathrm{term}(x)$ of any Euler path are uniquely determined by $x$ and we have
	$$\sum_{e_{1:L}\in\hat{\pi}^{-1}(x)}\pi_{\rvstate(e_1)}(q)=\frac{\mathrm{ET}(x,\mathrm{term}(x))}{\prod_{j,\a} (Lx^{j,\a})!} \pi_{\mathrm{init}(x)}(q),$$
	where $\mathrm{ET}(x,\mathrm{term}(x))=\mathrm{ET}(x)$ denotes the number of Euler paths in $x$ with terminal vertex $\mathrm{term}(x)$. In the other case, every Euler path is a circuit. In this case we set the values $\mathrm{init}(x)=\mathrm{term}(x)$ to some arbitrary internal state and obtain through cyclical permutation of Euler-circuits:
	$$\sum_{e_{1:L}\in\hat{\pi}^{-1}(x)}\pi_{\rvstate(e_1)}(q)=\frac{\mathrm{ET}(x,\mathrm{term}(x))}{\prod_{j,\a} (Lx^{j,\a})!}.$$
	We set
	\begin{equation}\label{eq:weightdef}
		w(x):=\begin{cases}
		1 &, \textnormal{if the Euler paths of $x$ are circuits,}\\
		\pi_{\mathrm{init}(x)}(q) &, \textnormal{otherwise},
		\end{cases}	
	\end{equation}
	and show
	\begin{lem}\label{lem:preciselargedeviation}
		Let $\hat{G}\subset \R_+^{\edgeset(\gamma)}$ be any compact subset. There is $L_0\in\N$ such that for any $L\geq L_0$ and any $x\in \hat{\pi}(\edgeset_{1:L}(\gamma))\cap\hat{G}$, we can write
		$$\hat{\pi}_*\Pr_{\rvedge_{1:L}}(q)[x]= \frac1L w(x)\hat{\kappa}_0(x)\hat{\kappa}_{1,L}(x)\e^{-L\mathrm{h}_\gamma(\psi(x)\parallel q)+\sum_{j\in\stateset}r_{L}^j(x)\sum_{\a \in\alphset(\gamma,j)} \psi(x)^{j,\a}\log\frac{\psi(x)^{j,\a}}{q^{j,\a}}},$$
		for a continuous function $\hat{\kappa}_0: \R_+^{\edgeset(\gamma)}\to\R_+$ and functions $\hat{\kappa}_{1,L}(x)$ and
		$$r_{L}^j(x)=L(x^j-\pi_j(\psi(x)))\quad ,\ x^j:=\sum_{\a \in\alphset(\gamma,j)} x^{j,\a},$$ on $\hat{\pi}(\edgeset_{1:L}(\gamma))\cap\hat{G}$ which are all uniformly bounded wrt. $L$ and such that $$\lim_{L\to\infty}\sup_{x\in \hat{\pi}(\edgeset_{1:L}(\gamma))\cap \hat{G}} |\hat{\kappa}_{1,L}(x)-1|= 0.$$
	\end{lem}
	\begin{proof}
		After the previous comments, we know that
		\begin{equation}\label{eq:empiricalhatdistribution}
			\hat{\pi}_*\Pr_{\rvedge_{1:L}}(q)[x]= w(x) \frac{\mathrm{ET}(x,\mathrm{term}(x))}{\prod_{j,\a} (Lx^{j,\a})!}\  \e^{L \sum_{j,\a} x^{j,\a}\log q^{j,\a}}.
		\end{equation}
		The number $\mathrm{ET}(x,\mathrm{term}(x))$ can be 
		computed by a standard graph-combinatorial argument, commonly referred to as the BEST-theorem:
		$$\mathrm{ET}(x,\mathrm{term}(x))= \mathrm{Arb}(x,\mathrm{term}(x))\prod_{j\in\stateset}(\mathrm{outdeg}_{x}(j)-1)!,$$
		where $\mathrm{Arb}(x,\mathrm{term}(x))$ is the number of arborescences with root $\mathrm{term}(x)$. It can be computed as a principal minor  by Kirchhoff's matrix-tree theorem ($[\cdot]_{jj}$ denotes the $j$-th principal minor):
		\begin{align*}
		\mathrm{Arb}(x,\mathrm{term}(x))
		&=\left[\sum_{j\in\stateset} \mathrm{indeg}_{x}(j)\cdot \ket{j}\bra{j}- \sum_{j,k\in\stateset}  \mathrm{mult}_{x}(j\to k)\cdot \ket{j}\bra{k} \right]_{\mathrm{term}(x) \mathrm{term}(x)}\\
		&= \left[\sum_{j\in\stateset} \sum_{(i,\a): \gamma(i,\a)=j} L x^{i,\a}\ket{j}\bra{j}- \sum_{jk\in\stateset_{1:2}} \sum_{\a: \gamma(j,\a)=k}L x^{j,\a} \ket{j}\bra{k} \right]_{\mathrm{term}(x) \mathrm{term}(x)}\\
		&= L^{|\stateset|-1}\left[\sum_{j\in\stateset} \sum_{(i,\a): \gamma(i,\a)=j}  x^{i,\a}\ket{j}\bra{j}- \sum_{jk\in\stateset_{1:2}} \sum_{\a: \gamma(j,\a)=k} x^{j,\a} \ket{j}\bra{k} \right]_{\mathrm{term}(x) \mathrm{term}(x)}.
		\end{align*}
		Furthermore we use $Lx^j=\mathrm{outdeg}_x(j)$ and Stirling's approximation for the factorial, namely $(Lx^j)!=\sqrt{2\pi} \e^{-Lx^j}\sqrt{Lx^j}(Lx^j)^{Lx^j} (1+\mathcal{O}((L x^j)^{-1}))$ with $0\leq \mathcal{O}((L x^j)^{-1})\leq \frac{\e}{\sqrt{2\pi}}-1$, to compute:
		\begin{align*}
		\frac{\mathrm{ET}(x,\mathrm{term}(x))}{\prod_{j,\a} (Lx^{j,\a})!} &= \frac{\mathrm{Arb}(x,\mathrm{term}(x))}{L^{|\stateset|}\prod_{j} x^j}\frac{\prod_{j} (Lx^j)!}{\prod_{j,\a} (Lx^{j,\a})!}\\
		&=\frac{\mathrm{Arb}(x,\mathrm{term}(x))}{L^{|\stateset|}\prod_{j} x^j} \cdot\frac{\prod_j \mathcal{O}((L x^j)^{-1})}{\prod_{j,\a} \mathcal{O}((L x^{j,\a})^{-1})}\cdot \frac{\prod_{j} \sqrt{ x^j}(Lx^j)^{Lx^j}}{\prod_{j,\a} \sqrt{ x^{j,\a}} (Lx^{j,\a})^{Lx^{j,\a}}}\\
		&=\frac{\mathrm{Arb}(x,\mathrm{term}(x))}{L^{|\stateset|}\sqrt{\prod_{j,\a} x^j x^{j,\a}}}  \cdot \frac{\prod_j \mathcal{O}((L x^j)^{-1})}{\prod_{j,\a} \mathcal{O}((L x^{j,\a})^{-1})}\cdot
		\prod_{j,\a} \left(\frac{x^{j,\a}}{x^j}\right)^{-Lx^{j,\a}}\\
		&=\underbrace{\frac{\mathrm{Arb}(x,\mathrm{term}(x))}{L^{|\stateset|}\sqrt{\prod_{j,\a} x^j x^{j,\a}}}}_{=:\frac1L \hat{\kappa}_0(x)}  \cdot \underbrace{\frac{\prod_j \mathcal{O}((L x^j)^{-1})}{\prod_{j,\a} \mathcal{O}((L x^{j,\a})^{-1})}}_{=:\hat{\kappa}_{1,L}(x)}\ \e^{-L\sum_{j,\a} x^j \psi(x)^{j,\a}\log \psi(x)^{j,\a}}.
		\end{align*}
		The functions $\hat{\kappa}_0$ and $\hat{\kappa}_{1,L}$ satisfy the requirements from the lemma due to the previous expression of $\mathrm{Arb}(x,\mathrm{term}(x))$ and the error bounds resp. error asymptotics of Stirling's approximation.
		Equation (\ref*{eq:empiricalhatdistribution}) now turns into
		\begin{align*}
			\hat{\pi}_*\Pr_{\rvedge_{1:L}}(q)[x] &= \frac1L w(x)\hat{\kappa}_0(x)\hat{\kappa}_{1,L}(x)\ \e^{-L\sum_{j,\a} x^j \psi(x)^{j,\a}\log \frac{\psi(x)^{j,\a}}{q^{j,\a}}}\\
			&= \frac1L w(x)\hat{\kappa}_0(x)\hat{\kappa}_{1,L}(x)\ \e^{-L\mathrm{h}_{\gamma}(\psi(x)\parallel q)+\sum_{j,\a} L(\pi^j(\psi(x))-x^j) \psi(x)^{j,\a}\log \frac{\psi(x)^{j,\a}}{q^{j,\a}}}
		\end{align*}
		To get the expression from the lemma, we set $r^j_L(x):=L(\pi^j(\psi(x))-x^j)$. It remains to be shown that these values are uniformly bounded for $x\in\hat{\pi}(\edgeset_{1:L}(\gamma))\cap\hat{G}$ and $L\in\N$. For this we first show the following useful
		\begin{lem}\label{lem:emptypemap}
			$\psi$ has a right-inverse $\varphi$, i.e. $\psi\circ\varphi=\mathrm{id}_{\M_\gamma}$, given by
			$$\varphi(q):= (\pi_j(q) q^{j,\a}),$$
			whose image is the semi-affine subspace
			$$\hat{\M}_\gamma:=\Set{x\in\R_+^{\edgeset(\gamma)}}{\sum_{e\in\edgeset(\gamma)} x^e=1;\ \forall j\in\stateset:\ \eta_j\cdot x= 0}\quad \textnormal{with}\quad \eta_j=\sum_{e:\ \rvstate(e)=j} \bra{e} -\sum_{e:\ \gamma(e)=j} \bra{e}.$$
		\end{lem}
		\begin{proof}
			The statement about the right-inverse is obvious. To compute the image of $\varphi$, recall that $\pi(q)$ is a stationary mixed state of $\ubar{q}$, implying that $\varphi(q)$ lies in the standard simplex in $\R^{\edgeset(\gamma)}$ and that we have:
			$$\sum_{j\in\stateset}\sum_{\gamma(e)=j}\pi_{\rvstate(e)}(q)q^e\ \bra{j}= \pi(q)\cdot\ubar{q}= \pi(q)\cdot\One =\pi(q)\cdot\sum_{j\in\stateset}\left(\sum_{e:\ \rvstate(e)=j} q^e\right)\ket{j}\bra{j}=\sum_{j \in\stateset}\sum_{\a\in\alphset(\gamma,j)} \pi_j(q) q^{j,\a}\ \bra{j}.$$
			Equating the coefficients of $\bra{j}$ yields exactly the remaining equations used to define $\hat{\M}_\gamma$.
		\end{proof}
		Continuing in the proof of Lem. \ref*{lem:preciselargedeviation}, we now observe that for any empirical type $x=\hat{\pi}(e_{1:L})\in\R_+^{\edgeset(\gamma)}$ we can artificially produce another empirical type $x_0=\hat{\pi}(e_{1:L+m})\in\R_+^{\edgeset(\gamma)}$ by adding a circuit-free path of $m<|\stateset|$ edges such that $\gamma(e_{1:L+m})=\rvstate(e_1)$. This means that $x_0\in\hat{\M}_\gamma$ and therefore, according to Lem. \ref*{lem:emptypemap}, $x_0^j=\varphi(\psi(x_0))^j$. Moreover we have $|x^j-x_0^j|\leq C_1\cdot 1/L$ for some constant $C_1\geq 0$ (independent of $x$). On the other hand  $\varphi\circ\psi$ is continuously differentiable on $\hat{G}$ and therefore we also have $|\varphi(\psi(x_0))^j-\varphi(\psi(x))^j|\leq C_2\cdot 1/L$ for some constant $C_2\geq 0$ only depending on $\hat{G}$ (and not on $x\in \hat{G}$ or $L\in\N$). Using this we estimate for $x\in\hat{G}$:
		\begin{align*}
		|r_L^j(x)| &= L |x^j-\pi_j(\psi(x))|\\
		&\leq L |x^j-x^j_0|+ L |x^j_0-\pi_j(\psi(x))|\\
		&= L |x^j-x^j_0|+ L |\varphi(\psi(x_0))^j-\varphi(\psi(x))^j|\\
		&\leq C_1+C_2.
		\end{align*}
		This completes the proof of Lem. \ref*{lem:preciselargedeviation}.
	\end{proof}
	
	We continue with the proof of the theorem. Since $\lim_{L\to\infty} q_L\in\M_\gamma$, there is a compact set $\hat{G}\subset\R_+^{\edgeset(\gamma)}$ and $L_0\in\N$ such that, for $L\geq L_0$, $\psi^{-1}(q_L)\subset\hat{G}$. We can thus apply the representation from Lem. \ref{lem:preciselargedeviation} and set
	$$\hat{\kappa}_L(x):=\frac1L w(x)\hat{\kappa}_0(x)\hat{\kappa}_{1,L}(x)\e^{\sum_{j\in\stateset}r_{L}^j(x)\sum_{\a \in\alphset(\gamma,j)} \psi(x)^{j,\a}\log\frac{\psi(x)^{j,\a}}{q^{j,\a}}}.$$ The lemma implies that there exists $C_3> 0$ such that, for $L\geq L_0$ and $x\in \hat{\pi}(\edgeset_{1:L}(\gamma))\cap\hat{G}$,
	$$\frac{1}{C_3}\frac{1}{L} \leq  \hat{\kappa}_L(x)\leq C_3\frac{1}{L}.$$
	We write for $L\geq L_0$:
	\begin{align*}
		\Pr_{\mathrm{Q}^L}(q)[q_L]&=\sum_{x\in\psi^{-1}(q_L)\cap\hat{\pi}(\edgeset_{1:L}(\gamma))} \hat{\pi}_*\Pr_{\rvedge_{1:L}}(q)[x]\\
		&= \underbrace{\sum_{x\in\psi^{-1}(q_L)\cap\hat{\pi}(\edgeset_{1:L}(\gamma))} \hat{\kappa}_L(x)}_{=:\kappa_L(q_L)}\ \e^{-L \mathrm{h}_\gamma(q_L\parallel q)}.
	\end{align*}
	The proof is concluded by estimating
	\begin{align*}
		0&=\lim_{L\to\infty}\frac1L \log \left(\frac{1}{C_3}\frac{1}{L}\right)\\
		&\leq\lim_{L\to\infty}\frac1L \log \inf_{x\in\psi^{-1}(q_L)\cap\hat{\pi}(\edgeset_{1:L}(\gamma))}\hat{\kappa}_L(x)\\
		&\leq\lim_{L\to\infty}\frac1L \log \kappa_L(q_L)\\ 
		&\leq \lim_{L\to\infty}\frac1L \log\left(\mathrm{card}(\hat{\pi}(\edgeset_{1:L}(\gamma)))\sup_{x\in\psi^{-1}(q_L)\cap\hat{\pi}(\edgeset_{1:L}(\gamma))} \hat{\kappa}_L(x)\right)\\
		&\leq \lim_{L\to\infty}\frac1L \log\left(L^{|\edgeset(\gamma)|}C_3\frac{1}{L}\right)\\
		&=0.
	\end{align*}
	\qed
	
	\newpage
	
	\section{Proof of Thm. \ref*{thm:centrallimitthm}}\label{subsec:proofcentrallimitthm}
	
	We fix $q\in\M_\gamma$ and define the measures
	$\mu_L, \mu$ as in the proof outline. In order to show weak convergence of $(\mu_L)$ to $\mu$ we would like to employ the Portmanteau-theorem, more precisely:	
	\begin{lem}\label{lem:wlog}
		For a probability measure $\rho$ and a sequence of measures $\rho_L$ on $T_q\M_\gamma$, the following are equivalent
		\begin{enumerate}\label{lem:portmanteau}
			\item $(\rho_L)$ converges weakly to $\rho$.
			\item For any continuity set of $\rho$ we have
			$\lim_{L\to\infty}\rho_L[S]=\rho[S]$.
			\item For any relatively-compact open set $U\subset T_q\M_\gamma$ we have
			$$\rho[U]\leq\liminf_{L\to\infty}\rho_L[U].$$
		\end{enumerate}		
	\end{lem}
	\begin{proof}
		Without the condition of relative-compactness in (iii), these equivalences are provided by the Portmanteau-theorem. It remains to be shown that statement (iii) implies $\rho[V]\leq\liminf_{L\to\infty}\rho_L[V]$ for any open set $V\subset T_q\M_\gamma$. Indeed, since $T_q\M_\gamma$ is $\sigma$-compact, $V$ can be written as a countable ascending union of relatively-compact open sets $V_l,\ l\in\N$. Since $\rho$ is a probability measure, continuity of measure implies that for any $\varepsilon>0$ there is $l_0$ such that $\rho[V_{l_0}]\geq \rho[V]-\varepsilon$. If (iii) holds, we get the bound
		$$\rho[V]-\varepsilon\leq \rho[V_{l_0}]\leq \liminf_{L\to\infty} \rho_L[V_{l_0}]\leq \liminf_{L\to\infty} \rho_L[V]$$
		for any $\varepsilon>0$.
	\end{proof}
	We fix a relatively-compact open set $U\subset T_q\M_\gamma$.
	For convenience we reproduce the diagram of maps from the proof outline
	\begin{center}
		\begin{tikzcd}
			\edgeset_{1:L}(\gamma) \arrow[d,"\hat{\pi}"] \arrow[rd,"\mathrm{Q}^L"] & T_q\M_\gamma\\
			\bar{\Sigma}^{\edgeset(\gamma)} \arrow[r,two heads,swap,"\psi"] & \bar{\M}_\gamma\cap\M_\stateset \arrow[l,bend left,"\varphi"] \arrow[u,swap,"\mathrm{u}_{q,L}"]
		\end{tikzcd}
	\end{center}
	and introduce the shorthands $\psi_L:=\mathrm{u}_{q,L}\circ\psi$ and $\varphi_L:=\varphi\circ \mathrm{u}_{q,L}^{-1}$. We carry out the four steps of the proof as previewed in the outline.
	
	\paragraph{\textbf{Step 1:} } Since every $x\in\hat{\pi}(\edgeset_{1:L}(\gamma))$ comes from a path in $\gamma$ (as detailed in the proof of Thm. \ref{thm:largeconsistentdev}) it is clear that \begin{equation}\label{eq:minimalnbhd}
		x\in\bar{\Sigma}^{\edgeset(\gamma)}\quad \text{and}\quad |\eta_j\cdot x|\leq \frac1L, \ \text{for } j\in\stateset,
	\end{equation}
	where $\eta_j=\sum_{\rvstate(e)=j} \bra{e} - \sum_{\gamma(e)=j} \bra{e}.$
	In fact we observed earlier (Lem. \ref{lem:emptypemap}) that the semi-affine subspace
	$$\hat{\M}_\gamma=\Set{x\in\Sigma^{\edgeset(\gamma)}}{\forall j\in\stateset:\ \eta_j\cdot x= 0}$$ is in bijection with $\M_\gamma$ via the mutually inverse diffeomorphisms $\varphi$ and $\psi|_{\hat{\M}_\gamma}$.
	We can identify the tangent spaces to $\hat{\M}_\gamma$ with the linear subspace
	$$\hat{\M}_{\gamma,0}:=\Set{x\in\R^{\edgeset(\gamma)}}{\sum_{e\in\edgeset(\gamma)} x^e=0;\ \forall j\in\stateset:\ \eta_j\cdot x= 0}.$$
	\begin{lem}\label{lem:basis}
		There is a basis $B$ of the linear subspace $\hat{\M}_{\gamma,0}$ such that
		$$\left(\hat{\pi}(\edgeset_{1:L}(\gamma))+\sum_{b\in B} \mathbb{Z}\cdot\frac{1}{L}\ b \right)\cap \R^{\edgeset(\gamma)}_{+} \subset \hat{\pi}(\edgeset_{1:L}(\gamma))\cap \R^{\edgeset(\gamma)}_{+}$$
		and $\hat{\pi}(\edgeset_{1:L}(\gamma))\cap \R^{\edgeset(\gamma)}_{+}$ is not empty for $L$ large enough. Moreover, for any $x,x'\in \hat{\pi}(\edgeset_{1:L}(\gamma))\cap \R^{\edgeset(\gamma)}_{+}$ with $x'-x \in \sum_{b\in B} \mathbb{Z}\cdot\frac{1}{L}\ b$, one of the following two cases applies:
		\begin{enumerate}
			\item Every $e_{1:L}\in\edgeset_{1:L}(\gamma)$ with $\hat{\pi}(e_{1:L})=x$ is a cycle and then the same holds for every $e'_{1:L}\in\edgeset_{1:L}(\gamma)$ with $\hat{\pi}(e'_{1:L})=x'$.
			\item There is a unique internal state $\mathrm{init}(x)=\mathrm{init}(x')$ such that every $e_{1:L}\in\edgeset_{1:L}(\gamma)$ with $\hat{\pi}(e_{1:L})=x$ satisfies $\rvstate(e_1)=\mathrm{init}(x)$ and every $e'_{1:L}\in\edgeset_{1:L}(\gamma)$ with $\hat{\pi}(e'_{1:L})=x'$ satisfies $\rvstate(e'_1)=\mathrm{init}(x')$.
		\end{enumerate}
	\end{lem}
	\begin{proof}
	It is clear that $\hat{\pi}(\edgeset_{1:L}(\gamma))\cap \R^{\edgeset(\gamma)}_{+}$ is not empty for $L$ large enough because it contains the image under $\hat{\pi}$ of some string $e_{1:L}$ having a prefix which uses every edge of $\gamma$ at least once. We denote by $\mathrm{Circ}(\gamma)$ the set of elementary circuits (i.e. directed cycles with no vertex appearing twice as initial vertex of edges) and consider the linear map
	$$\alpha:\ \R^{\mathrm{Circ}(\gamma)}\to \R^{\edgeset(\gamma)}\ ,\quad \ket{C}\mapsto \sum_{e\in\edgeset(C)} \ket{e}.$$
	Moreover we set
	$$D:=\Set{|\edgeset(C)|\ket{C'}-|\edgeset(C')|\ket{C}\in \mathbb{Z}^{\mathrm{Circ}(\gamma)}}{C,C'\in\mathrm{Circ}(\gamma)}.$$
	It is easily seen that for every $d\in D$ and every $x\in \hat{\pi}(\edgeset_{1:L}(\gamma))$ such that $x+\frac1L\alpha(d)\in\R^{\edgeset(\gamma)}_{+}$, we also have $x+\frac1L\alpha(d)\in \hat{\pi}(\edgeset_{1:L}(\gamma))$. Moreover, we saw earlier (in the proof of Thm. \ref{thm:largeconsistentdev}) that, for $x\in\hat{\pi}(\edgeset_{1:L}(\gamma))$, either every $e_{1:L}\in\hat{\pi}^{-1}(x)$ is a cycle or there is a unique initial internal state $\mathrm{init}(x)$ for each such $e_{1:L}$. These properties are clearly invariant under addition of $\frac1L\alpha(d)$ (as long as $x+\frac1L\alpha(d)\in\R^{\edgeset(\gamma)}_{+}$) and in the second case $\mathrm{init}(x+\frac1L\alpha(d))=\mathrm{init}(x)$. Thus all that remains to be done is to find a subset $D_0\subset D$ such that $B:=\alpha(D_0)$ is a linearly independent set of cardinality $\dim \M_\gamma$. To this end, we recall a well-known theorem from graph-combinatorics saying that every strongly-connected graph has an ear decomposition, meaning that $\edgeset(\gamma)$ can be subdivided into the edge sets of a sequence of ``ears'' $\gamma_1,\ldots \gamma_R$. By definition $\gamma_1$ is an elementary circuit of $\gamma$ and, for every $1<r\leq R$, $\gamma_r$ is either a path in $\gamma$ with distinct vertices or an elementary circuit, such that the intersection of the vertex set of  $\gamma_r$ with the vertex set of $\bigcup_{1\leq r'<r} \gamma_{r'}$ contains exactly the first and last vertex of $\gamma_r$ (which coincide in the elementary circuit case). Notice that by associating to every ear $\gamma_r$, $1<r\leq R$ its first edge we obtain a bijection between $\{\gamma_2,\ldots ,\gamma_R\}$ and a subset of $\edgeset(\gamma)$ containing $\mathrm{outdeg}_\gamma(j)-1$ edges for every $j\in\stateset$. This implies that $R=|\edgeset(\gamma)|-|\stateset(\gamma)|+1$. Clearly, we can complete every ear $\gamma_r$ into an elementary circuit $C_r$ by possibly adjoining a simple path from $\bigcup_{1\leq r'<r} \gamma_{r'}$ with distinct vertices (If $\gamma_r$ is already an elementary circuit, we simply set $C_r:=\gamma_r$). Setting now
	$$d_{r}:=|\edgeset(C_{r})|\ket{C_{r+1}}-|\edgeset(C_{r+1})|\ket{C_{r}} \in D\ ,\quad \textnormal{for } 1\leq r< R,$$
	we see that $B:=\{\alpha(d_r)\}_{1\leq r<R}$ is a linearly independent set of cardinality  $|\edgeset(\gamma)|-|\stateset(\gamma)|=\dim \M_\gamma$.
	\end{proof}
	
	We can extend the basis $B$ of $\hat{\M}_{\gamma,0}$ to a basis of $\R^{\edgeset(\gamma)}$ through the linearly independent vectors
	$$n_j(\varphi(q)):= \sum_{\a \in\alphset(\gamma,j)} \varphi(q)^{j,\a} \ket{j,\a}\ ,\quad j\in\stateset.$$
	Due to (\ref{eq:minimalnbhd}) it is clear that there exists $C_0>0$ such that
	\begin{equation}\label{eq:tubularnbhd}
		\hat{\pi}(\edgeset_{1:L}(\gamma))\subset \varphi(q)+\hat{\M}_{\gamma,0} + \frac{C_0}{L} \sum_{j\in\stateset} [-1,1]\cdot  n_j(\varphi(q)).
	\end{equation}
	and we set $\mathcal{N}_L(x):=x+\frac{C_0}L\sum_{j\in\stateset} [-1,1]\cdot  n_j(\varphi(q))$. We define, for every $t\in\mathbb{Z}^{B}$, the half-open parallelotope
	$$P^0_t:=\varphi(q)+\frac1L \sum_{b\in B} t_b b +\frac1L \sum_{b\in B} [0, 1) \cdot b.$$
	and (\ref*{eq:tubularnbhd}) translates to
	$$\hat{\pi}(\edgeset_{1:L}(\gamma))\subset \bigsqcup_{t\in\mathbb{Z}^B} \mathcal{N}_L(P^0_t).$$
	We now define the measures
	$$\hat{\nu}_{L,0}:=\sum_{x\in\hat{\pi}(\edgeset_{1:L}(\gamma))} w(x) \delta_x\quad , \quad  \hat{\nu}_L:=\frac{\hat{\nu}_{L,0}}{\hat{\nu}_{L,0}[\hat{\pi}(\edgeset_{1:L}(\gamma))]}$$
	with
	$$w(x):=\begin{cases}
	1 &, \textnormal{if the elements of $\hat{\pi}^{-1}(x)$ are cycles,}\\
	\pi_{\mathrm{init}(x)}(q) &, \textnormal{otherwise}.
	\end{cases}$$
	Lem. \ref{lem:basis} leads to
	\begin{equation}\label{eq:equalmeasurehat}
		\hat{\nu}_L(\mathcal{N}_L(P^0_t))=\hat{\nu}_L(\mathcal{N}_L(P^0_{t'})),
	\end{equation}
	for any $t,t'\in\mathbb{Z}^B$ such that $\mathcal{N}_L(P^0_t), \mathcal{N}_L(P^0_{t'})\subset\R_+^{\edgeset(\gamma)}$.
	
	\paragraph{\textbf{Step 2:} }
	We are going to define a collection of disjoint partially-open polytopes $\hat{\mathfrak{K}}_L(\hat{\M}_\gamma)$ in $\varphi(q)+\mathcal{N}_L(\hat{\M}_{\gamma,0})$ containing all points of $\hat{\pi}(\edgeset_{1:L}(\gamma))$ that are not in asymptotically vanishing fringe regions of $\mathcal{N}_L(\hat{\M}_\gamma)$. We would then like to map this collection via $\psi_L$ to a collection $\mathfrak{K}_L(\M_\gamma)$ of disjoint polytopes in $T_q\M_\gamma$ having the properties listed in the outline. First of all, to make sure that the elements of $\mathfrak{K}_L(\M_\gamma)$ will still be disjoint, we need to take care that each fibre of the map $\psi_L$ intersects at most one element of the initial collection $\hat{\mathfrak{K}}_L(\hat{\M}_\gamma)$. In fact the fibre $\psi^{-1}(\psi(x))$ is the positive cone spanned by the linearly independent vectors
	$$n_j(x):=\sum_{\a \in\alphset(\gamma,j)} x^{j,\a} \ket{j,\a} ,\ j\in\stateset.$$
	We set
	$$\mathcal{F}_L(x):=\psi^{-1}(\psi(x))\cap \mathcal{N}_L(q+\hat{\M}_{\gamma,0}).$$
	We now denote $L_1:=\lfloor\frac12 L^{\frac14}\rfloor$
	and define the collection $\mathfrak{P}_L$ of half-open parallelotopes of the form
	\begin{equation}\label{eq:parallelotope}
		\varphi(q)+\frac{2 L_1}L \sum_{b\in B} t_b b +\frac1L \sum_{b\in B} \big[-L_1, L_1 \big) \cdot b\ ,\quad t\in\mathbb{Z}^{B}.
	\end{equation}
	These parallelotopes form a decomposition of the affine subspace $\varphi(q)+\hat{\M}_{\gamma,0}$ and each of them is composed of $(2 L_1)^{\dim(\M_\gamma)}$ smaller parallelotopes of the form $P^0_{t'}$.
	For any $\hat{S}\subset\hat{\M}_\gamma$, we define the collections of polytopes
	$$\hat{\mathfrak{K}}_L(\hat{S}):=\set{\mathcal{F}_L(P)\subset \R^{\edgeset(\gamma)}}{P\in \mathfrak{P}_L,\ P\subset \hat{S}}$$
	and the pushed-forward collection
	$$\mathfrak{K}_L(S):=\set{\psi_{L}(\hat{K})\subset T_q\M_\gamma}{\hat{K}\in \hat{\mathfrak{K}}_L(\varphi_L(S))},$$
	for any $S\subset \mathrm{u}_{q,L}(\M_\gamma)\subset T_q\M_\gamma$. By construction, we have $\cup\mathfrak{K}_L(S)\subset S$ and the elements of $\mathfrak{K}_L(S)$ are disjoint.
	\begin{lem}\label{lem:uniformboundary}
		Let $G\subset\M_\gamma$ be a compact subset containing $q$. There is a constant $C > 0$ and $L_0\in\N$ such that for any $K,K'\in \mathfrak{K}_L(\mathrm{u}_{q,L}(G))$ and $L\geq L_0$:
		\begin{itemize}
			\item[(i)] $\mathrm{diam}(K)\leq C L^{-\frac14}$ and in particular we have, for any set $S\subset T_q\M_\gamma$ with $\mathrm{u}_{q,L}(G)\supset S$:
			$$\cup \mathfrak{K}_L(S) \supset  \mathcal{B}^-_{CL^{-1/4}}(S):=\set{v\in S}{d_2(v,T_q\M_\gamma\setminus S)> C L^{-1/4}},$$
			where $d_2$ is the Euclidean metric induced from the ambient $\R^{\edgeset(\gamma)}$.
			\item[(ii)] $\left|\frac{\nu_L(K)}{\nu_L(K')}-1\right|\leq C L^{-\frac14}$.
		\end{itemize}
	\end{lem}
	\begin{proof}
		(i): Let $\hat{K}\in\hat{\mathfrak{K}}_L(\hat{\M}_\gamma)$ such that $K=\psi_{L}(\hat{K})$ and observe that the set $\mathcal{F}_L(\varphi(G))\subset \R^{\edgeset(\gamma)}_+$ is compact and contains $\hat{K}$. We therefore have
		$$\mathrm{diam}(K)\leq \left(\sup_{x\in\mathcal{F}_L(\varphi(G))} \|T_{x} \psi_{L} \|_2\right) \mathrm{diam}(\hat{K}) = L^\frac12 \left(\sup_{x\in\mathcal{F}_L(\varphi(G))} \| T_{x} \psi \|_2\right) \mathrm{diam}(\hat{K})$$
		and the supremum is finite because $\psi$ is continuously differentiable on $\mathcal{F}_L(\varphi(G))$. Moreover $\mathrm{diam}(\hat{K})$ is bounded by $const.\cdot 2L_1/L$ uniformly on $\mathcal{F}_L(\varphi(G))$ and hence there exists $C_1$ (depending only on $G$) such that $\mathrm{diam}(K) \leq C_1 L^{-\frac14}$ showing assertion (i).
		
		For (ii), the idea is to find a lower bound for $\hat{\nu}_L(\hat{K})$ by selecting a subset of small base parallelepipeds $P^0_t$ of $\hat{K}$ such that the union of the corresponding sets $\mathcal{N}_L(P^0_t)$ is contained in $\hat{K}$. An upper bound is derived in a similar way. Denote by $\{b^*\}_{b\in B}\cup \{n_j(\varphi(q))^*\}_{j\in\stateset}$ the dual basis to $B\cup \{n_j(\varphi(q))\}_{j\in\stateset}$. Let $x\in\varphi_L(K)$ and $x'\in\mathcal{F}_L(x)$. It is easy to see that since $\mathcal{F}_L(x)$ is in the compact set $\mathcal{F}_L(\varphi(G))$, there exists a constant $C_0'$ such that $x'=x+\sum_j r_j n_j(x)$ with $|r_j|<C_0'/L$. It then follows that
		$$|b^*\cdot x'|\leq \frac{C_0'}{L}\sum_{j\in\stateset}\max_{x_0\in \varphi(G)}\left|b^*\cdot n_j(x_0)\right|.$$
		and we set
		$$m_0:=\left\lceil C_0'\sum_{j\in\stateset}\max_{\stack{b\in B}{x_0\in \varphi(G)}}\left|b^*\cdot n_j(x_0)\right| \right\rceil.$$
		Assuming that $\hat{K}=\mathcal{F}_L(P)$, for $P\in\mathfrak{P}_L$, we conclude that, for any $L>(2m_0)^4$, i.e. $L_1>m_0$, the sets
		$$\mathcal{N}_L\left( \varphi(q)+\frac{2 L_1}L \sum_{b\in B} t_b b +\frac1L \sum_{b\in B} \big[-L_1\pm m_0, L_1\mp m_0 \big) \cdot b \right)$$
		are a subset resp. superset of $\hat{K}$. Moreover, by choosing $L_0>(2m_0)^4$ big enough, we can assume that these sets are contained in $\R_+^{\edgeset(\gamma)}$ for $L\geq L_0$. This makes (\ref{eq:equalmeasurehat}) applicable to their constituting parallelotopes $\mathcal{N}_L(P^0_t)$ yielding the bounds
		$$(2(L_1- m_0))^{\dim(\M_\gamma)} \hat{\nu}_L(\mathcal{N}(P^0_0)) \leq \nu_L(K)\leq (2(L_1+ m_0))^{\dim(\M_\gamma)} \hat{\nu}_L(\mathcal{N}(P^0_0)).$$
		We can apply the same reasoning to $K'$ and conclude
		\begin{align*}
			|\frac{\nu_L(K)}{\nu_L(K')}-1| &\leq \left|\frac{\nu_L(K)-\nu_L(K')}{\nu_L(K')}\right|\leq \frac{(L_1+m_0)^{\dim(\M_\gamma)}-(L_1-m_0)^{\dim(\M_\gamma)}}{(L_1-m_0)^{\dim(\M_\gamma)}}\\
			&\leq const. \frac{1}{L_1}\leq C_2 L^{-\frac14}
		\end{align*}
		for some constant $C_2>0$ (only depending on $G$). We set $C:=\max\{C_1,C_2\}$.
	\end{proof}
	
	\paragraph{\textbf{Step 3:} }
	
	We are now going to compare $\nu_L$ on the elements of $\mathfrak{K}_L(U)$ with yet another measure $\lambda_L$ defined as follows: We consider the canonical Riemannian volume form $\omega_\mathrm{g}$ on $\M_\gamma$ wrt. some chosen orientation. We can push it forward along $\varphi$ to a volume form on $\hat{\M}_\gamma$. Its value at $\varphi(q)$ is a top-dimensional form $\varphi_*(\omega_\mathrm{g})(\varphi(q))\in\bigwedge^{\dim(\M_\gamma)} T^*_{\varphi(q)}\hat{\M}_\gamma$. This form has the property that, for any positively-oriented $\varphi_*(\mathrm{g}(q))$-orthonormal basis $(\hat{v}_m)_{1\leq m\leq \dim\M_\gamma}$ of $T_{\varphi(q)}\hat{\M}_\gamma$, it satisfies
	$$\varphi_*(\omega_\mathrm{g})(\varphi(q))\cdot (\hat{v}_1\otimes\cdots\otimes \hat{v}_{\dim(\M_\gamma)})=1.$$
	It can be extended to a constant volume form $\hat{\omega}_\mathrm{g}$ (wrt. translation) on the affine manifold $q+\hat{\M}_{\gamma,0}$, i.e. we have
	$$\hat{\omega}_\mathrm{g}(\varphi(q))= \varphi_*(\omega_\mathrm{g})(\varphi(q)).$$
	We now define the measures
	$\hat{\lambda}_0[S]:=\int_S \hat{\omega}_\mathrm{g},$
	for any Borel-set $S\subset q+\hat{\M}_{\gamma,0}$, and $\hat{\lambda}:=\frac{1}{\hat{\lambda}_0[\hat{\M}_\gamma]}\ \hat{\lambda}_0|_{\hat{\M}_\gamma}$
	Note that these measures are translation-invariant by construction. We pull $\hat{\lambda}$ back along $\varphi_L$ to a measure $\lambda_L$ on $\mathrm{u}_{q,L}(\M_\gamma)\subset T_q\M_\gamma$ and extend it by zero to all of $T_q\M_\gamma$.
	\begin{lem}
		For every bounded Borel-set $S\subset T_q\M_\gamma$ we have
		$$\lim_{L\to\infty} L^{\frac12 \dim(\M_\gamma)} \lambda_L[S] = \frac{1}{\hat{\lambda}_0[\hat{\M}_\gamma]}\int_{S} \omega_{\mathrm{g},q}$$
		with the constant form $\omega_{\mathrm{g},q}(v)=\omega_\mathrm{g}(q)\circ(T_v\mathrm{u}_{q,1})^{\otimes \dim(\M_\gamma)}$.
	\end{lem}\label{lem:volumeformlimit}
	\begin{proof}
		Observe that, for $L$ large enough, $S\subset \mathrm{u}_{q,L}(\M_\gamma)$ and
		$\lambda_L[S]=\frac{1}{\hat{\lambda}_0[\hat{\M}_\gamma]}\int_{S} \varphi_L^*(\hat{\omega}_\mathrm{g}).$
		We compute
		\begin{align*}
			\lim_{L\to\infty} L^{\frac12 \dim(\M_\gamma)} \lambda_L[S] &= 
			\lim_{L\to\infty} \frac{L^{\frac12 \dim(\M_\gamma)}}{\hat{\lambda}_0[\hat{\M}_\gamma]} \int_{v\in S} \varphi^*(\hat{\omega}_\mathrm{g})(\mathrm{u}_{q,L}^{-1}(v))\circ(T_v\mathrm{u}_{q,L})^{\otimes \dim(\M_\gamma)}\\
			&= 
			\lim_{L\to\infty} \frac{1}{\hat{\lambda}_0[\hat{\M}_\gamma]} \int_{v\in S} \varphi^*(\hat{\omega}_\mathrm{g})(\mathrm{u}_{q,L}^{-1}(v))\circ(T_v\mathrm{u}_{q,1})^{\otimes \dim(\M_\gamma)}.
		\end{align*}
		The form in the integrand is smooth on the compact $\bar{S}$ and hence we can exchange the order of limit and integration, yielding:
		\begin{align*}
		\lim_{L\to\infty} L^{\frac12 \dim(\M_\gamma)} \lambda_L[S] &= 
		\frac{1}{\hat{\lambda}_0[\hat{\M}_\gamma]} \int_{v\in S} \varphi^*(\hat{\omega}_\mathrm{g})(q)\circ(T_v\mathrm{u}_{q,1})^{\otimes \dim(\M_\gamma)}\\
		&= 
		\frac{1}{\hat{\lambda}_0[\hat{\M}_\gamma]} \int_{v\in S} \omega_\mathrm{g}(q)\circ(T_v\mathrm{u}_{q,1})^{\otimes \dim(\M_\gamma)}
		\end{align*}
	\end{proof}
	
	The measures $\lambda_L$ and $\mu_L$ can be compared on the elements of $\mathfrak{K}_L(\M_\gamma)$ according to the following
	\begin{lem}\label{lem:measurecomparison}
		Let $G\subset T_q\M_\gamma$ be a compact subset containing $q$. For every $\varepsilon>0$ there exists $L_0\in\N$ such that for every $L\geq L_0$ and $K\in\mathfrak{K}_L(\mathrm{u}_{q,L}(G))$
		$$\left|\frac{\nu_L(K)}{\lambda_L(K)}-1 \right|<\varepsilon.$$
	\end{lem}
	\begin{proof}
		We begin by showing
		\begin{equation}\label{eq:fringemeasure}
			\lim_{L\to\infty} \nu_L[\cup\mathfrak{K}_L(\mathrm{u}_{q,L}(\M_\gamma))]=1=\lim_{L\to\infty} \lambda_L[\cup\mathfrak{K}_L(\mathrm{u}_{q,L}(\M_\gamma))].
		\end{equation}
		For this we observe that
		$$\cup\hat{\mathfrak{K}}_L(\hat{\M}_\gamma)= \mathcal{F}_L(\bar{\hat{\M}}_\gamma) \setminus \bigcup_{\stack{t\in\mathbb{Z}^B:}{P^0_t\cap\partial \hat{\M}_\gamma\neq \varnothing}} \mathcal{F}_L(P^0_t\cap\bar{\hat{\M}}_\gamma),$$
		where $\partial \hat{\M}_\gamma:=\bar{\hat{\M}}_\gamma\setminus \hat{\M}_\gamma$. As in the proof of Lem. \ref*{lem:uniformboundary}, it can be shown that there is a constant $C'>0$ such that for the above boundary parallelotopes $P^0_t$: $$\hat{\nu}_L[\mathcal{F}_L(P^0_t\cap\bar{\hat{\M}}_\gamma)]\leq (1+C')\hat{\nu}_L[\hat{K}_0],\ \text{for any } L\in\N,$$
		where $\hat{K}_0=\mathcal{F}_L(\varphi(q)+\frac1L \sum_{b\in B} [-L_1, L_1 ))$. On the other hand we have according to Lem. \ref{lem:uniformboundary}.(ii):
		$$1\geq \hat{\nu}_L[\cup\hat{K}_L(\varphi(G))]\geq (1-CL^{-1/4}) \hat{\nu}_L[\hat{K}_0] \mathrm{card}\set{t\in\mathbb{Z}^B}{P^0_t\subset\varphi(G)}.$$
		In total we get
		$$\hat{\nu}_L[\cup\hat{\mathfrak{K}}_L(\hat{\M}_\gamma)]\geq 1-  \frac{1+C'}{1-CL^{-1/4}}\frac{\mathrm{card}\set{t\in\mathbb{Z}^B}{P^0_t \cap\partial\hat{\M}_\gamma\neq\varnothing}}{\mathrm{card}\set{t\in\mathbb{Z}^B}{P^0_t\subset\varphi(G)}}.$$
		The second fraction in this expression converges to $0$ due to the well-known fact from geometric measure theory that the numerator grows like $\mathcal{O}(L^{\dim(\M_\gamma)-1})$ (because $\partial\hat{\M}_\gamma$ has box-counting-dimension $\dim(\M_\gamma)-1$, being contained in a union of smooth codimension 1 submanifolds) and the denominator grows like $\mathcal{O}(L^{\dim(\M_\gamma)})$. This proves the left-hand limit of (\ref*{eq:fringemeasure}). The right-hand limit follows in an analogous fashion exploiting the translation invariance of $\hat{\lambda}$ and deducing
		$$\hat{\lambda}[\cup\hat{\mathfrak{K}}_L(\hat{\M}_\gamma)]\geq 1- \frac{\mathrm{card}\set{t\in\mathbb{Z}^B}{P^0_t \cap\partial\hat{\M}_\gamma\neq\varnothing}}{\mathrm{card}\set{t\in\mathbb{Z}^B}{P^0_t\subset\hat{\M}_\gamma}}.$$
		We can now finish the proof of the lemma by
		\begin{align*}
			\frac{\nu_L[K]}{\lambda_L[K]} &=\frac{\frac{1}{\lambda_L[K]} \lambda_L[\cup\mathfrak{K}_L(\mathrm{u}_L(\M_\gamma))]}{\frac{1}{\nu_L[K]} \nu_L[\cup\mathfrak{K}_L(\mathrm{u}_L(\M_\gamma))]} \cdot\frac{ \nu_L[\cup\mathfrak{K}_L(\mathrm{u}_L(\M_\gamma))]}{ \lambda_L[\cup\mathfrak{K}_L(\mathrm{u}_L(\M_\gamma))]}\\
			&=\frac{\sum_{K'\in \mathfrak{K}_L(\mathrm{u}_L(\M_\gamma))}\frac{\lambda_L[K']}{\lambda_L[K]}}{\sum_{K'\in \mathfrak{K}_L(\mathrm{u}_L(\M_\gamma))}\frac{\nu_L[K']}{\nu_L[K]}} \cdot\frac{ \nu_L[\cup\mathfrak{K}_L(\mathrm{u}_L(\M_\gamma))]}{ \lambda_L[\cup\mathfrak{K}_L(\mathrm{u}_L(\M_\gamma))]}
		\end{align*}
		which implies
		$$\frac{1}{1+CL^{-\frac14}} \cdot\frac{ \nu_L[\cup\mathfrak{K}_L(\mathrm{u}_L(\M_\gamma))]}{ \lambda_L[\cup\mathfrak{K}_L(\mathrm{u}_L(\M_\gamma))]}\leq\frac{\nu_L[K]}{\lambda_L[K]}\leq \frac{1}{1-CL^{-\frac14}} \cdot\frac{ \nu_L[\cup\mathfrak{K}_L(\mathrm{u}_L(\M_\gamma))]}{ \lambda_L[\cup\mathfrak{K}_L(\mathrm{u}_L(\M_\gamma))]}.$$
		Both bounds converge to $1$ independently of $K$ according to (\ref*{eq:fringemeasure}).
	\end{proof}
	
	\paragraph{\textbf{Step 4:} }
	
	It is now time to relate the measures $\mu$ and $\mu_L$ using the auxiliary measures $\lambda_L$ and $\nu_L$. Let $G\subset\M_\gamma$ be a compact neighbourhood of $q$. Then
	\begin{equation}\label{eq:nbhdtoq}
		\lim_{L\to\infty}\sup_{q'\in \mathrm{u}_{q,L}^{-1}(U)} d_2(q',q)=0
	\end{equation}
	implies that $\mathrm{u}_{q,L}^{-1}(U)\subset G$, for $L$ large enough. Furthermore this implies,
	due to Lem. \ref{lem:preciselargedeviation}, that we can write
	$$\mu_L[U] = \frac{\hat{\nu}_{L,0}(\R^{\edgeset(\gamma)})}{L} \int_{\mathcal{F}_L(\varphi_L(U))} \hat{\kappa}_0(x)\hat{\kappa}_{1,L}(x)\e^{\sum_{j\in\stateset}r_{L}^j(x)\sum_{\a \in\alphset(\gamma,j)} \psi(x)^{j,\a}\log\frac{\psi(x)^{j,\a}}{q^{j,\a}}} \e^{-L\mathrm{h}_\gamma(\psi(x)\parallel q)}\hat{\nu}_L(dx).$$
	Note that some of the functions in the integrand are only defined on $\hat{\pi}(\edgeset_{1:L}(\gamma))$ but the integral is still well-defined since $\hat{\nu}_L$ vanishes outside of this discrete set.
	Lem. \ref{lem:preciselargedeviation} moreover says that the sequence of functions $r_L^j(x)=L(x^j-\pi_j(\psi(x)))$ is uniformly bounded on the domain of integration. Together with (\ref*{eq:nbhdtoq}) we may conclude that the first exponential also converges uniformly to $1$.  Furthermore $\hat{\kappa}_0$ is continuous and hence, again by (\ref*{eq:nbhdtoq}) and the fact that $|x-\varphi(\psi(x))|\leq const.\cdot 1/L$, we have the uniform convergence  $$\lim_{L\to\infty}\sup_{x\in\mathcal{F}_L(\varphi_L(U_L))\cap\hat{\pi}(\edgeset_{1:L}(\gamma))}|\hat{\kappa}_0(x)-\hat{\kappa}_0(\varphi(q))|=0.$$
	Moreover, $\hat{\kappa}_{1,L}(x)$ converges uniformly to $1$. Together, these uniform convergence results imply that we have: For every $\varepsilon_1>0$ there exists $L_1\in\N$ such that for $L\geq L_1$:
	\begin{align*}
		\mu_L[U]  &\geq (1- \varepsilon_1)\cdot \frac{\hat{\nu}_{L,0}(\R^{\edgeset(\gamma)})\kappa_0(\varphi(q))}{L} \int_{U} \e^{-L\mathrm{h}_\gamma(\mathrm{u}_{q,L}^{-1}(v) \parallel q)} \nu_L(d v)\\
		&\geq (1- \varepsilon_1)\cdot \frac{\hat{\nu}_{L,0}(\R^{\edgeset(\gamma)})\kappa_0(\varphi(q))}{L}\sum_{K\in \mathfrak{K}_L(U)} \int_{K} \e^{-L\mathrm{h}_\gamma(\mathrm{u}_{q,L}^{-1}(v) \parallel q)} \nu_L(d v)\\
		&\geq (1- \varepsilon_1)\cdot \frac{\hat{\nu}_{L,0}(\R^{\edgeset(\gamma)})\kappa_0(\varphi(q))}{L}\sum_{K\in \mathfrak{K}_L(U)} \inf_{v\in K}\e^{-L\mathrm{h}_\gamma(\mathrm{u}_{q,L}^{-1}(v) \parallel q)} \nu_L(K).
	\end{align*}
	 Application of Lem. \ref{lem:measurecomparison} yields that for every $\varepsilon_2>0$ there exists $L_2\in\N$ such that for $L\geq L_2$:
	$$\mu_L[U] \geq (1- \varepsilon_2)\cdot \frac{\hat{\nu}_{L,0}(\R^{\edgeset(\gamma)})\kappa_0(\varphi(q))}{L}\sum_{K\in \mathfrak{K}_L(U)} \inf_{v\in K}\e^{-L\mathrm{h}_\gamma(\mathrm{u}_{q,L}^{-1}(v) \parallel q)} \lambda_L(K).$$
	The exponential in this expression is continuously differentiable on the compact set $\bar{U}$. Together with Lem. \ref{lem:uniformboundary}.(i), this turns the above sum into a Riemann-sum and hence, for every $\varepsilon_3>0$, there exists $L_3\in\N$ such that for $L\geq L_3$:
	$$\mu_L[U] \geq (1- \varepsilon_3)\cdot \frac{\hat{\nu}_{L,0}(\R^{\edgeset(\gamma)})\kappa_0(\varphi(q))}{L}\int_{\mathfrak{K}_L(U)} \e^{-L\mathrm{h}_\gamma(\mathrm{u}_{q,L}^{-1}(v) \parallel q)} \lambda_L(dv).$$
	On the compact set $\bar{U}$, the exponential converges uniformly to $\e^{-\frac12\mathrm{g}\cdot v^{\otimes 2}}$, enabling us to lower-bound the above integral. Additionally we can apply Lem. \ref{lem:uniformboundary}.(i) again to restrict the domain of integration. Together we obtain that, for every $\varepsilon_4>0$, there exists $L_4\in\N$ such that for $L\geq L_4$:
	\begin{align*}
		\mu_L[U] &\geq (1- \varepsilon_4)\cdot \frac{\hat{\nu}_{L,0}(\R^{\edgeset(\gamma)})\kappa_0(\varphi(q))}{L}\int_{\mathcal{B}^-_{CL^{-1/4}}(U)} \e^{-\frac12\mathrm{g}\cdot v^{\otimes 2}} \lambda_L(dv)\\
		&=(1- \varepsilon_4)\cdot \frac{\hat{\nu}_{L,0}(\R^{\edgeset(\gamma)})\kappa_0(\varphi(q))}{\hat{\lambda}_0(\hat{\M}_\gamma) L^{1+\frac{1}{2}\dim(\M_\gamma)}}\int_{U}\chi_{\mathcal{B}^-_{CL^{-1/4}}(U)} \e^{-\frac12\mathrm{g}\cdot v^{\otimes 2}}\  L^{\frac{1}{2}\dim(\M_\gamma)}\lambda_L(dv).
	\end{align*}
	Due to Lem. \ref{lem:volumeformlimit} and dominated convergence we can deduce that for every $\varepsilon_5>0$ there exists $L_5\in\N$ such that for $L\geq L_5$:
	$$\mu_L[U]\geq (1- \varepsilon_5)\cdot \underbrace{\frac{\hat{\nu}_{L,0}(\R^{\edgeset(\gamma)})\kappa_0(\varphi(q))}{\hat{\lambda}_0(\hat{\M}_\gamma) L^{1+\frac{1}{2}\dim(\M_\gamma)}}}_{=:C_L} \underbrace{\int_{v\in U} \e^{-\frac12\mathrm{g}\cdot v^{\otimes 2}} \omega_{\mathrm{g},q}(v)}_{=\mu[U]}.$$
	We can conclude that
	$$\liminf_{L\to\infty} \frac{1}{C_L}\mu_L[U]\geq \mu[U]$$
	and by Lem. \ref{lem:wlog} this means that the sequence of measures $\frac{1}{C_L}\mu_L$ converges weakly to $\mu$. In particular, since $T_q\M_\gamma$ is a continuity set of $\mu$, we have
	$$1=\mu[T_q\M_\gamma]=\lim_{L\to\infty} \frac{1}{C_L}\mu_L[T_q\M_\gamma]=\lim_{L\to\infty}\frac{1}{C_L}.$$
	This finally implies
	$$\liminf_{L\to\infty} \mu_L[U]= \liminf_{L\to\infty} \frac{1}{C_L}\mu_L[U]\geq \mu[U]$$
	and, again by by Lem. \ref{lem:wlog}, $\mu_L$ converges weakly to $\mu$.
	\qed

	\section{Proof of Thm. \ref*{thm:graddescent}}\label{subsec:proofgraddescent}
	
	Let $q\in\mathcal{C}$. We begin by proving a slightly generalized dominated convergence lemma:
	\begin{lem}\label{lem:domconv}
		Let $(f_L)$ be a sequence of functions on $T_q\M_\gamma$ having a pointwise limit. Moreover assume that there exists a non-negative function $F$ with $|f_L|\leq F$ for any $L\in\N$ and that $\Lambda[F]$ as well as the $\Lambda[f_L]$ are finite. Then we have
		$$\lim_{L\to\infty} \left\langle f_L \right\rangle_{(\mathrm{u}_{q,L})_*\Pr_{\mathrm{Q}^{L}}(q)}= \Lambda[\lim_{L\to\infty} f_L].$$
	\end{lem}
	\begin{proof}
		By dominated convergence, $\Lambda[\lim_{L\to\infty} f_L]$ is finite and also
		$$\lim_{L_0\to\infty}\Lambda[\inf_{L'\geq L_0}f_{L'}]=\Lambda[\lim_{L\to\infty} f_L]=\lim_{L_0\to\infty}\Lambda[\sup_{L'\geq L_0}f_{L'}].$$
		On the other hand, the functions $\inf_{L'\geq L_0}f_{L'}$ and $\sup_{L'\geq L_0}f_{L'}$ are $(\mathrm{u}_{q,L})_*\Pr_{\mathrm{Q}^L}(q)$-integrable for any $L\in\N$ and Thm. \ref{thm:centrallimitthm} implies
		\begin{align*}
		\Lambda[\inf_{L'\geq L_0}f_{L'}]&=\liminf_{L\to\infty} \langle  \inf_{L'\geq L_0}f_{L'} \rangle_{(\mathrm{u}_{q,L})_*\Pr_{\mathrm{Q}^L}(q)}\leq\liminf_{L\to\infty} \langle f_L \rangle_{(\mathrm{u}_{q,L})_*\Pr_{\mathrm{Q}^L}(q)}, \\
		\Lambda[\sup_{L'\geq L_0}f_{L'}]&=\limsup_{L\to\infty} \langle \sup_{L'\geq L_0}f_{L'} \rangle_{(\mathrm{u}_{q,L})_*\Pr_{\mathrm{Q}^L}(q)}\geq \limsup_{L\to\infty} \langle f_L \rangle_{(\mathrm{u}_{q,L})_*\Pr_{\mathrm{Q}^L}(q)}.
		\end{align*}
		Letting $L_0\to\infty$ we can conclude
		$$\limsup_{L\to\infty} \langle f_L \rangle_{(\mathrm{u}_{q,L})_*\Pr_{\mathrm{Q}^L}(q)}\leq \Lambda[\lim_{L\to\infty} f_L]\leq \liminf_{L\to\infty} \langle f_L \rangle_{(\mathrm{u}_{q,L})_*\Pr_{\mathrm{Q}^L}(q)}.$$
	\end{proof}
	We set $U_{v,L}:=\set{v'\in \mathrm{u}_{q,L}(\bar{\M}_\gamma\cap\M_\stateset) }{\Phi\circ\Pr_{\rvalphra}(q+L^{-\frac12}v')<\Phi\circ\Pr_{\rvalphra}(q+L^{-\frac12}v)}$
	and note that, as a consequence of the implicit function theorem, we have pointwise convergence $\chi_{U_{v,L}}\to \chi_{U_v}$ with $U_v:=\set{v'\in T_q\M_\gamma}{\d (\Phi\circ\Pr_{\rvalphra})(q)\cdot v'< \d (\Phi\circ\Pr_{\rvalphra})(q) \cdot v}.$
	We can write, for any $v\in\mathrm{u}_{q,L}(\bar{\M}^{(L)}_\gamma)$:
	$$R^{L}_{N,\Phi}[q+L^{-\frac12} v |q] = N\ (\mathrm{u}_{q,L})_*\Pr_{\mathrm{Q}^{L}}(q)[v]\ \Big( \left( (\mathrm{u}_{q,L})_*\Pr_{\mathrm{Q}^{L}}(q)[U_{v,L}]\right)^{N-1} + \varepsilon_L(v)\Big),$$
	where the term $0\leq \varepsilon_L(v)\leq const.\cdot (\mathrm{u}_{q,L})_*\Pr_{\mathrm{Q}^{L}}(q)[\partial U_{v,L}]$ accounts for the possibility that two offspring maximize $\Phi\circ\Pr_{\rvalphra}$ simultaneously. Denoting by $\bar{U}_{v,L}^0$ the enclosed volume between $\partial U_{v,L}$ and $\partial U_v$, we have $\chi_{\bar{U}_{v,L}^0}\to\chi_{\partial U_v}$ pointwise and a first application of Lem. \ref{lem:domconv} yields
	$$0\leq \lim_{L\to\infty} \varepsilon_L(v)\leq const.\cdot \lim_{L\to\infty}\langle \chi_{\bar{U}_{v,L}^0}\rangle_{(\mathrm{u}_{q,L})_*\Pr_{\mathrm{Q}^{L}}(q)}=\Lambda[\chi_{\partial U_v}]=0$$
	We now chose a $\mathrm{g}$-orthonormal basis $(v_m)_{1\leq m\leq \dim(\M_\gamma)}$ of $T_q\M_\gamma$ with $v_1=\frac{\nabla_\mathrm{g}\Phi(q)}{\| \nabla_\mathrm{g}\Phi(q) \|_\mathrm{g}}$ and expand
	\begin{align*}
		&\left\langle q' \right\rangle_{R^{L}_{N,\Phi}[q' |q]} = q + L^{-\frac12} \sum_{m=1}^{\dim(\M_\gamma)} v_m \sum_{v\in\mathrm{u}_{q,L}(\bar{\M}^{L}_\gamma)} (v_m^\flat\cdot v)\ R^{L}_{N,\Phi}[q+L^{-\frac12} v |q]\\
		&= q + L^{-\frac12} \sum_{m=1}^{\dim(\M_\gamma)} v_m \left\langle N (v_m^\flat\cdot v) \left(
		\left\langle \chi_{U_{v,L}} \right\rangle_{(\mathrm{u}_{q,L})_*\Pr_{\mathrm{Q}^L}(q)} \right)^{N-1} + \varepsilon_L(v) \right\rangle_{(\mathrm{u}_{q,L})_*\Pr_{\mathrm{Q}^L}(q)[v]}.
	\end{align*}
	Due to Lem. \ref{lem:domconv}, the term in the outer averaging brackets converges to $N (v_m^\flat\cdot v) 
	\Lambda[\chi_{U_v}]^{N-1}$. Another application of Lem. \ref{lem:domconv} readily yields 
	$$\left\langle q' \right\rangle_{R^{L}_{N,\Phi}[q' |q]} =q + L^{-\frac12} \sum_{m=1}^{\dim(\M_\gamma)} v_m \Lambda\left[v\mapsto N (v_m^\flat\cdot v) 
	\Lambda[\chi_{U_v}]^{N-1}  \right]$$
	It is straightforward to compute
	$$\Lambda[\chi_{U_v}]=\frac{1}{\sqrt{2\pi}}\int_{-\infty}^{v_1^\flat\cdot v} \e^{-\frac{1}{2}x^2} \d x = \frac12 (1+\mathrm{erf}(\frac{1}{\sqrt{2}} v_1^\flat\cdot v))$$
	and hence
	$$\Lambda\left[v\mapsto N (v_m^\flat\cdot v) 
	\Lambda[\chi_{U_v}]^{N-1}  \right] =
	\begin{cases}
	0 & \textnormal{if } m>1,\\
	\alpha(N) & \textnormal{if } m=1.
	\end{cases}$$
	The recursion (\ref{eq:expectationiteration}) now becomes:
	\begin{equation}\label{eq:iteration}
	q^L(t_{n+1})=q^L(t_n)+L^{-\frac12} X(q^L(t_n)) + o(L^{-\frac12}),
	\end{equation}
	with $X(q)=\alpha(N)\frac{\nabla_{\mathrm{g}}(\Phi\circ\Pr_{\rvalphra})(q)}{\|\nabla_{\mathrm{g}}(\Phi\circ\Pr_{\rvalphra})(q)\|_\mathrm{g}}$. It remains to be verified that the hereby defined sequence of paths $(q^L)$ indeed converges uniformly to an integral curve of $X$. This is fairly standard: First one derives from (\ref{eq:iteration}) that there exists $C_1>0$ such that for any $t\in (t_n,t_{n+1})$: $\|\dot{q}^L(t)\|_2\leq C_1$ and therefore $\|q^L(t)-q^L(t_n)\|_2\leq C_1 \tau_L$. This together with the Lipschitz-continuity of $X$ (which holds since $\Phi\circ\Pr_{\rvalphra}$ is continuously differentiable) leads to
	$$\|\dot{q}^L(t)-X(q^L(t))\|_2=\|\frac{L^{-\frac12}}{\tau_L}X(q^L(t_n))+o(1)-X(q^L(t))\|_2\leq const.\cdot \tau_L + o(1)=o(1).$$
	Thus there exists a sequence $(r_L)$ of non-negative real numbers converging to $0$ such that for any $T>0$ and $t\in[0,T]$ we have
	\begin{equation}\label{eq:uniquelimit}
	\|q^L(t)-q^L(0)-\int_{s=0}^{t} X(q^L(s)) \d s \|_2 \leq r_L T.
	\end{equation}
	Now let $(L_b)_{b\in\N}$ be any strictly increasing sequence of positive integers and set
	$$b(m):=\min\set{k\in\N}{\forall L\geq L_k:\ r_{L}\leq \frac{1}{2^m}}.$$
	We claim that the subsequence $(q^{L_{b(m)}})_{m\in\N}$ converges uniformly on $[0,T]$ to an integral curve of the vector field $X$. To show this, we use (\ref*{eq:uniquelimit}) and the triangle-inequality yielding
	$$\|q^{L_{b(m+1)}}(t)-q^{L_{b(m)}}(t)\|_2 \leq \frac{3T}{2}\frac{1}{2^m}+ C_2 \int_{s=0}^{t} \|q^{L_{b(m+1)}}(s)-q^{L_{b(m)}}(s)\| \d s,$$
	where $C_2$ is a Lipschitz-constant for $X$.
	Gronwall's lemma now gives us
	$$\sup_{t\in[0,T]} \|q^{L_{b(m+1)}}(t)-q^{L_{b(m)}}(t)\|_2 \leq \frac{3T}{2}\frac{1}{2^m} \e^{C_{2} T}.$$
	Hence, the sequence $(q^{L_{b(m)}})_{m\in\N}$ of continuous functions  converges uniformly on $[0,T]$ to the continuous limit
	$$q(t):=\lim_{m\to\infty}q^{L_{b(m)}}(t)=q^{L_{b(1)}}+\sum_{l=1}^{\infty} \left(q^{L_{b(l+1)}}(t)-q^{L_{b(l)}}(t)\right)$$
	because the sum therein converges normally. We consider (\ref*{eq:uniquelimit}) with $L=L_{b(m)}$ and, since $X$ is continuous, we can pass to the limit $m\to\infty$ yielding
	$$q(t)-q(0)-\int_{s=0}^t X(q(s))\d s = 0,$$
	which proves the claim by verifying that indeed $q(\cdot)$ is an integral curve of $X$. In fact we showed the following: For every subsequence of $(q^L)$ there exists a subsubsequence which converges uniformly on $[0,T]$ to the unique integral curve of $X$ with initial point $q(0)$. This implies that also the sequence $(q^L)$ itself must converge uniformly to the same limit.
	\qed
	
	\addtocontents{toc}{\protect\setcounter{tocdepth}{1}}

	\bibliography{mybib}{}
	\bibliographystyle{plain}

\end{document}